\documentclass[11pt,a4paper,oneside,reqno]{amsart}
\usepackage{shadethm,framed}
\usepackage[T1]{fontenc}
\usepackage[utf8]{inputenc}
\usepackage[english]{babel}
\usepackage[papersize={21cm,29.7cm},left=2.5cm,right=2.5cm,top=2.5cm,bottom=2.5cm]{geometry}
\usepackage{enumerate}
\usepackage{listings}
\usepackage{amsthm, amsfonts, amssymb, amsmath}
\usepackage{graphicx}
\usepackage{xcolor}
\usepackage{multirow} 
\usepackage{color}
\usepackage{framed}
\usepackage{wallpaper}
\usepackage{tikz,tikz-cd}
\usepackage{pgf,tikz,pgfplots}
\usepackage{tcolorbox}
\usepackage{scalefnt}
\usepackage{mdframed}
\usepackage{booktabs}
\usepackage{caption}
\usepackage{makecell}
\usepackage{float}
\usepackage{calrsfs}
\usepackage{transparent}
\usepackage{adjustbox}
\usepackage[
  backend=biber,
  style = alphabetic,
  maxalphanames=99,
    minalphanames=99,
    maxnames=99
]{biblatex}

\AtBeginBibliography{\small}

\usepackage[table]{xcolor}
\usepgfplotslibrary{patchplots}
\usetikzlibrary{patterns, positioning, arrows,decorations.markings}

\definecolor{blou}{rgb}{0.66, 0.44, 0.98}
\usepackage[hidelinks,urlcolor=blue]{hyperref}
\hypersetup{
    colorlinks,
    linkcolor={blue!50!black},
    citecolor=blou,
    urlcolor={blue!80!black}
}

\usepackage{pgfplots}
\pgfplotsset{width=7cm, compat=1.10}
\usepgfplotslibrary{fillbetween}
\pgfmathdeclarefunction{poly}{0}{\pgfmathparse{sin(x)+x}}

\tikzcdset{arrow style=tikz,diagrams={>={Straight Barb[scale=0.8]},all arrow/.append style = -{Latex[scale=0.8]}}}

\tcbset{
    frame code={}
    center title,
    left=0pt,
    right=0pt,
    top=0pt,
    bottom=0pt,
    colback=red!70,
    colframe=white,
    width=\dimexpr\textwidth\relax,
    enlarge left by=0mm,
    boxsep=5pt,
    arc=0pt,outer arc=0pt,
    }

\newtheoremstyle{pourdef}
  {10pt}
  {10pt}
  {}
  {}
  {\bf}
  {.~}
  { }
  {}
\newtheoremstyle{pourth}{10pt}{10pt}{\em}{}{\sc}{.~}{ }{}
\newtheoremstyle{pourpp}{10pt}{10pt}{\em}{}{\bf \em}{.~}{ }{}
\newtheoremstyle{pourrk}{10pt}{10pt}{}{}{\em}{.~}{ }{}
\newtheoremstyle{pourlm}{10pt}{10pt}{\em}{}{\bf \em}{.~}{ }{}
\newtheoremstyle{pourco}{10pt}{10pt}{\em}{}{\bf \em}{.~}{ }{}

\definecolor{black}{cmyk}{1,1,1,1}
\definecolor{colordef}{rgb}{0.2,0.5,0.07} 
\definecolor{colorprop}{rgb}{0.2,0.1,0.5}

\definecolor{color1}{rgb}{0.6,0.4,0.8}
\definecolor{color1}{rgb}{0.9,0.6,0.4}
\definecolor{color1}{rgb}{0.36, 0.54, 0.66}

\definecolor{color2}{rgb}{0.2,0.1,0.5}
\definecolor{color2}{rgb}{0.91, 0.84, 0.42}
\definecolor{color2}{rgb}{0.87, 0.36, 0.51}
\definecolor{color2}{rgb}{0.4, 0.69, 0.2}
\definecolor{color2}{rgb}{0.84, 0.23, 0.24}

\definecolor{color3}{rgb}{1.0, 0.13, 0.32}
\definecolor{color3}{rgb}{0.54, 0.17, 0.89}

\definecolor{babypink}{rgb}{0.96, 0.76, 0.76}

\definecolor{babyblue}{rgb}{0.76, 0.76, 0.95}

\makeatletter
\renewenvironment{proof}[1][\proofname]{\par
  \normalfont
  \trivlist
  \item[\hskip\labelsep\itshape
    #1.]\ignorespaces
}{%
  \endtrivlist
}
\makeatother

\newtheorem{lem}{Lemma}[section]
\newtheorem{cor}[lem]{Corollary}

\newtheorem{thm}[lem]{Theorem}
\newtheorem{fact}[lem]{Fact}
\newtheorem{prop}[lem]{Proposition}
\theoremstyle{definition}
\newtheorem{definition}[lem]{Definition}
\newtheorem{ex}[lem]{Example}

\newtheorem{rmk}[lem]{Remark}

\newtheorem{conj}[lem]{Conjecture}

\newtheorem{notation}[lem]{Notation}

\numberwithin{equation}{section}

\newcommand{\RUN}{\mathcal{A}_1}

\newcommand{\LP}{\mathfrak{p}}
\newcommand{\pr}{\mathsf{pr}}

\newcommand{\lenk}{\mathsf{L}_\Omega}
\newcommand{\PRSP}{\Lambda^p \mathbb{R}^{p+q}}

\newcommand{\phoo}{\leftrightsquigarrow}

\newcommand{\GrassC}{\operatorname{Gr}_p (\mathbb{C}^{p+q})}
\newcommand{\GrassH}{\operatorname{Gr}_p (\mathbb{H}^{p+q})}
\newcommand{\Grassq}{\operatorname{Gr}_q (\mathbb{R}^{p+q})}

\newcommand{\compact}{\mathcal{C}}

\def\O{\Omega}

\newcommand{\Grass}{\operatorname{Gr}_p (\mathbb{R}^{p+q})}
\newcommand{\id}{\operatorname{id}}

\newcommand{\Fl}{\mathcal{F}}
\newcommand{\Phot}{\ell}
\newcommand{\Affstd}{\mathbb{A}}
\newcommand{\Affstdstd}{\Affstd^-}
\newcommand{\MRr}{\mathcal{R}}

\def\Rr{\operatorname{Extr}_{\MRr}}
\newcommand{\Contlam}{\mathsf{X}}
\newcommand{\hyp}{\operatorname{Z}}

\newcommand{\ncdn}{\mathsf{F}_{G'}^\opp}
\newcommand{\ncd}{\mathsf{F}_{(\mathfrak{g}, \alpha)}}
\newcommand{\uu}{\mathfrak{u}}
\newcommand{\Cartinv}{\sigma_\mathfrak{g}}
\newcommand{\g}{\mathfrak{g}}

\newcommand{\kk}{\mathfrak{k}}

\newcommand{\aaa}{\mathfrak{a}}
\newcommand{\oppinv}{\mathsf{i}}

\newcommand{\tr}{\mathsf{t}}

\newcommand{\vv}{\mathsf{v}}

\newcommand{\FS}{\operatorname{\Delta}}

\newcommand{\Cf}{\mathbb{C}}
\newcommand{\Rf}{\mathbb{R}}
\newcommand{\Hf}{\mathbb{H}}

\newcommand{\isl}{\zeta}

\newcommand{\ro}{\mathsf{k}}
\newcommand{\kob}{\mathsf{K}}
\newcommand{\varphistd}{\varphi_-}

\newcommand{\plong}{\tau}
\newcommand{\Ad}{\operatorname{Ad}}
\newcommand{\ad}{\operatorname{ad}}
\newcommand{\leng}{\mathsf{Len}_{\O}}
\newcommand{\plongsl}{\operatorname{j}}

\newcommand{\pstd}{\operatorname{pr}}
\newcommand{\E}{\operatorname{E}}
\newcommand{\F}{\operatorname{F}}
\newcommand{\He}{\operatorname{H}}
\newcommand{\nG}{\mathsf{n}(G)}

\newcommand{\Iso}{\mathrm{Isom}}

\newcommand{\Aut}{\mathsf{Aut}}

\newcommand{\SL}{\operatorname{SL}}
\newcommand{\PO}{\operatorname{PO}}
\newcommand{\GL}{\operatorname{GL}}
\newcommand{\PGL}{\operatorname{PGL}}
\newcommand{\SU}{\operatorname{SU}}
\newcommand{\Sp}{\operatorname{Sp}}
\newcommand{\spp}{\mathfrak{sp}}

\newcommand{\sll}{\mathfrak{sl}}

\newcommand{\SO}{\operatorname{SO}}
\newcommand{\soo}{\mathfrak{so}}
\renewcommand{\qed}{\hfill\square}
\DeclareMathOperator{\Ker}{ker}

\newcommand{\suu}{\mathfrak{su}}
\newcommand{\eeee}{\mathfrak{e}}
\newcommand{\opp}{-}

\newcommand{\std}{\mathsf{std}}

\newcommand{\inver}{\mathsf{s}}
\newcommand{\typ}{\mathsf{typ}}
\newcommand{\pos}{\mathsf{pos}}
\newcommand{\Oo}{\mathcal{O}_\O}
\newcommand{\Extr}{\operatorname{Extr}_{\mathcal{R}}}
\newcommand{\cpl}{\mathbf{C}_\O}

\makeatletter
\newcommand{\tpitchfork}{%
  \vbox{
    \baselineskip\z@skip
    \lineskip-.52ex
    \lineskiplimit\maxdimen
    \m@th
    \ialign{##\crcr\hidewidth\smash{$-$}\hidewidth\crcr$\pitchfork$\crcr}
  }%
}
\makeatother

\title[Metrics on domains in Nagano spaces]{Metric properties of domains in real-type Nagano spaces}
\author{Blandine Galiay}

\begin{document}

\begin{abstract} Nagano spaces are compact symmetric spaces that admit large transformation groups. They include for instance all the Grassmannians and the Einstein Universes. In this paper, we study a Kobayashi-type pseudometric on domains in real-type Nagano spaces. When the Nagano space is real projective space, this metric coincides with the classical Kobayashi pseudometric. For a dually convex domain of a general real-type Nagano space, we prove that this pseudometric is a genuine metric if and only if the domain does not contain a photon minus a point. We compute this metric on the proper symmetric domains and prove that it is obtained by integrating the~$L^1$-norm along flats. 

We prove that in higher rank, the Kobayashi metric of a strongly~$\mathcal{R}$-proper dually convex divisible domain is never Gromov hyperbolic. This contrasts with the rank-one case corresponding to real projective space, where a classical result of Benoist shows that this metric is Gromov hyperbolic if and only if the domain is strictly convex. 

\end{abstract}

\maketitle

\section{Introduction}

Initially introduced in the complex setting by Kobayashi in the 1960s \cite{kobayashi1967invariant}, Kobayashi pseudometrics in complex geometry have given rise to real analogues. For instance, in real projective geometry, a Kobayashi pseudometric can be defined using chains of projective segments \cite{kobayashi1984projectively}. In conformal pseudo-Riemannian geometry, an analogous pseudometric has been introduced by Markowitz, replacing projective segments by segments of \emph{photons} \cite{markowitz_1981}. These constructions were later adapted to domains (i.e.\ connected open subsets) of different \emph{flag manifolds}, such as Grassmannians by Limbeek--Zimmer \cite{van2019rigidity}, or \emph{causal flag manifolds} by the author \cite{galiay2024rigidity}.

The first goal of this paper is to place these real constructions into the Lie-theoretic generality of \emph{real-type Nagano spaces}, whose classification is known \cite{nagano1965transformation} and recalled in Table~\ref{table_nagano_full} on page~\pageref{table_nagano_full}: they include real projective space, \emph{Einstein Universes} (item~(viii) of Table~\ref{table_nagano_full}), but also \emph{minimal Grassmannians} (item~(iv)), \emph{causal flag manifolds} (items (iii), (viii, $\min(p,q)=2$), (x), (xi), and (xii)), and some exceptional flag manifolds. These pseudometrics are built from chains of \emph{photon} segments, whose construction we recall.

We identify a natural family of domains on which the Kobayashi pseudometric is a genuine metric, namely the family of $\MRr$-proper photon-convex domains. As a consequence, they enjoy the following property: their automorphism group acts properly on them (Corollary~\ref{cor_action_proper}). In the real projective case, these domains are exactly properly convex domains, but in general they may not be bounded in an affine chart. We thus extend the considerations of \cite{zimmer2015convexity, van2019rigidity, galiay2024rigidity} to a broader class of domains. In the special case of the \emph{Einstein universe}, these domains were also studied independently by Chalumeau with techniques from conformal pseudo-Riemannian geometry, see \cite{chalumeau2024metriques, chalumeau2025markowitz, chalumeau2026gromov}.

The second goal of this paper is to apply our construction to geometric questions on manifolds locally modelled on Nagano spaces. In particular, Kobayashi pseudometrics have already proven to play a key role in proving partial results of a rigidity conjecture of Limbeek--Zimmer (Conjecture~\ref{question_lim_Zim}), see \cite{van2019rigidity, galiay2024rigidity, chalumeau2024rigidity}. Our construction provides new directions toward Limbeek--Zimmer's conjecture: for instance, building on work of Zimmer \cite{zimmer2015convexity}, we show that the automorphism group of a strongly~$\MRr$-proper divisible dually convex domain in a higher-rank Nagano space is \emph{never} Gromov hyperbolic. This result contrasts with the rank-one case, corresponding to real projective space. More generally, the analysis carried out in this paper thus allows us to formally relate Limbeek--Zimmer's conjecture for real-type Nagano spaces to a higher-rank phenomenon.

\subsection{Domains in Nagano spaces}\label{sect_photons_intro} \emph{Nagano spaces}, also called \emph{R-symmetric spaces} (``R'' for ``root'') or \emph{extrinsic symmetric spaces}, are compact symmetric spaces~$N$ admitting a noncompact transformation group, i.e.\ such that there exists a noncompact semisimple Lie group~$G$ acting transitively on~$N$ with compact kernel. Their geometry can be understood as generalizing both conformal (pseudo-)Riemannian geometry and projective geometry, and they are classified, see Table~\ref{table_nagano_full}. We refer for instance to \cite{nagano1965transformation,chow1949geometry, kobayashi1964filtered, takeuchi1965cell, kobayashi1965filtered, takeuchi1968minimal, loos1971jordan, peterson1987arithmetic, takeuchi1988basic, makarevivc1973open, kaneyuki1998sylvester, kaneyuki2006causal, kaneyuki2011automorphism} for a rich literature on this topic.

Every Nagano space decomposes as a product of \emph{irreducible} Nagano spaces, i.e.\ spaces whose transformation group $G$ is a \emph{simple} Lie group; most arguments concerning Nagano spaces reduce to the
irreducible case.

A domain in a Nagano space is a connected open set. We are in particular interested in a special class of domains, those that are \emph{proper}, and recall their definition. A fundamental result is that every Nagano space with transformation group~$G$ is a
\emph{flag manifold} of~$G$ \cite{nagano1965transformation}; in other words, the stabilizer of a point
$x \in N$ is a \emph{parabolic subgroup} $P$ of $G$, which yields the
identification $N \simeq G/P$.
In particular, such a space admits \emph{affine charts}. These are distinguished dense open subsets canonically diffeomorphic to affine spaces; this structure generalizes the well-known one of real projective space, see Section~\ref{sect_incidence} for more details. A domain $\O \subset N$ is said to be \emph{proper} if it
is bounded in an affine chart. We can also define the \emph{automorphism group} of a domain~$\O \subset N \simeq G/P$ as
$$\Aut(\O) := \{ g \in G \mid g \cdot \O = \O \}~.$$

Nagano proved that the \emph{non-compact dual} symmetric space $\mathbb{X}(N)$ of a Nagano space
$N$, that is, the unique non-compact symmetric space having the same
complexification as $N$, embeds~$\Iso(\mathbb{X}(N))$-equivariantly into $N$ as a \emph{proper domain} \cite{nagano1965transformation}. Two such embeddings are always conjugate in~$G$. We call the image of such an embedding a \emph{realization of
$\mathbb{X}(N)$ in $N$}; up to translation by an element of~$G$, there is a unique realization of~$\mathbb{X}(N)$ in~$N$. For instance, real hyperbolic space is the non-compact dual of real projective space $N = \mathbb{P}(\mathbb{R}^n)$, and the Klein model
\begin{equation}\label{eq_hyperbolique}
    \mathbb{H}^n
    = \mathbb{P}\bigl(\{x \in \mathbb{R}^{n+1} \mid
    x_1^2 + \dots + x_n^2 - x_{n+1}^2 > -1\}\bigr)
\end{equation}
realizes it as a proper domain of~$\mathbb{P}(\mathbb{R}^{n+1})$.

\subsection{Kobayashi metric and~$\MRr$-properness}\label{sect_intro_kob} 
We define and study a \emph{Kobayashi pseudometric} on domains of real-type Nagano spaces. Its construction relies on the existence of \emph{photons}, which we introduce in Section~\ref{sect_intro_photons} below.

\subsubsection{Photons}\label{sect_intro_photons} Among irreducible Nagano spaces, we single out those that are
\emph{of real type}. Any Nagano space admits a generic~$G$-invariant family of natural spheres, called \emph{Helgason spheres} (see e.g.\ \cite{takeuchi1988basic}). When these spheres are one-dimensional, we say that the Nagano space is \emph{of real type}. In this case, Helgason spheres are projective lines, called \emph{photons}. In Table~\ref{table_nagano_full}, the real-type Nagano spaces are those for which the column ``$\dim(\mathfrak{g}_\alpha)$'' contains a~$1$; their item case is greyed out. This is for instance the case of the Grassmannian $\Grass$ of
$p$-planes in $\mathbb{R}^n$.

The photons in real projective space are exactly the projective lines. From another perspective, in the flat model of conformal pseudo-Riemannian geometry (the \emph{Einstein Universe}), they are exactly the non-oriented null geodesics. Genralizations have been investigated in \cite{van2019rigidity} for Grassmannians, in \cite{galiay2024rigidity} in causal flag manifolds, and in \cite{beyrer2024positivity} for general flag manifolds defined by roots of multiplicity 1.

Photons in flag manifolds have recently proven useful in the study of higher Teichmüller spaces \cite{beyrer2024positivity}. This suggests that their properties are of intrinsic interest. In this work we further investigate them in the setting of Nagano spaces; see, for instance, Section~\ref{sect_type_I}.

\subsubsection{The Kobayashi pseudometric in Nagano Spaces}
Given two points~$x,y $ in a domain~$\O$, a \emph{chain of photons} is a continuous path defined by concatenating segments contained in photons. Since each of these photons is a projective line, one may compute the length of such a path as the sum of the Hilbert lengths of each of its segments of photon. The Kobayashi pseudo-distance~$\kob_\O(x,y)$ between~$x$ and~$y$ is the infimum of the lengths of chains of photon joining~$x$ to~$y$ inside~$\O$ (see Section~\ref{sect_definition_kobayashi} for details). For this distance to be finite, we require photons to exist in generic positions, in order to always be able to find a chain of photons between two points of~$\O$. This is ensured by the condition of being a Nagano space; see Fact~\ref{fact_kostant}. For general flag manifolds, this construction fails. This is why in this paper we restrict ourselves to real-type Nagano spaces.

\begin{rmk}
    Kobayashi pseudometrics exist more generally on \emph{parabolic geometries} (see~\cite[Rmk~2.11]{markowitz_1981}), which are manifolds \emph{infinitesimally} modeled on flag manifolds. The construction presented in this paper concerns \emph{flat kleinian} parabolic geometries, which constitute a natural first step, but a subsequent direction would be to further investigate the non-flat case (see Remark~\ref{rmk_infinitesimal_form} for more details). Indeed, to the best of our knowledge, Kobayashi pseudometrics constructed from photons have not been isolated and studied in the general setting of (even flat kleinian) parabolic geometries modeled on Nagano spaces. 
\end{rmk}

\subsubsection{$\MRr$-properness and photon-convexity}\label{sect_intro_def_R_proper} Photons allow us to define a natural family of domains in $N$, namely those that do not contain a photon minus a point (resp.\ whose closure does not contain a photon); these are called \emph{$\mathcal{R}$-proper} (resp.\ \emph{strongly~$\MRr$-proper}) domains, following \cite{van2019rigidity}. This definition can be viewed as an analogue of the well studied condition of \emph{Brody hyperbolicity} in complex geometry \cite{brody1978compact}. One clearly has:
\begin{equation}\label{eq_implication_properness}
    \text{properness } \Longrightarrow \text{ strong~$\MRr$-properness } \Longrightarrow \text{ $\MRr$-properness}~.
\end{equation}
We can also define a natural notion of convexity: a domain~$\O \subset N$ is \emph{photon-convex} if~$\Phot \cap \O$ is connected for every photon~$\Phot$, and if~$\overline{\O \cap \Phot} = \overline{\O} \cap \Phot$ whenever~$\Phot \cap \O \ne \varnothing$. We prove in Lemma~\ref{lem_continuité_intersection_photon} that any~$\MRr$-proper dually convex domain is photon-convex.

When~$N$ has higher rank, the implications in~\eqref{eq_implication_properness} are never equivalences, even assuming the domain to be photon-convex.

\subsubsection{Kobayashi hyperbolicity} The families of domains introduced in the previous section are good candidates for Kobayashi-hyperbolic domains (i.e.\ domains on which the Kobayashi pseudometric is a metric). We prove:

\begin{thm}\label{thm_kobayashi_metric_nagano} Let~$N$ be an irreducible real-type Nagano space, and let~$\O \subset N$ be a domain.
    \begin{enumerate}
        \item (Prop.~\ref{prop_generate_topo_std} and  Cor.~\ref{cor_kobayashi_geodesic}) If~$\O$ is \emph{proper}, then~$\kob_\O$ is an $\Aut(\O)$-invariant metric generating the standard topology on~$\O$. If~$\O$ is moreover dually convex, then~$\kob_\O$ is a proper geodesic metric, and segments of photons are geodesics.
        \item (Thm.~\ref{thm_kob_hyp_R_propre} and Cor.~\ref{cor_photon_convex_quasi_hom_complete}) If~$\O$ is~\emph{$\MRr$-proper and photon-convex}, then~$\kob_\O$ is an $\Aut(\O)$-invariant metric generating the standard topology on~$\O$. If~$\O$ is moreover dually convex, then segments of photons are geodesics. Finally, if moreover~$\Aut(\O)$ acts cocompactly on~$\O$, then~$\kob_\O$ is a proper geodesic metric. 
        
        \item (Prop.~\ref{prop_kob_hyp_implies_R_proper}) If~$\O$ is dually convex and~$\kob_\O$ is a metric, then~$\O$ is~$\MRr$-proper.
    \end{enumerate}
\end{thm}

The proof of Theorem~\ref{thm_kobayashi_metric_nagano}.(1) is essentially identical to that for causal flag manifolds \cite{galiay2024rigidity}. By contrast, Theorem~\ref{thm_kobayashi_metric_nagano}.(2) and its proof are new, and generalize techniques from \cite{van2019rigidity}, which deals with the Grassmannian case (item~(iv) of Tabe~\ref{table_nagano_full}), when~$\O$ is also \emph{convex in an affine chart}. 

An important consequence of Theorem~\ref{thm_kobayashi_metric_nagano} is that the automorphism group of an~$\MRr$-proper photon-convex domain~$\O$ always acts properly on~$\O$ (Corollary~\ref{cor_action_proper}). This result generalizes, in Nagano spaces, the analogous one proven by Zimmer for proper domains in flag manfifolds \cite{zimmer2018proper}.

Note that Theorem~\ref{thm_kobayashi_metric_nagano}.(1) was also proven in joint work with Chalumeau for item (viii) in Table~\ref{table_nagano_full} \cite{chalumeau2024rigidity}. In this case, the standard projective embedding of~$G/P$ (called the \emph{Einstein universe}) is well understood and makes the proof more direct. A similar version of Theorem~\ref{thm_kobayashi_metric_nagano}.(2) for item (viii) in Table~\ref{table_nagano_full} was then independently proven in \cite[Thm~C]{chalumeau2025markowitz} by Chalumeau, with techniques coming from pseudo-Riemannian geometry.

\begin{rmk} Some aspects of this construction extend to general Nagano spaces. Indeed, one can always define a Kobayashi pseudometric by replacing photons with the \emph{Helgason spheres} mentionned in Section~\ref{sect_intro_photons}. However, several arguments become more involved in this greater generality. A particularly interesting case is that of complex Nagano spaces (those for which the column ``$\dim(\g_{\alpha})$'' in Table~\ref{table_nagano_full} contains a~$2$), where photons are replaced by complex projective lines. This situation is well understood in complex projective space (item~(v, $\min(p,q)=1$)), as studied in \cite{kobayashi1984projectively}.
\end{rmk}

\subsubsection{Computation on realizations of the noncompact dual} The Kobayashi metric is in general hard to compute on concrete examples, and it is moreover hard to find a geodesic between two points which are not on a same photon. However, Proposition~\ref{lem_key_lemma} gives us a tool to compute it. The tools developped in this paper as well as a fine analysis of the boundary of realizations of the noncompact dual, allow us to compute the Kobayashi metric on it.

\begin{thm}[Thm~\ref{thm_calcul_kob}]\label{prop_calcul_geod}
Let~$N$ be an irreducible real-type Nagano space of rank~$r \geq 1$, and let~$\O \subset N$ be a realization of the noncompact dual~$\mathbb{X}(N)$ of~$N$. Any pair of points of~$\O$ can be joined by a geodesic $r$-chain for~$\kob_\O$. Furthermore, the pullback of the Kobayashi metric restricted to a flat of~$\mathbb{X}(N)$ coincides with the integration of the $L^1$-norm on the flat.
\end{thm}

In projective geometry, the Kobayashi metric on the ellipsoid, which is also its Hilbert metric, coincides with the hyperbolic metric. In contrast, the last assertion of Theorem~\ref{prop_calcul_geod} implies that when~$N$ has higher rank, the Kobayashi metric on a realization of its noncompact dual is \emph{never} Riemannian.

Note that Theorem~\ref{prop_calcul_geod} for item~(viii) of Table~\ref{table_nagano_full} was already proven in collaboration with Chalumeau \cite[Ex.\ 4.2]{chalumeau2024rigidity}, relying on two arguments that were specific to that case: first, the infinitesimal form of~$\kob_\O$ given by Markowitz \cite{markowitz_1981}, and second, the fact that the noncompact dual is a product of rank-one Lie groups. The argument would thus be more involved in general. In the present paper, our proof relies on an analysis of the boundary of realizations of the noncompact dual and on the explicit construction of geodesic~$r$-chains. This allows us to compute all the Caratheodory metrics on realizations on the noncompact dual (see Theorem~\ref{thm_calcul_kob}).

\subsubsection{Constructions and obstructions}

In Section~\ref{sect_examples_kob_hyp}, we provide and elementary sufficient condition for a domain to be~$\MRr$-proper. One may also ask under which conditions such domains can be preserved by groups whose limit set is sufficiently generic, extending an analysis carried out in \cite{galiay2025transverse} for proper domains. Theorem~\ref{thm_kobayashi_metric_nagano}, combined with the analysis of \cite{galiay2025transverse}, allows us to show that for certain Nagano spaces, there are no~$\MRr$-proper domains preserved by a group with at least three elemts in its limit set: see Corollary~\ref{thm_Z_dense_Nagano_prop_I}.

\subsection{Proper divisible domains in flag manifolds and rigidity}\label{sect_intro_proper_divisible}

A proper domain~$\O$ in a flag manifold is \emph{divisible} (resp.\ \emph{almost-homogeneous}) if there exists a discrete subgroup of~$\Aut(\O)$ acting cocompactly on~$\O$ (resp.\ if any element of~$\partial \O$ is in the closure of the~$\Aut(\O)$-orbit of a point of~$\O$). 

\subsubsection{Two flexible cases} \label{sect_proj_case_intro}
When~$G = \PGL(n, \mathbb{R})$ and~$P$ is the stabilizer of a line in~$\mathbb{R}^n$, i.e.\ when~$G/P = \mathbb{P}(\mathbb{R}^n)$, such domains are well understood and are called \emph{divisible convex sets}. All proper symmetric domains can be identified with Riemannian symmetric spaces, and their isometry group coincides with their automorphism group
(as open subsets of projective space).  They are therefore divisible, and their complete classification is known
in all dimensions \cite{koecher1999minnesota}.
The simplest example is real hyperbolic space~$\mathbb{H}^{n-1}$, embedded
in projective space~$\mathbb{P}(\mathbb{R}^{n})$ via the Klein model defined in~\eqref{eq_hyperbolique}. There also exist nonsymmetric divisible convex domains, although their construction is generally technical and remains an active area of research \cite{koszul1968deformations,
vinberg1967quasi, johnson1987deformation, choi1993convex, kapovich2007convex, benoist2006convexes, choi2020convex,
ballas2018convex, blayac2024divisible}. The diversity of proper divisible domains in~$\mathbb{P}(\mathbb{R}^n)$ highlights the importance of general results concerning them, beyond those related to cocompact actions on Riemannian symmetric spaces.

Here, the $n$-dimensional \emph{conformal sphere} is the unique flag manifold of the group~$G = \PO(n+1,1)$, where $P$ is the unique (up to conjugation)
proper parabolic subgroup of $G$. In dimension~$2$, it is known to
admit proper non-symmetric divisible domains: take one connected component of the complement of the limit set of a quasi-Fuchsian non-Fuchsian representation into~$\PO(3,1)$. Here ``proper''
means that the complement of the domain has non-empty interior.

\subsubsection{The conjecture of Limbeek--Zimmer}\label{sect_question_lim_zim} The question of whether the theory of proper divisible domains of real projective space can be generalized to flag manifolds other than real projective space was
raised by Limbeek--Zimmer.  The realizations of the noncompact dual of a Nagano space (Section~\ref{sect_photons_intro}) are divisible proper domains, but they are symmetric.  The problem is to determine whether nonsymmetric examples exist, as in real projective case:

\begin{conj}[\cite{van2019rigidity}]\label{question_lim_Zim}\label{question_Lim_Zim}
    Given a flag manifold~$G/P$ different than real projective space and conformal sphere, any proper divisible domain of~$G/P$ is symmetric.
\end{conj}

Conjecture~\ref{question_Lim_Zim} has been solved for particular classes of flag manifolds~\cite{zimmer2018proper, galiay2024rigidity, chalumeau2024rigidity}, and partially solved for some other classes in~\cite{wong1977characterization, rosay1979caracterisation, frankel1989complex, zimmer2013rigidity, van2019rigidity}. It remains open in many cases of interest, including Grassmannians of~$p$-planes of~$\mathbb{R}^n$, $2 \leq p \leq n-2$.

\subsubsection{Higher rank and rigidity} There is a deep connection between the rank of an irreducible real-type Nagano space~$N$ and its \emph{degrees of transversality}; see for instance Fact~\ref{thm_takeuchi_hyp}. This rank is~$1$ if and only if~$N$ is real projective space. When the rank is at least~$2$, strong geometric constraints appear on the structure of the boundary of a proper almost-homogeneous domain (see e.g.\  Proposition~\ref{prop_geom_prop_extremal_points}). In \cite{van2019rigidity, chalumeau2024rigidity, galiay2024rigidity}, these constraints play a key role in the proof of the rigidity of proper divisible domains. We thus interpret this rigidity as a higher-rank phenomenon.

In the following paragraphs, we give some applications of the construction of the Kobayashi metric and other general objects introduced in this paper, that support Conjecture~\ref{question_lim_Zim}.

\subsubsection{Faces and complexity}\label{sect_intro_facettes}
In Section~\ref{sect_R_extremality}, \emph{$\MRr$-faces} are defined as generalizations of the well-known faces of boundary points of properly convex domains in projective space (see Definition~\ref{def_R_face}). This allows to define the \emph{complexity}~$\cpl(p)$ of a boundary point~$p$ of a proper domain~$\O$ (see Definition~\ref{def_complexity_boundary_point}), which measures the level of transversality of~$p$ with respect to the other points of the closure of its~$\MRr$-face. In Theorem~\ref{lem_geom_prop_visuel_general}, we obtain a lower bound on the complexity of~$p$ when~$\O$ is almost-homogeneous, involving the rank of the Nagano space~$N$:
\begin{thm}\label{lem_geom_prop_visuel_general}
Let~$N$ be an irreducible real-type Nagano space. 
Let~$\O$ be a proper almost-homogeneous domain of~$\Fl(\g,\alpha)$, and let~$p \in \partial \O$. Then one has:
\begin{equation}\label{eq_complexity}
    \min_{x \in \O} d_H(x, p) \;\geq\; \operatorname{rk}(N) - \cpl(p)~.
\end{equation}
\end{thm}

Here~$d_H(x,p)$ is the \emph{arithmetic distance} between~$x$ and~$p$ (see Section~\ref{sect_arithm_distance}), which measures the degree of transversality between~$p$ and~$x$; for instance, if~$p$ and~$x$ are distinct and lie on the same photon, then~$d_H(x,p) = 1$; if they are in generic positions, then~$d_H(x,p)$ is maximal. Equation~\eqref{eq_complexity} implies in particular the existence, in higher rank, of boundary points of~$\O$ with large complexity; these points are far from being extremal. The fact of having a boundary with ``big'' faces is in a sense a higher-rank behavior, even in real projective space (see e.g.\ \cite{zimmer2020higher}). This result supports the idea that proper almost-homogeneous domains in higher-rank irreducible real-type Nagano spaces exhibit higher-rank behaviour.

\subsubsection{Non-Hyperbolicity of the Kobayashi metric}\label{sect_intro_hyp}
Gromov hyperbolicity of a metric space is typically a ``rank-one'' behaviour. In Section~\ref{sect_non_hyp_kob}, we show that for a large family of almost-homogeneous, Kobayashi-hyperbolic domains~$\O$, the distance~$\kob_\O$ cannot be Gromov hyperbolic:
\begin{thm}\label{thm_group_non_hyperbolic_1}
Let~$N$ be a higher-rank irreducible real-type Nagano space, with transformation group~$G$. Let~$\O \subset N$ be a domain. Assume that one of the following conditions holds:
\begin{enumerate}
    \item $\O$ is proper and almost-homogeneous;
    \item $\O$ is strongly~$\MRr$-proper, dually convex and quasi-homogeneous.
\end{enumerate}
Then~$(\O, \kob_\O)$ is not Gromov hyperbolic.
\end{thm}

This result contrasts with the situation in rank~1, corresponding to the case~$N = \mathbb{P}(\mathbb{R}^{n+1})$ (Table~\ref{table_nagano_full}.(iv) for~$p =1$), where Benoist has proven that the Gromov hyperbolicity of the Kobayashi metric on a proper divisible domain is equivalent to the strict convexity of the domain \cite{benoist2001convexes}, see Remark~\ref{rmk_comp_rang_1}. Our proof generalizes that of \cite[Thm~1.5]{zimmer2015convexity}, which concerned proper quasi-homogeneous domains in real Grassmannians. With the tools developed in this paper, we can adapt the proof to any higher-rank irreducible real-type Nagano space, and to the~$\MRr$-proper case.  The idea, in te notation of Theorem~\ref{thm_group_non_hyperbolic_1} is to find a boundary point~$p$ of~$\O$ with nontrivial~$\MRr$-face and a photon through~$p$ that enters~$\O$. This leads to the existence of thick rectangles whose sides are segments of photons. The existence of such a point~$p$ is not ensured anymore when~$\O$ is not almost-homogeneous. However, in the special case of item~(viii, $\min(p,q) = 1$) in Table~\ref{table_nagano_full}, an explicit description of the boundary of dually convex domains in the \emph{Lorentzian Einstein Universe} given by Smaï \cite{smai2025maximality} allows to recover a stronger version of Theorem~\ref{thm_group_non_hyperbolic_1}.(2), as proven independently by Chalumeau \cite{chalumeau2026gromov}.

For the Nagano spaces corresponding to items~(iii, viii, x, xi, xii) of Table~\ref{table_nagano_full}, \cite[Thm 1.4]{galiay2024rigidity} and~\cite[Thm 1.1]{chalumeau2024rigidity} imply that proper almost-homogeneous domains of~$N$ are realizations of the higher-rank symmetric space~$\mathbb{X}(N)$. Hence Theorem~\ref{thm_group_non_hyperbolic_1}.(1) is just a consequence of the fact that the rank of~$\mathbb{X}(N)$ is at least~$2$. Since Theorem~\ref{thm_group_non_hyperbolic_1}.(1) is already proven in \cite{zimmer2015convexity} for the Grassmannians, the cases at issue here are the remaining higher-rank real-type Nagano pairs, namely items~(i, vii, xix) of  Table~\ref{table_nagano_full}. Besides, Theorem~\ref{thm_group_non_hyperbolic_1}.(2) is new in any case different from item~(viii, $\min(p,q) = 1$) of Table~\ref{table_nagano_full}.

\subsection{Symmetric behavior} The results of this paper support the idea that the proper almost-homogeneous domains of higher-rank real-type Nagano spaces should be higher-rank symmetric spaces, in particular the realizations of the non-compact dual. Concerning \emph{complex Nagano spaces} (those that are flag manifolds of complex Lie groups), a forthcoming paper in collaboration with Nicolas Tholozan proves Conjecture \ref{question_lim_Zim}: the only proper divisible domains of such Nagano spaces are the realizations of the noncompact dual. It is reasonable to expect the same phenomenon to occur in Nagano spaces with a quaternonic or octonionic structure.

Besides, we have the following direct consequence of \cite{nagano1965transformation}, that we prove in Section \ref{sect_charac_nagano}:  

\begin{lem}\label{lem_symmetric_nagano} 
Let $G/P$ be a flag manifold that is not a Nagano space. Then $G/P$ contains no symmetric domains. 
\end{lem}

This leads us to the followinf refinement of Conjecture~\ref{question_lim_Zim}:

\begin{conj}[Limbeek--Zimmer, reformulated]\label{question_Lim_Zim_1}
    Let~$G$ be a simple Lie group and let~$P$ be a parabolic subgroup of~$G$. Then:

    \begin{enumerate}
        \item If~$G/P$ admits proper almost-homogeneous domains, then it is an irreducible Nagano space.

        \item If~$G/P$ is an irreducible Nagano space~$N$, and if it is neither Example (ix) nor (iv, $\min(p,q)=1$) of Table~\ref{table_nagano_full}, then any proper almost-homogeneous domain of~$N$ is a realization of~$\mathbb{X}(N)$, and is in particular symmetric.
    \end{enumerate}
\end{conj}

Note that if~$G$ is \emph{semisimple} but not simple, then \cite[Thm~1.7]{zimmer2018proper} and Conjecture~\ref{question_Lim_Zim_1} allow one, conjecturally, to determine the proper almost-homogeneous domains in~$G/P$.

Throughout the paper, we explain, through remarks and comments, why the results stated in Section~\ref{sect_intro_proper_divisible} support this conjecture. We refer to \cite[Sect.\ 8.9]{galiay2025convex} for a more detailed discussion on this subject.

\subsection{Organization of the paper} In Section~\ref{sect_flag_mfds_domains}, we give reminders and preliminary lemmas on Lie theory, flag manifolds and their domains. 

In Section~\ref{sect_def_nagano_space}, we recall the definition of Nagano spaces, their properties, the embedding of their noncompact dual, and prove elementary characterizations of them among flag manifolds, and of real projective space among Nagano spaces. 

In Section~\ref{sect_HS_photons_defs}, we recall the definition and properties of photons. 

In Section~\ref{sect_type_I}, we define the \emph{real type} and isolate special embeddings of real-type Nagano spaces, defined from what we call \emph{Plücker triples}; these embeddings generalize the Plücker embedding of Grassmannians. We investigate the image of photons under these embeddings (Proposition~\ref{prop_photons_egal_proj_lines}). 

In Section~\ref{sect_kob}, we define the Kobayashi pseudometric on domains of real-type Nagano spaces and prove a comparison result between the Kobayashi metric and Caratheodory metrics (Proposition~\ref{lem_key_lemma}). This relies on the fundamental Lemma~\ref{lem_comp_crossratios}, which is useful in the remainder of the paper (for instance in the proof of Theorem~\ref{thm_kobayashi_metric_nagano}.(2)). 

In Section~\ref{sect_kob_hyperbolicity}, we prove Theorem~\ref{thm_kobayashi_metric_nagano}. 

In Section~\ref{sect_dynamics_self_opposite}, we focus on self-opposite Nagano spaces and investigate the dynamical properties of sequences of automorphisms (Lemma~\ref{lem_photon_contenu_dans_D}). We deduce Corollary~\ref{lem_geom_prop_visuel_4}, which allows us to prove Corollary~\ref{thm_Z_dense_Nagano_prop_I} in Section~\ref{sect_cor_type}; this corollary is central in the proof of Theorem~\ref{thm_group_non_hyperbolic_1}.(2). 

In Section~\ref{sect_examples_kob_hyp}, we construct elementary examples of~$\MRr$-proper domains, and in Lemma~\ref{lem_R_proper_not_proper}, we give a recipe for constructing~$\MRr$-proper nonproper domains, with Zariski-dense automorphism group. 

In Section~\ref{sect_R_extremality}, we define~$\MRr$-faces,~$\MRr$-extremality and complexity. In particular, in subsection~\ref{sect_convex_hull}, we investigate the convex hull of the image of a proper domain in projective space under the embedding induced by a Plücker triple; this allows us to prove a genericity property for~$\MRr$-extremal points (Lemma~\ref{lem_existence_extremal_dans_face}). Subsection~\ref{sect_geometric_prop_R_extr} is devoted to the proof of Proposition~\ref{prop_geom_prop_extremal_points}, which states a geometric property for~$\MRr$-extremal points of some almost-homogeneous domains. This proposition is the key property that forces the rigidity results of this paper (Theorems~\ref{lem_geom_prop_visuel_general} and~\ref{thm_group_non_hyperbolic_1}).

In Section~\ref{sect_geod_sym_domains}, we prove Theorem~\ref{prop_calcul_geod}.

In Section~\ref{sect_rigidity}, we prove Theorems~\ref{lem_geom_prop_visuel_general} and~\ref{thm_group_non_hyperbolic_1}. In Subsection~\ref{sect_rank_and_complexity}, we prove Theorem~\ref{lem_geom_prop_visuel_general}, after establishing the auxiliary Proposition~\ref{prop_proximal_limit_set} which is of independent interest. In Subsection~\ref{sect_non_hyp_kob}, we prove Theorem~\ref{thm_group_non_hyperbolic_1}. Finally, in Subsection~\ref{sect_rank_one_almost}, we prove another rigidity result (Proposition~\ref{prop_limit_proximal_pas_tout}) directly following from Proposition~\ref{prop_proximal_limit_set}.

\subsection*{Acknowledgements}
Much of the content of this paper was developed during my PhD, or arose from comments and questions by members of my committee. I would therefore like to thank my PhD advisor, Fanny Kassel; the referees of my dissertation, Olivier Guichard and Andy Zimmer; and the members of my defense committee, Beatrice Pozzetti, Karin Melnick, and Yves Benoist. I am also grateful to Yosuke Morita and Alex Nolte for the fruitful discussions that helped improve this paper. This material is based upon work supported by the National Science Foundation under Grant No.\ DMS-1928930, while the author was in residence at the Simons Laufer Mathematical Sciences Institute in Berkeley, California, during the first semester of 2026.

\section{Flag manifolds and their domains}\label{sect_flag_mfds_domains}
In this section, we provide a review of Lie theory, flag manifolds and discrete group actions on them, and their proper domains.

\subsection{Preliminaries on Lie theory}\label{sect_lie_theory} All the Lie groups and Lie algebras in this paper are supposed to be linear. Given a semisimple Lie group~$G$, we will always denote by~$\g$ its Lie algebra in this memoir. In this section, we fix a noncompact real semisimple Lie group~$G$.

\subsubsection{$\mathfrak{sl}_2$-triples} \label{sect_sl2-triples} A triple~$\tr = (e,h,f)$ of nonzero elements of~$\g$ satisfying the equalities~$[h,e] = 2e$,~$[h,f] = -2f$ and~$[e,f] = h$ is called an \emph{$\mathfrak{sl}_2$-triple}. There is a Lie algebras embedding~$\plongsl_\tr: \mathfrak{sl}_2(\mathbb{R}) \hookrightarrow \g$ such that~$\plongsl_\tr(\E) = e$,~$\plongsl_\tr (\He) = h$ and~$\plongsl_\tr(\F) = f$, where
\begin{equation*}
    \E = \begin{pmatrix}
        0 & 1 \\ 0 & 0
    \end{pmatrix}; \quad \He = \begin{pmatrix}
        1 & 0 \\ 0 & -1
    \end{pmatrix}; \quad \F = \begin{pmatrix}
        0 & 0 \\ 1 & 0
    \end{pmatrix}.
\end{equation*}

\subsubsection{Cartan decomposition}\label{sect_cartan_decomp} Let~$B$ be the Killing form on~$\g$. Let~$K \leq G$ be a maximal compact subgroup and~$\mathfrak{h}$ be the~$B$-orthogonal of the Lie algebra~$\mathfrak{k}$ of~$K$ in~$\g$. Then one has~$\g = \mathfrak{k} \oplus \mathfrak{h}$. The \emph{Cartan involution }of~$\g$ (with respect to~$K$) is then the Lie algebra automorphism~$ \Cartinv: \g \rightarrow \g$ defined by~$ (\Cartinv)_{|\mathfrak{k}} = \operatorname{id}_{\mathfrak{k}}$ and~$ (\Cartinv)_{|\mathfrak{h}} = -\operatorname{id}_{\mathfrak{h}}$. It induces a Lie group automorphism of~$G$, denoted by~$\Cartinv$ and called the \emph{Cartan involution of }$G$.

\subsubsection{The restricted Weyl group}\label{sect_restricted_weyl_grouup} Let~$\aaa \subset \mathfrak{h}$ be a maximal abelian subspace, and~$\g_0$ the centralizer of~$\aaa$ in~$\g$. We denote by~$\aaa^*$ the space of all linear forms on~$\aaa$. The \emph{restricted Weyl group}~$W$ of~$G$ is the quotient~$N_K(\aaa)/Z_K(\aaa)$ of the normalizer of~$\aaa$ in~$K$ (for the adjoint action) by the centralizer of~$\aaa$ in~$K$. ~For its natural embedding in~$\GL(\aaa)$, it is a finite group generated by the~$s_\alpha$, for~$\alpha \in \FS$. By duality with respect to~$B$ (which induces a scalar product on~$\aaa$), the action of~$W$ on~$\aaa$ induces an action on~$\aaa^*$ preserving~$\Sigma$. There exists a unique~$w_0 \in W$, called the \emph{longest element}, such that~$w_0 \cdot \Sigma^+ = -\Sigma^+$. The element~$\oppinv: \aaa^* \rightarrow \aaa^*$ defined as~$ \oppinv = - w_0$ is called the \emph{opposition involution}, and satisfies~$\oppinv(\FS) = \FS$.

\subsubsection{Restricted root system} \label{sect_real_lie_alg} For~$\alpha \in \aaa^*$, we define~$\g_{\alpha} := \{ X \in \g \mid [H, X] = \alpha(H)X \quad \forall H \in \aaa\}$. One has~$[\g_{\alpha}, \g_{\beta}] \subset \g_{\alpha + \beta}$ for any~$\alpha, \beta \in \aaa^*$. If~$\alpha \in \aaa^* \smallsetminus \{0\}$ satisfies~$\g_{\alpha} \ne \{0 \}$, then we say that~$\alpha$ is \emph{a restricted root} of~$(\g, \aaa)$. We denote by~$\Sigma = \Sigma(\g, \aaa)$ the set of all restricted roots of~$(\g, \aaa)$. One has~$\g = \g_0 \oplus \bigoplus_{\alpha \in \Sigma} \g_{\alpha}$. We fix a \emph{fundamental system}~$\FS = \{ \alpha_1, \dots , \alpha_N\} \subset \Sigma$, i.e.\ a family of restricted roots such that any root of~$\g$ can be uniquely written as~$\alpha = \sum_{i=1}^N n_i \alpha_i$, where the~$n_i$ all have same sign for~$1 \leq i \leq N$. The elements of~$\FS$ are called \emph{simple restricted roots}. From now on, whenever we fix a real semisimple Lie algebra of noncompact type~$\g$, it will always implicitly be endowed a fixed set~$\FS$ of simple restricted roots. 

The choice of a fundamental system determines a set of \emph{positive roots}~$\Sigma^+$, i.e.\ those roots~$\alpha$ where the~$n_i$ are all nonnegative.

For any~$\alpha \in \Sigma$ and~$X \in \g_{\alpha}\smallsetminus \{0\}$, there exists a unique scalar multiple~$X'$ of~$X$ such that~$(X',  [\Cartinv(X'), X'],  \Cartinv(-X'))$ is an~$\mathfrak{sl}_2$-triple. The element~$[ \Cartinv(X'), X']$ does not depend on the choice of~$X \in \g_{\alpha}$, and is denoted by~$h_{\alpha}$. We denote by~$s_\alpha$ the~$B$-orthogonal reflexion of~$\aaa$ with respect to~$\ker(\alpha)$. A representative of~$s_\alpha$ in~$K$ is given by~$\exp(\frac{\pi}{2}(X'+ \Cartinv(X')))$.

The family~$(h_{\alpha})_{\alpha \in \FS}$, whose elements are called the \emph{coroots of~$\g$}, forms a basis of~$\aaa$, whose dual basis in~$\aaa^*$ is denoted by~$(\omega_{\alpha})_{\alpha \in \FS}$.

The \emph{nonnegative Weyl chamber associated with~$\FS$} is~$\overline{\aaa}^+ = \{X \in \aaa \mid \alpha(X) \geq 0 \quad \forall \alpha \in \FS\}$.

\subsubsection{Parabolic subgroups} \label{sect_parab_subgroups} Let~$\Theta \subset \FS$ be a subset of the simple restricted roots. \emph{The standard parabolic subgroup}~$P_{\Theta}^+$ (resp.\ the \emph{standard opposite parabolic subgroup}~$P_{\Theta}^{\opp}$) is defined as the normalizer in~$G$ of the Lie algebra
\begin{equation}\label{eq_lie_u}
    \uu_{\Theta}^+ := \bigoplus_{\alpha \in \Sigma_{\Theta}^+} \g_{\alpha} \quad \Bigg{(}\text{resp.\ }  \uu_{\Theta}^- := \bigoplus_{\alpha \in \Sigma_{\Theta}^+} \g_{-\alpha} \Bigg{)}~,
\end{equation}
where~$\Sigma_{\Theta}^+ := \Sigma^+ \smallsetminus \operatorname{Span}(\FS \smallsetminus \Theta)$. By “standard”, we mean with respect to the above choices. The Lie algebra of~$P_\Theta^+$ (resp.~$P_\Theta^\opp$) is denoted by~$\LP_\Theta^+$ (resp.~$\LP_\Theta^\opp$). For any representative~$k_0 \in N_K(\aaa)$ of~$w_0$, one has~$k_0 \LP_{\Theta}^\opp k_0 = P_{\oppinv(\Theta)}^+$. 

More generally, a \emph{parabolic subgroup of type~$\Theta$} of~$G$ is a conjugate of~$P_{\Theta}^+$ in~$G$. A \emph{Borel subgroup} is a conjugate of~$P_{\FS}^+$ in~$G$. The \emph{Levi subgroup} associated with~$\Theta$ is the reductive Lie group defined as the intersection~$L_{\Theta} := P_{\Theta}^+ \cap P_{\Theta}^{\opp}$. The unipotent radical of~$P_{\Theta}^+$ (resp.\ $P_{\Theta}^{\opp}$) is~$U_{\Theta}^+ := \exp (\uu_{\Theta}^+)$ (resp.\ $U_{\Theta}^- := \exp (\uu_{\Theta}^-)$). One then has~$P_{\Theta}^+ = U_{\Theta}^+ \rtimes L_{\Theta}$ (resp.\ $P_{\Theta}^{\opp} = U_{\Theta}^- \rtimes L_{\Theta}$).

\subsection{Flag manifolds}\label{sect_flag_mfds}
The \emph{flag manifold} determined by~$G$ and~$\Theta$ is the set~$\Fl(\g, \Theta)$ of parabolic subalgebras of type~$\Theta$ of~$\g$. Since~$G$ acts transitively on~$\Fl(\g, \Theta)$ (by the adjoint action), and the stabilizer of~$\LP_\Theta^+$ is~$P_\Theta^+$, we have a natural identification~$G/P_\Theta^+ \simeq \Fl(\g, \Theta)$. The adjoint action of~$G$ on~$\Fl(\g, \Theta)$ by conjugation will be denoted by~$g \cdot \LP = \Ad(g)(\LP)$ for all~$\LP \in \Fl(\g, \Theta)$.

If~$\oppinv(\Theta) = \Theta$, then~$\Theta$ and~$\Fl(\g, \Theta)$ are said to be \emph{self-opposite}, and~$\Fl(\g, \Theta)= \Fl(\g, \Theta)^\opp$.

\subsubsection{Automorphisms of flag manifolds}\label{sect_aut_lie_alg} The group of all Lie algebra automorphisms of~$\g$ is called the \emph{automorphism group} of~$\g$ and denoted by~$\Aut(\g)$. It is a Lie group with Lie algebra~$\g$. When~$G$ is semisimple, the map~$\Ad: G \rightarrow \Aut(\g)$ has finite kernel.

In general, the group~$\Aut(\g)$ does not act on~$\Fl(\g, \Theta)$. However, it admits a finite-index subgroup that does: indeed, any~$g \in \Aut(\g)$ induces an automorphism~$\psi_g$ of the fundamental system~$\FS$. This defines a group homomorphism~$\Aut(\g) \rightarrow \Aut(\FS)$. For~$\Theta \subset \FS$, we denote by~$\Aut_\Theta(\g)$ the subgroup of~$\Aut(\g)$ of all Lie algebra automorphisms~$g$ such that~$\psi_g$ fixes~$\Theta$. This group contains the kernel~$\Aut_1(\g)$ of~$f$, which itself contains~$\Ad(G)$. The group~$\Aut_\Theta(\g)$ acts on~$\Fl(\g, \Theta)$. In particular~$\ker(\Ad)$ acts trivially on~$\Fl(\g, \Theta)$. 

\begin{definition}
    We denote by~$\mathcal{G}_\Theta(\g)$ the set of finite-index subgroups of~$\Aut_\Theta(\g)$. 
\end{definition}
Since~$\ker(\Ad)$ acts trivially on~$\Fl(\g, \Theta)$, we will always be able to assume that~$G \in \mathcal{G}_\Theta(\g)$, identifying it with its image under~$\Ad$.

\subsubsection{Incidence}\label{sect_incidence} Incidence is a way to compare two elements belonging to two different flag manifolds defined by~$\g$. It should be thought of as a function measuring the degrees of transversality between two flags. 

Formally, given~$\Theta \subset \FS$, the \emph{Weyl group with respect to~$\Theta$}, denoted by~$W_\Theta$, is the subgroup of the Weyl group~$W$ of~$G$ generated by the reflections~$s_\alpha$, with~$\alpha \in \Theta$.

Let~$\Theta, \Theta' \subset \FS$. Any element~$(x,y) \in \Fl(\g, \Theta) \times \Fl(\g, \Theta')$ can be written~$(x,y) = (g \cdot \LP_{\Theta}^+, \ gw \cdot \LP_{\Theta'}^+)$, with~$g \in G$ and~$w \in W$. Let 
$$\pos^{(\Theta, \Theta')}: \Fl(\g, \Theta) \times \Fl(\g, \Theta') \longrightarrow W_{\FS \smallsetminus \Theta} \backslash W / W_{\FS \smallsetminus \Theta'}$$
be the~$\operatorname{diag}(G)$-invariant map such that~$\pos^{(\Theta, \Theta')}(\LP_{\Theta}^+, w \cdot \LP_{\Theta'}^+) = \overline{w}$ for all~$w \in W$. This map only depends on the Lie algebra~$\g$ of~$G$. Given a point~$x \in \Fl(\g, \Theta)$ and an element~$\overline{w} \in W_{\FS \smallsetminus \Theta} \backslash W / W_{\FS \smallsetminus \Theta'}$, let 
\begin{equation*}
    \mathsf{C}_{\overline{w}}^{(\Theta, \Theta')}(x) = \{x' \in \Fl(\g, \Theta') \mid \pos^{(\Theta, \Theta')}(x,x') = \overline{w}\}~.
\end{equation*}
This follows from the definitions that for all~$w,w' \in W$, if~$\mathsf{C}_{\overline{w}}^{(\Theta, \Theta')}(x) \cap \mathsf{C}_{\overline{w}'}^{(\Theta, \Theta')}(x) \ne \varnothing$, then~$\overline{w} = \overline{w}'$.

\begin{ex}[Positions in the Grassmannians]\label{exx_grass} Let us consider~$\g = \sll(n, \mathbb{R})$, with its root system~$\FS = \{\alpha_1, \dots, \alpha_{n-1}\}$ of type~$A_{n-1}$. Then~$\Fl(\g, \alpha_p) = \operatorname{Gr}_p(\mathbb{R}^n)$ is the Grassmannian of~$p$-planes of~$\mathbb{R}^n$ (see item~(iv) of Table~\ref{table_nagano_full}). If~$ 1 \leq l\leq n-1$, then for all~$y \in \operatorname{Gr}_l(\mathbb{R}^n)$ and~$w \in W$, the integer
    \begin{equation*}
       p - \dim(y \cap x), \quad x \in \mathsf{C}_{\overline{w}}^{(\{\alpha_{l}\}, \{\alpha_{p}\})}(y)
    \end{equation*}
    is constant on~$\mathsf{C}_{\overline{w}}^{(\{\alpha_{l}\}, \{\alpha_{p}\})}(y) \subset \operatorname{Gr}_p(\mathbb{R}^n)$. The larger~$\overline{w}$ is for~$\leq$, the larger this number is. Hence~$\pos^{(\{\alpha_{l}\}, \{\alpha_p\})}$ simply detects the dimension of the intersection of a~$p$-plane with an~$l$-plane.
\end{ex}

The usual partial order on~$W_{\FS \smallsetminus \Theta} \backslash W / W_{\FS \smallsetminus \Theta'}$, defined in the following way: for~$w, w' \in W$, we have
\begin{equation*}
   \overline{w} \leq \overline{w}' \ \Longleftrightarrow \ \mathsf{C}_{\overline{w}}^{(\Theta, \Theta')}(x) \subset \overline{\mathsf{C}_{\overline{w}'}^{(\Theta, \Theta')}(x)} \text{ for some (hence any) }x \in \Fl(\g, \Theta)~.
\end{equation*}
Note that this partial order admits a maximum, which is the class~$\overline{w_0}$ of the longest element~$w_0$ of~$W$ defined in Section~\ref{sect_restricted_weyl_grouup}.

In the case where~$\Theta' = \oppinv(\Theta)$, a point~$x \in \Fl(\g, \Theta)$ is said to be \emph{transverse to~$y \in \Fl(\g, \Theta)^\opp$} if one has~$ \pos^{(\Theta, \oppinv(\Theta))}(x,y) = \overline{w_0}$. The set
\begin{equation}\label{eq_def_affine_chart}
\Affstd_y := \mathsf{C}_{\overline{w_0}}^{(\Theta, \oppinv(\Theta))}(y)
\end{equation}
is an open dense subset of~$\Fl(\g, \Theta)$, called an \emph{affine chart}. The set
\begin{equation}\label{eq_hyp_def}
    \hyp_y := \Fl(\g, \Theta) \smallsetminus \Affstd_y
\end{equation}
is an algebraic subvariety of~$\Fl(\g, \Theta)$ (resp.\ of~$\Fl(\g, \Theta)^\opp$), called a \emph{maximal proper Schubert subvariety}. 

The affine chart defined by~$P_\Theta^\opp$ will be called the \emph{standard affine chart} and be denoted by
\begin{equation}\label{eq_standard_affine_chart}
    \Affstdstd := \Affstd_{\LP_\Theta^\opp}~.
\end{equation}
The bijection
\begin{equation}\label{eq_A_egal_exp}
\begin{array}{cccc}
  \varphistd: &
        \uu_{\Theta}^- &\overset{\sim}{\longrightarrow} & \Affstdstd \\
        &X &\longmapsto &\exp(X) \cdot \LP_{\Theta}^+
\end{array}
\end{equation}
induces an affine structure on~$\Affstdstd$. Since~$G$ acts transitively on~$\Fl(\g, \Theta)^{\opp}$, any affine chart~$\Affstd_y$ with~$y \in \Fl(\g, \Theta)^\opp$ admits an affine structure, which moreover depends only on~$y$ (and not on the choice of~$g \in G$ such that~$y = g \cdot \LP_{\Theta}^{\opp}$).

\begin{lem}\label{lem_positions_weyl}
\begin{enumerate}
    \item Let~$x, y \in  \Fl(\g, \Theta)^\opp$ be a triple such that~$w_0 \in \pos^{(\oppinv(\Theta), \oppinv(\Theta))}(x,y)$. Then there exists~$z \in \Fl(\g, \Theta) \smallsetminus \hyp_y$ such that~$\id \in \pos^{(\Theta, \oppinv(\Theta))}(z,x)$.
    
    \item Let~$(x,y,z)$ be a triple of~$ \Fl(\g, \Theta)^2 \times \Fl(\g, \Theta)^\opp$ such that~$\id \in \pos^{(\Theta, \oppinv(\Theta))}(x,z)$ and~$w_0 \in \pos^{(\Theta, \oppinv(\Theta))}(y,z)$. Then~$w_0 \in \pos^{(\Theta, \Theta)}(x,y)$.
\end{enumerate}

\end{lem}

\begin{proof} Point (1) is proven in \cite[Lem.\ 8.1]{galiay2025transverse}. Let us prove (2). Since~$\id \in \pos^{(\Theta, \oppinv(\Theta))}(x,z)$, we may assume that
\begin{equation*}
    (x,z) = (\LP_{\Theta}^+, \LP_{\oppinv(\Theta)}^+)~.
\end{equation*}
By the Bruhat decomposition (see e.g.\ \cite{knapp1996lie}), there exist~$p \in P_{\FS}^+, w \in W$ such that~$y = p w \cdot \LP_{\Theta}^+$. One has~$w_0 \in \pos^{(\Theta, \oppinv(\Theta))} (y, \LP_{\oppinv(\Theta)}^+) = \pos^{(\Theta, \oppinv(\Theta))} (\LP_{\Theta}^+, w^{-1} \cdot \LP_{\oppinv(\Theta)}^+)$, which implies that~$\overline{w^{-1}} = \overline{w_0}$. Thus there exist~$(a,b) \in W_{\FS \smallsetminus \Theta} \times W_{\FS \smallsetminus \oppinv(\Theta)}$ such that~$w^{-1} = aw_0 b = ab'w_0$, with~$b' \in W_{\FS \smallsetminus \Theta}$. On the other hand, since~$W_{\FS \smallsetminus \Theta} \subset L_\Theta$, we have 
\begin{equation*}
    \pos^{(\Theta, \Theta)}(y , \LP_{\Theta}^+) = \pos^{(\Theta, \Theta)}(w_0 b'^{-1}a^{-1} \cdot \LP_{\Theta}^+, \LP_{\Theta}^+) = \pos^{(\Theta, \Theta)}(\LP_{\Theta}^+, w_0 \cdot \LP_{\Theta}^+)~.
\end{equation*}
Thus~$w_0 \in \pos^{(\Theta, \Theta)}(y , \LP_{\Theta}^+)$.~$\qed$  
\end{proof}

\subsubsection{The~$\Theta$-limit set}\label{sect_theta_limit_set} Let~$\Theta \subset \FS$ be a subset of simple restricted roots. A sequence~$(g_k) \in G^\mathbb{N}$ is said to be~\emph{$\Theta$-contracting} if there exist~$(a,b) \in (\Fl(\g, \alpha)) \times (\Fl(\g, \alpha)^\opp)$ such that~$g_k \cdot y \rightarrow a$ uniformly on compact subsets of~$(\Fl(\g, \alpha)) \smallsetminus \hyp_b$. We also say that~$(g_k)$ is \emph{$\Theta$-contracting with respect to~$(a,b)$}. The point~$a$ is then uniquely determined by~$(g_k)$ and is called the~\emph{$\Theta$-limit} of the sequence~$(g_k)$. The following fact is well known: 
\begin{fact}\label{fact_KAK_divergent} Let~$(g_k) \in G^{\mathbb{N}}$.
\begin{enumerate}
    \item Assume that that there exist an open subset~$\mathcal{U} \subset \Fl(\g, \Theta)^+$ and a point~$x \in \Fl(\g, \Theta)^+$ such that~$g_k \cdot \mathcal{U} \rightarrow \{ x \}$ for the Hausdorff topology. Then~$(g_k)$ admits a~$\Theta$-contracting subsequence with~$\Theta$-limit~$x$.
    \item It admits a~$\Theta$-contracting subsequence with respect to~$(x, \xi) \in \Fl(\g, \Theta)^+\times \Fl(\g, \Theta)^\opp$ if and only if the sequence~$(g_k^{-1})$ admits a~$\oppinv(\Theta)$-contracting subsequence with respect to~$(\xi, x)$.
\end{enumerate}
\end{fact}

Given a subgroup~$H$ of~$G$, we denote by~$\Lambda_\Theta(H)$ the set of~$\Theta$-limits of~$\Theta$-contracting sequences of elements of~$H$.

\subsection{Proximal representations}\label{sect_prox_reps} In this section, we recall some fundamental results on (irreducible, proximal) representations of semisimple Lie groups~$G$. Let~$(V, \rho)$ be a finite-dimensional real linear (resp.\ projective) representation of~$G$, i.e.\ a group homomorphism~$G \rightarrow \GL(V)$ (resp.\ $G \rightarrow \PGL(V)$). We will denote by~$\rho_*: \g \rightarrow \operatorname{End}(V)$ the differential of~$\rho$ at~$\operatorname{id}$. 

\subsubsection{Restricted weights}
 
 For any~$\lambda \in \aaa^*$, if the space~$ V^{\lambda} := \{ v \in V \mid \rho_*(h) \cdot v = \lambda (h) v \quad \forall h \in \aaa \}$ is nontrivial, then we say that~$\lambda$ is a \emph{restricted weight} of~$(V, \rho)$. Given~$\alpha, \lambda \in \aaa^*$, one has~$\rho_*(X)\cdot V^{\lambda} \subset V^{\lambda + \alpha}$ for all~$X \in \g_{\alpha}$. For each~$\alpha \in \FS$, the element~$\omega_{\alpha} \in \aaa^*$ introduced in Section~\ref{sect_real_lie_alg} is called the \emph{fundamental weight} associated with~$\alpha$. The cone generated by the simple restricted roots determines a partial ordering on~$\aaa^*$ given by: $\lambda \leq \lambda' \Longleftrightarrow \lambda' - \lambda \in \sum_{\alpha \in \FS} \mathbb{R}_+ \alpha$. If~$(V, \rho)$ is a finite-dimensional real irreducible linear or projective representation of~$G$, then the set of restricted weights of~$(V, \rho)$ admits a unique maximal element for that ordering (see \cite[Cor.\ 3.2.3]{goodman2009symmetry}). This element is called the \emph{highest weight} of~$\rho$, and denoted by~$\chi_{\rho}$.

\subsubsection{Proximality and~$\Theta$-proximal representations}\label{sect_reps_proximales} An element~$g \in\GL(V)$ (resp.\ $\PGL(V)$) is said to be \emph{proximal} in~$\mathbb{P}(V)$ if it has (resp.\ any lift of~$g$ in~$\GL(V)$ has) a unique eigenvalue of maximal modulus and if the corresponding eigenspace is one-dimensional. Given a nonempty subset of the simple restricted roots~$\Theta \subset \FS$, we say that a linear (resp.\ projective), finite-dimensional real irreducible representation~$(V, \rho)$ is \emph{$\Theta$-proximal} if~$\rho(G)$ contains a proximal element and~$\{\alpha \in \FS \mid \langle \chi_\rho, \alpha \rangle >0\} = \Theta$ (see \cite{gueritaud2017anosov}). Note that the last condition is equivalent to saying that~$\chi_\rho \in \sum_{\alpha \in \Theta} \mathbb{N} \omega_\alpha$. If~$(V, \rho)$ is proximal, then
we denote by~$V^{<\chi_\rho}$ the sum of all weight spaces of weights~$\lambda \ne \chi_\rho$ of~$(V, \rho)$.

We will often use irreducible~$\Theta$-proximal representations in this paper. Thus we will use the following terminology:

\begin{definition}\label{def_proximal_triple} Let~$\g$ be a real semisimple Lie algebra and~$\Theta$ be a subset of the simple restricted roots of~$\g$. A triple~$(G, \rho, V)$ is a \emph{linear (resp.\ projective)~$\Theta$-proximal triple} if~$G \in \mathcal{G}_\Theta(\g)$ and if~$(V, \rho)$ is a finite-dimensional, real, irreducible, linear (resp.\ projective),~$\Theta$-proximal representation of~$G$.
\end{definition}

The following fact is a result due to Guéritaud--Guichard--Kassel--Wienhard:

\begin{fact}[{\cite[Prop.\ 3.3]{gueritaud2017anosov}}]\label{prop_ggkw} Let~$(G, \rho, V)$ be a linear (resp.\ projective)~$\Theta$-proximal triple for~$\g$.
    
\begin{enumerate}
    \item The stabilizer of~$V^{\chi_\rho}$ in~$G$ (resp.\ $V^{< \chi_\rho}$) is~$P_{\Theta}^+$ (resp.\ $P_{\Theta}^{\opp}$).
    \item The orbit maps~$g \mapsto \rho(g) \cdot V^{\chi_\rho}$ and~$g \mapsto \rho(g) \cdot V^{<\chi_\rho}$ induce two~$\rho$-equivariant embeddings: 
    \begin{equation*}
        \iota_\rho: \Fl(\g, \Theta) \longrightarrow \mathbb{P}(V)\ \text{ and } \ \iota^{\opp}_\rho: \Fl(\g, \Theta)^{\opp} \longrightarrow \mathbb{P}(V^*)~.
    \end{equation*}
    Two elements~$x \in \Fl(\g, \Theta)$ and~$\xi \in \Fl(\g, \Theta)^{\opp}$ are transverse if and only if their images~$\iota_\rho(x)$ and~$\iota_{\rho}^{\opp}(\xi)$ are.
\end{enumerate}
\end{fact}

The notations of Fact~\ref{prop_ggkw}, in particular the notation~$\iota_\rho, \iota_\rho^\opp$ for the embeddings induced by~$(G, \rho, V)$, will be used during all the paper.

\subsection{Domains in flag manifolds}\label{sect_generalities_proper_domains}

In this section, we let~$\g$ be a real semisimple Lie algebra of noncompact type,~$\Theta \subset \FS$ a subset of the simple restricted roots, and~$G \in \mathcal{G}_{\Theta}(\g)$.

\begin{definition}\label{def_domain_proper}
    Let~$\O \subset \Fl(\g, \Theta)$ be a subset. We say that~$\O$ is a \emph{domain} if~$\O$ is open, nonempty and connected. It is said to be \emph{proper} if there exists~$\xi \in \Fl(\g, \Theta)^\opp$ such that~$\overline{\O}\cap \hyp_{\xi}  
 = \varnothing$. In particular, if~$\xi = \LP_\Theta^\opp$, then we will say that~$\O$ is \emph{proper in~$\Affstdstd$}. This is equivalent to saying that~$\overline{\O}\subset \Affstdstd$.   
\end{definition}

Recall that the hypersurface~$\hyp_\xi$, for~$\xi \in \Fl(\g, \Theta)^\opp$, is defined in~\eqref{eq_hyp_def}, and that the standard affine chart~$\Affstdstd$ is defined in~\eqref{eq_standard_affine_chart}.

Since~$G$ acts transitively on the set of affine charts, we will always be able to assume that a proper domain is proper in~$\Affstdstd$.

\subsubsection{The automorphism group}
Given an open subset~$\O \subset \Fl(\g, \Theta)$, the \emph{automorphism group} of~$\O$ with respect to~$G$ is~$\Aut_G(\O) = \left\{ g \in G \mid g \cdot \O = \O \right\}$. One has:

\begin{fact}\label{fact_autom_group_proper}\cite[Cor.\ 5.3]{zimmer2018proper} The group~$\Aut_G(\O)$ is a Lie subgroup of~$G$. Moreover, it acts properly on~$\O$ as soon as~$\O$ is a proper domain.
\end{fact}

A domain~$\O$ is said to be \emph{quasi-homogeneous}, resp.\ \emph{divisible}, if~$\Aut_G(\O)$ acts cocompactly on~$\O$, resp.\ if there exists a discrete surbgroup~$\Gamma \leq \Aut_G(\O)$ acting cocompactly on~$\O$. It is said to be \emph{symmetric} if for any~$x \in \O$ there exists an order-two element~$s_x \in \Aut_\Theta(\g)$ such that~$s_x \cdot \O = \O$ and~$x$ is the only fixed point of~$s_x$ in~$\O$. None of these definitions depend on the choice of~$G \in \mathcal{G}_\Theta(\g)$.

The \emph{full orbital limit set}~\cite{DGKproj} of~$\O$ is the set~$\Lambda_{\O}^{\operatorname{orb}} := \bigcup_{x \in \O} (\overline{\Aut_G(\O) \cdot x}) \smallsetminus (\Aut(\O) \cdot x)$. If~$\O$ is proper, by Fact~\ref{fact_autom_group_proper}, we have~$\Lambda_{\O}^{\operatorname{orb}} \subset \partial \O$. A proper domain~$\O$ is said to be \emph{almost-homogeneous} if~$\Lambda_{\O}^{\operatorname{orb}} = \partial \O$ (note that this is weaker than quasi-homogeneity). A proper domain~$\O \subset \Fl(\g, \Theta)$ is almost-homogeneous if and only if for all~$a \in \partial \O$, there exist~$x \in \O$ and~$(g_k) \in \Aut_G(\O)^{\mathbb{N}}$ such that~$g_k \cdot x \rightarrow a$. Again, this property does not depend on the choice of~$G \in \mathcal{G}_\Theta(\g)$.

\subsubsection{The dual}\label{sect_the_dual}
The \emph{dual} of a subset~$\O$ is the set
\begin{equation}\label{eq_def_dual}
    \O^*:= \{\xi \in \Fl(\g, \Theta)^\opp \mid \hyp_{\xi} \cap \ \O = \varnothing\} \subset \Fl(\g, \Theta)^\opp~.
\end{equation}
An open subset~$\O \subset \Fl(\g, \Theta)$ is \emph{dually convex} if for all~$a \in \partial \O$, there exists~$\xi \in \O^*$ such that~$a \in \hyp_\xi$ \cite[Def.\ 1.11]{zimmer2018proper}. In other words, the set~$\hyp_\xi$ is a \emph{supporting maximal Schubert subvariety} of~$\O$ at the point~$a$. When~$\Fl(\g, \Theta)$ is real projective space, the dual convexity of a proper domain is equivalent to the classical convexity.

\begin{rmk}\label{rmk_convex_carte_affine} Recall that any affine chart of~$\Fl(\g, \Theta)$ admits an affine structure.
    A naive definition of convexity in flag manifolds would then be given by convexity in an affine chart. However, this definition is not intrinsic: except in real projective space, a domain which is convex in an affine chart may not be convex in another one. See for instance~\cite[Sect.\ 3]{galiay2025convex}.
\end{rmk}

In \cite{zimmer2018proper}, Zimmer proves that quasi-homogeneous domains are dually convex. In \cite{galiay2024rigidity}, using the arguments of \cite[Cor.\ 9.3]{zimmer2018proper}, we prove the following slightly stronger result, that will be useful in this paper:

\begin{fact}\label{prop_zimmer_dual_convex}
Any proper almost-homogeneous domain of~$\Fl(\g, \Theta)$ is dually convex.
\end{fact}

\subsubsection{Caratheodory metrics}\label{sect_carat_metrics}
Important ingredients in the study of domains in flag manifolds are the \emph{Caratheodory metrics}. Let $(G, V, \rho)$ be a $\Theta$-proximal triple for~$\g$ (see Definition~\ref{def_proximal_triple}). Given $x,y \in \Fl(\g, \Theta)$ and~$\xi, \eta \in \Fl(\g, \Theta)^{\opp}$, let
\begin{equation*}
    \nu_x \in \iota_\rho(x) \smallsetminus \{ 0 \}~; \quad \nu_y \in \iota_\rho(y) \smallsetminus \{ 0 \}~; \quad f_{\xi} \in \iota_\rho^\opp(\xi) \smallsetminus \{ 0 \}~; \quad f_{\eta} \in \iota_\rho^\opp(\eta) \smallsetminus \{ 0 \}~.
\end{equation*}
We define the \emph{cross ratio} of $\xi, x, y, \eta$ \emph{relative to} $\rho$ as follows:
\begin{equation}\label{eq_cross_ratio}
    \left[\xi : x : y: \eta \right]_{\rho} :=   \frac{f_{\xi}(\nu_x)f_{\eta}(\nu_y)}{f_{\xi}(\nu_y)f_{\eta}(\nu_x)}~.
\end{equation}
This quantity does not depend on the choice of representatives $\nu_x, \nu_y, f_{\xi}, f_{\eta}$. A.\ Zimmer introduces the following map $C_{\O}^\rho$ associated with the representation~$(V, \rho)$ and a domain $\O \subset \Fl(\g, \alpha)$ \cite{zimmer2018proper}: 
\begin{equation}\label{Cara}
    C_{\O}^\rho: 
    \begin{cases}
    \O \times \O &\longrightarrow \quad \quad \mathbb{R}_+ \\
    (x,y ) &\longmapsto \sup_{\xi, \eta \in \O^*} \log \big{|} \left[\xi : x : y :\eta \right]_{\rho} \big{|}~.
    \end{cases}
\end{equation}

By \cite[Thm 5.2]{zimmer2018proper}, the map $C_{\O}^\rho$ is an $\Aut_{\Aut(\g)}(\O)$-invariant pseudometric, and it is a metric generating the standard topology as soon as $\O$ is a proper domain of $\Fl(\g, \alpha)$. More generally, the proofs of \cite[Thm 5.2 and 9.1]{zimmer2018proper} give:
\begin{fact}\label{fact_Z_dense}
    If~$\O^*$ is Zariski-dense in~$\Fl(\g, \Theta)^\opp$, then~$C_\O^\rho$ is an invariant metric generating the standard topology. If~$\O$ is moreover dually convex, then~$C_\O^\rho$ is a proper metric.
\end{fact}

\section{Nagano spaces}\label{sect_def_nagano_space}

In this section, we define Nagano spaces and recall their well-known properties. We also set up some notation to study them in the rest of the paper. Sections~\ref{sect_graded_Lie_algebra} to~\ref{sect_embedding_concompact_dual} are mostly expository. Section~\ref{sect_charac_nagano} contains characterizations of Nagano spaces, that allow us to better understand Conjecture~\ref{question_lim_Zim}.

\subsection{Algebraic structure}\label{sect_graded_Lie_algebra} Let~$\g$ be a real semisimple Lie algebra with no compact factors and~$\Theta$ be a subset of the simple restricted roots of~$\g$. We say that the pair $(\g, \Theta)$ is a \emph{Nagano pair}, and the flag manifold $\Fl(\g, \Theta)$ is a \emph{Nagano space}, if~$\uu^\opp_\Theta$ is abelian. 

Since~$\g$ is semisimple, there exist $N > 0$, simple Lie subalgebras $\g_1, \dots, \g_N$ of $\g$ and subsets $\Theta_i$ of the simple restricted roots of $G_i$ for $1 \leq i \leq N$ such that $\Theta = \Theta_1 \cup \cdots \cup \Theta_N$ and $\g = \g_1 \oplus \cdots \oplus \g_N$ (up to finite index). We then have a natural identification~$\Fl(\g, \Theta) \simeq \Fl(\g_1, \Theta_1) \times \cdots\times \Fl(\g_N, \Theta_N)$. The pair~$(\g, \Theta)$ is then a Nagano pair if and only if all of the~$(\g_i, \Theta_i)$ are, for~$1 \leq i \leq N$. It motivates the definition: 
\begin{definition}
    A Nagano pair~$(\g, \Theta)$ is \emph{irreducible} if~$\g$ is simple.
\end{definition}
If~$(\g, \Theta)$ is an irreducible Nagano pair, then~$\Theta$ is a singleton~$\{\alpha\}$, and we write~$(\g, \alpha)$ instead of~$(\g, \{\alpha\})$, for simplicity. The longest root~$\alpha_{\FS}$ of the root system~$\Sigma$ of~$\g$ can then be written 
\begin{equation}\label{eq_longest_root}
    \alpha_{\FS} = \alpha + \sum_{\beta \in \FS \smallsetminus \{\alpha\}} n_\beta \beta~,
\end{equation} 
where~$n_\beta \in \mathbb{N}$ (see e.g.\ \cite{takeuchi1988basic}).

Since any Nagano space is a product of irreducible Nagano spaces, geometric considerations on Nagano spaces can be reduced to the irreducible case, which we will do in this paper. The list of irreducible Nagano pairs, established by Nagano \cite{nagano1965transformation}, is given in Table~\ref{table_nagano_full}. 

\begin{notation}\label{notation_nagano}
    Given an irreducible Nagano pair~$(\g, \alpha)$, we will use the following simplified notations:
\begin{equation*}
         \ L=  L_{\{\alpha\}}~, \ \ \uu^{\pm} = \uu_{\{\alpha\}}^{\pm}~,  \ \ U^{\pm} = U_{\{\alpha\}}^{\pm}~, \ \ \LP^+ = \LP_{\{\alpha\}}^\pm~, \ \ P^\pm= P_{\{\alpha\}}^\pm~. 
\end{equation*}
It will be convenient to fix~$v^+ \in \g_{\alpha}, \ v^\opp \in \g_{-\alpha}$ such that~$\tr_{\mathsf{std}}: = (v^+, h_{\alpha},v^\opp)$ is an~$\sll_2$-triple (where~$h_\alpha$ is defined in Section~\ref{sect_sl2-triples}).
\end{notation}

The main consequence of the fact that~$\uu^\opp$ is abelian is the following:

\begin{fact}\label{fact_kostant}
  \cite{kostant2010root} The identity component~$L^{\circ}$ of~$L$ acts irreducibly on~$\uu^\pm$.
\end{fact}

It is known that there exists a Cartan involution~$\sigma_0$ of~$\g$ (as defined in Section~\ref{sect_cartan_decomp} for some maximal compact subalgebra of~$\g$) such that~$\sigma_0 \g_k  =\g_{-k}$ for~$k \in \{-1, 0, 1\}$, and an element~$H_0 \in \operatorname{Centr}(\g_0)$ such that~$\g_k$ is the eigenspace of~$\ad(H_0)$ for the eigenvalue~$k$. See for instance \cite{kobayashi1964filtered}.

\subsection{Dilations}\label{sect_dilations_translations} Recall the standard affine chart~$\Affstdstd$ and its parametrization~$\varphistd$ defined in~\eqref{eq_standard_affine_chart} and~\eqref{eq_A_egal_exp}. Let~$H_0$ be the element introduced in the previous section. If~$G \in \mathcal{G}_{\{\alpha_r\}}(\g)$, then for all~$t \in \Rf_{>0}$ we define~$g_0(t) = \exp \left(-\log\left(\sqrt{t}\right)H_0\right) \in L$. The element~$\Ad(g_0(t))$ acts on~$\uu^{\pm}$ by
\begin{equation}\label{eq_def_lzero}
    \Ad(g_0(t))X =
    \begin{cases}
         t X \quad &\forall X \in \uu^-~; \\
    \frac{1}{t} X \quad &\forall X \in \uu^+~.
    \end{cases}
\end{equation}
Hence any positive dilation of the standard affine chart~$\Affstdstd$ at~$\LP^+ = \varphistd(0)$ can be realized as the restriction to~$\Affstdstd$ of a map of the form~$x \mapsto g_0(t) \cdot x$ of~$\Fl(\g, \alpha)$ for some~$t \in \Rf_{>0}$. On the other hand, since~$\uu^-$ is abelian, any translation in~$\Affstdstd$ is realized as left multiplication by an element of~$U^- \leq G$. Hence any affine dilation~$d$ with center a point~$x_0 \in \Affstdstd$ can be realized as the restriction to~$\Affstdstd$ of the action of an element of~$P^+$.

\subsection{The symmetric structure} \label{sect_construction_symetri_structure} We fix~$(\g, \alpha)$ an irreducible Nagano pair. Let~$\g = \kk \oplus \mathfrak{h}$ be the Cartan decomposition associated with~$\sigma_0$ and~$K$ the associated maximal compact subgroup of~$G$ (recall Section~\ref{sect_cartan_decomp}). The group~$K$ acts transitively on~$\Fl(\g, \alpha)$; denote by~$K_0$ the stabilizer of~$\LP^+$ in~$K$, and~$\mathfrak{k}_0$ its Lie algebra. We have the natural identification~$\Fl(\g, \alpha) \simeq K/ K_0$. Thus there exists a natural~$K$-invariant metric~$g_{\g,\alpha}$ on~$\Fl(\g, \alpha)$. The element~$\exp(i \pi H_0) \in K_0$ then acts as a symmetry on~$T_{\LP^+}\Fl(\g, \alpha)$, making~$(\Fl(\g, \alpha), g_{\g,\alpha})$ a compact Riemannian symmetric space. 
\begin{definition}\label{def_rank_nagano_space}
    The \emph{rank}~$\operatorname{rk}(\g, \alpha)$ of an irreducible Nagano pair~$(\g, \alpha)$ (resp.\ Nagano space $\Fl(\g, \alpha)$) is the rank of the compact symmetric space~$(\Fl(\g, \alpha), g_{\g,\alpha})$. It does not depend on the choices above.
\end{definition}

The ranks of all irreducible Nagano pairs are given in Table~\ref{table_nagano_full}. Note that the \emph{rank} of a (not necessarily irreducible) Nagano space is then simply defined as the sum of the ranks of its irreducible factors.

\begin{rmk}
     Conversely, Nagano has characterized irreducible compact symmetric spaces that are Nagano spaces \cite{nagano1965transformation}. Those are the compact symmetric spaces~$(M, g)$ admitting a \emph{noncompact transformation group~$G$}. In this case, the stabilizer of a point~$x \in M$ is a parabolic subgroup~$P$ of~$G$, so that~$G/P$ identifies~$G$-equivariantly with~$M$.
\end{rmk}

\subsection{Embedding the noncompact dual}\label{sect_embedding_concompact_dual}

Let~$(\g, \alpha)$ be an irreducible Nagano pair of rank~$r$. One can choose a maximal system of \emph{strongly orthogonal} roots~$\beta_1, \dots, \beta_r \in \Sigma_{\{\alpha\}}^+$ with the same length for the Killing form, such that~$\beta_1$ is the longest root of~$\Sigma$ \cite{takeuchi1988basic}. Here \emph{strongly orthogonal} means that~$\beta_i \pm \beta_j \notin \Sigma$, for all~$i \ne j$.

We set~$s_0 = \id$ and~$s_k := s_{\beta_1} \cdots s_{\beta_k} \in W$ for~$1 \leq k \leq s$. The following is proven in \cite{takeuchi1988basic}:

\begin{fact}[\cite{takeuchi1988basic}]\label{lem_cardinal_coset} The set of symmetries~$\{s_1, \dots, s_s\}$ is a complete set of representatives of~$W_{\FS \smallsetminus \{\alpha\}} \backslash W / W_{\FS \smallsetminus \{\alpha\}}$, and for all~$1 \leq k <  k' \leq s$, one has~$[s_k] \leq [s_{k'}]$.
\end{fact}

The map $\sigma := \Ad(\exp(\pi i H_0))$ is an involutive automorphism of~$\g$, which commutes with~$\sigma_0$ and is equal to~$\id$ on~$\mathfrak{l}$ and~$- \id$ on~$\mathfrak{m} := \uu^+ \oplus \uu^\opp$. Since~$\sigma_0$ and~$\sigma$ commute, the Lie algebra~$\g$ can be decomposed into four spaces:~$\g = \mathfrak{k}_0 \oplus \mathfrak{m}_{\mathfrak{k}} \oplus \mathfrak{h}_0 \oplus \mathfrak{m}_{\mathfrak{h}}$, where~$\mathfrak{k}_0 = \mathfrak{l} \cap \mathfrak{k}$,~$\mathfrak{h}_0 = \mathfrak{l} \cap \mathfrak{h}$,~$\mathfrak{m}_{\mathfrak{k}} = \mathfrak{m} \cap \mathfrak{k}$ and~$\mathfrak{m}_{\mathfrak{h}} = \mathfrak{m} \cap \mathfrak{h}$. Let~$E_i \in \g_{\beta_i}$ be such that~$(E_i, h_{\beta_i}, -\sigma_0(E_i))$ is an~$\mathfrak{sl}_2$-triple. Recall that the notation~$h_\beta$ for~$\beta \in \Sigma$ is defined in Section~\ref{sect_real_lie_alg}. Then~$h_i \in \mathfrak{a}$. Now let~$m_i := E_i - \sigma_0(E_i) \in \mathfrak{m}_{\mathfrak{h}}$ and~$ \mathfrak{c} := \sum_{i=1}^s \mathbb{R} m_i$. The following is well known:

\begin{fact}\label{fact_construction_noncompact_dual}
    \begin{enumerate}
        \item \cite{kaneyuki1987orbit} The space~$\mathfrak{c}$ is a maximal abelian subspace in~$\mathfrak{m}_{\mathfrak{h}}$;
        \item \cite{nagano1965transformation, takeuchi1965cell} Let~$\g^* := \mathfrak{k}_0 \oplus \mathfrak{m}_{\mathfrak{h}}\subset \g$. Then~$\g^*$ is a subalgebra of~$\g$, and the triple~$(\g^*, \mathfrak{k}_0, \sigma_0)$ is the noncompact dual of the symmetric triple~$(\mathfrak{k}, \mathfrak{k}_0, \Cartinv)$.
    \end{enumerate}
\end{fact}
By Fact~\ref{fact_construction_noncompact_dual}.(2), the symmetric space defined by the symmetric triple~$(\g^*, \mathfrak{k}_0, \sigma_0)$ is uniquely defined by the pair~$(\g, \alpha)$, and called the \emph{noncompact dual of~$\Fl(\g, \alpha)$}. We will denote it by~$\mathbb{X}(\g, \alpha)$. By construction, it has same real rank as~$\Fl(\g, \alpha)$.

In the setting of Fact~\ref{fact_construction_noncompact_dual}, there exist~$G \in \mathcal{G}_{\{\alpha\}}(\g)$ such that the connected subgroup~$G^*$ of~$G$ with Lie algebra~$\g^*$ identifies, up to finite index, with the identity component of the isometry group of~$\mathbb{X}(\g, \alpha)$. The groups~$G^*$ (up to finite index) associated with all irreducible Nagano pairs~$(\g, \alpha)$ are given in Table~\ref{table_nagano_full}.

Let~$\O = G^* \cdot \LP^+$, and let~$x_0 \in \mathbb{X}(\g, \alpha) \simeq G^* / K_0$ be the class of~$K_0$. By \cite[Thm 6.2]{nagano1965transformation}, there are two~$G^*$-equivariant diffeomorphisms
\begin{equation}\label{eq_def_ncd}
\begin{split}
    \ncd:& ~\mathbb{X}(\g, \alpha) \longrightarrow \O~; \ g \cdot x_0 \longmapsto g \cdot \LP^+ ~; \\
    \ncd^\opp:& ~\mathbb{X}(\g, \alpha) \longrightarrow \operatorname{int}(\O^*)~; \ g \cdot x_0 \longmapsto g \cdot \LP^\opp ~, 
\end{split}
\end{equation}
and both~$\O$ and~$\O^*$ are domains of~$\Fl(\g, \alpha)$.
In particular, the set~$\ncd(\mathbb{X}(\g, \alpha))$ is proper in~$\Affstdstd$.

\begin{definition}
    We will say that a domain~$\O \subset \Fl(\g, \alpha)$ is a \emph{realization of~$\mathbb{X}(\g, \alpha)$} if there exists~$g \in G$ such that~$\O = g \cdot \ncd(\mathbb{X}(\g, \alpha))$.
\end{definition}

By construction, any realization of the noncompact dual of~$\Fl(\g, \alpha)$ is a symmetric domain of~$\Fl(\g, \alpha)$.

An example to have in mind is when~$(\g, \alpha)$ corresponds to item~(iv,~$p=1$) of Table~\ref{table_nagano_full}. In this case~$\Fl(\g,\alpha) = \mathbb{P}(\mathbb{R}^{q+1})$, and~$\mathbb{X}(\g, \alpha)$ is the real hyperbolic space~$\mathbb{H}^q$. Any realization of~$\mathbb{X}(\g, \alpha)$ is a translate of the \emph{Klein model} of hyperbolic space.

\begin{rmk}\label{rmk_realizations}\begin{enumerate}
    \item The embedding of the noncompact dual generalizes the well-known Harish-Chandra embedding of Hermitian symmetric spaces of noncompact type as bounded domains of complex vector spaces.
    \item Since realizations of~$\mathbb{X}(\g, \alpha)$ are proper, according to Section~\ref{sect_dilations_translations} they form a neighborhood basis of~$\Fl(\g, \alpha)$.
\end{enumerate}
\end{rmk}

\subsection{Example: the Grassmannians}\label{sect_grass_example} They correspond to item~(iv) of Table~\ref{table_nagano_full} and will be our lead examples in this paper. For~$\g = \sll(p+q, \mathbb{R}) = \operatorname{Mat}_{p+q}(\mathbb{R})$, we fix the Cartan subspace
$$\aaa := \Big{\{} \operatorname{diag}(\lambda_1, \dots, \lambda_{p+q}) \mid \lambda_i \in \mathbb{R} \quad \forall 1 \leq i \leq p+q, \ \sum_{i=1}^{p+q} \lambda_i = 0\Big{\}}~.$$

We denote by~$\varepsilon_i$ the map~$\varepsilon_i: \operatorname{diag}(\lambda_1, \dots, \lambda_{p+q}) \mapsto \lambda_i$, for~$1 \leq i \leq p+q$. Note that those~$\varepsilon_i$ do not coincide with those defined in Section~\ref{sect_embedding_concompact_dual}.

The root system of~$\sll(p+q, \mathbb{R})$ associated to~$\aaa$ is~$\Sigma = \{ \pm(\varepsilon_i - \varepsilon_{j}) \mid 1 \leq i < j \leq p+q\}$. A fundamental system is then 
\begin{equation*}
    \Delta := \{\alpha_i := \varepsilon_i - \varepsilon_{i+1} \mid 1 \leq i \leq p+q-1\}~.
\end{equation*}
The flag manifold~$\Fl(\g, \alpha_{p})$ identifies with the space~$\Grass$ of~$p$-planes of~$\mathbb{R}^{p+q}$, called the \emph{Grassmannian of~$p$-planes of~$\mathbb{R}^{p+q}$}. The opposite flag manifold identifies with the Grassmannian~$\Fl(\g, \oppinv(\alpha_p)) = \Fl(\g, \alpha_q) = \Grassq$ of~$q$-planes of~$\mathbb{R}^{p+q}$.

As written in Table~\ref{table_nagano_full}, the rank (as an irreducible Nagano pair) of~$(\sll(p+q, \mathbb{R}), \alpha_p)$ is~$\min(p,q)$. Hence the rank-one case corresponds to real projective space. 

The noncompact dual of~$\Grass$ is the symmetric space of~$\SO(p,q)$. If~$\varphi$ is a quadratic form on~$\mathbb{R}^{p+q}$ of signature~$(p,q)$, then the set of positive definite~$p$-planes of~$\mathbb{R}^{p+q}$ for~$\varphi$ is a realization of~$\mathbb{X}(\sll(p+q, \mathbb{R}), \alpha_p)$ (see e.g.\ \cite{van2019rigidity}). When~$p =1$, we recover the Klein model of real hyperbolic space.

Note that this construction has complex and quaternionic analogues, where one just replaces~$\mathbb{R}$ by~$\mathbb{C}$ or~$\mathbb{H}$. It also gives Nagano spaces (items~(v) and~(vi) of Table~\ref{table_nagano_full}), but they will not be \emph{of real type}, in the terminology of next Section~\ref{sect_HS_photons_defs}.

\subsection{Characterizations of Nagano spaces}\label{sect_charac_nagano}
This section is independent of the rest of the paper. We give elementary characterizations of Nagano spaces among flag manifolds, and of real projective space among Nagano spaces, coming from the classification of Nagano spaces \cite{nagano1965transformation} and of their symmetric domains \cite{makarevivc1973open}.

Nagano has proven that the noncompact dual~$\mathbb{X}(\g, \alpha)$ of an irreducible Nagano space~$\Fl(\g, \alpha)$ embeds as a proper symmetric domain of~$\Fl(\g, \alpha)$. However, the realizations of~$\mathbb{X}(\g, \alpha)$ into~$\Fl(\g, \alpha)$ are not necessarily the only symmetric domains in~$\Fl(\g, \alpha)$. Makarevic \cite{makarevivc1973open} has listed all possible symmetric domains of Nagano spaces~$\Fl(\g, \alpha)$ that have a reductive transitive automorphism group. In general, there is at least one other (strict) domain than realizations of~$\mathbb{X}(\g, \alpha)$ in the list. However, if one asks for properness one actually has:

\begin{lem}\label{lem_domaines_propres_espaces_de_nagano} Let~$(\g, \alpha)$ be an irreducible Nagano pair and assume that there exists a proper symmetric domain~$\O \subset N$ with transitive and reductive automorphism group, such that~$\O$ is not a realization of~$\mathbb{X}(\g, \alpha)$. Then there exists~$n \geq 3$ such that~$(\g, \alpha) = (\sll(n, \mathbb{R}), \alpha_1)$ or~$(\sll(n, \mathbb{R}), \alpha_{n-1})$, i.e.\ $\Fl(\g, \alpha)$ is either the real projective space of dimension~$n-1$ or its dual.
\end{lem}

\begin{proof} By Fact~\ref{fact_autom_group_proper}, the stabilizer of a point~$x \in \O$ is a compact Lie subgroup of~$\Aut_G(\O)$. But in the list of  \cite{makarevivc1973open}, whenever~$(\g, \alpha) \notin \{ (\sll(n, \mathbb{R}), \alpha_1), (\sll(n, \mathbb{R}), \alpha_{n-1})\mid n \in \mathbb{N}_{\geq 3}\}$, the only cases where the stabilizer of a point is compact is when~$\O$ is a realization of~$\mathbb{X}(\g, \alpha)$.~$\qed$
\end{proof}

Nagano's theorem and Makarevic's list tell us that Nagano spaces contain many symmetric domains. Conversely, we can now prove Lemma~\ref{lem_symmetric_nagano}:

\begin{proof}[Proof of Lemma~\ref{lem_symmetric_nagano}] Let~$\g$ be the Lie algebra of~$G$, and~$\Theta$ be the subset of simple restricted roots of~$\g$ such that~$P = P_\Theta^+$. We have~$G/P = \Fl(\g, \Theta)$.

Indeed, let~$x \in \O$ and let~$s_x \in \Aut(\O)$ be a symmetry. Since~$s_x$ has finite order, up to translating~$\O$ by an element of~$G$, we may assume that~$s_x \in K$. Since~$x$ is the only fixed point of~$s_x$ in~$\O$, we have~$d_x s_x = - \id_{T_x \O}$, and since~$\O$ is open in~$\Fl(\g, \Theta)$, one has~$\id_{T_x \O} = \id_{T_x \Fl(\g, \Theta)}$, so~$d_x s_x = - \id_{T_x \Fl(\g, \Theta)}$.

Since~$s_x$ fixes~$x$, it preserves the affine chart~$\Affstd_x$ of~$\Fl(\g, \oppinv(\Theta))$. Since it has finite order, it has a fixed point~$y \in \Affstd_x$. Then~$x$ and~$y$ are transverse, we may thus assume that~$(x,y) = (\LP^+_\Theta, \LP_\Theta^\opp)$. Hence~$s_x \in L_\Theta$. But then~$T_x s_x = -\id_{T_x \Fl(\g, \Theta)}$ implies that~$\Ad(s_x)\vert_{\uu_\Theta^\opp}= - \id$. We then have, for all~$X, Y \in \uu_\Theta^\opp$:
\begin{equation*}
    [X,Y] = [-X, -Y] = [\Ad(s_x) X, \Ad(s_x)Y] = \Ad(s_x)[X,Y] = -[X, Y]~.
\end{equation*}
Hence~$[X,Y] = 0$. This proves that~$\uu_\Theta^\opp$ is abelian, hence~$(\g, \Theta)$ is a Nagano pair.~$\qed$
\end{proof}

\section{Photons}\label{sect_HS_photons_defs}
The content of this section mainly consists of reminders on photons studied in \cite{galiay2024rigidity, beyrer2024positivity}. In that paper, photons are defined and studied only in causal flag manifolds. However, the results generalize verbatim. Moreover, these results are also proved in \cite{galiay2025convex} for real-type Nagano spaces and in \cite{beyrer2024positivity} for all flag manifolds defined by restricted simple roots of multiplicity~$1$. Most of the results stated in this section are given without proof, with a precise reference provided each time.

A distinguished class of Nagano spaces, which we shall study extensively in this paper, is that of \emph{real-type} Nagano spaces:

\begin{definition}
    An irreducible Nagano pair~$(\g, \alpha)$ is said to be \emph{of real type} if~$\dim(\g_{\alpha}) = 1$. In that case we also say that the Nagano space~$\Fl(\g, \alpha)$ is of real type. 
\end{definition}
See Table~\ref{table_nagano_full} for the list of real-type irreducible Nagano pairs. Those are given by the items for which the column ``$\dim(\g_\alpha)$'' contains a~$1$.

For the rest of this section, we fix an irreducible real-type Nagano pair~$(\g, \alpha)$ and a group~$G \in \mathcal{G}_{\{\alpha\}}(\g)$.

\subsection{Embedding the projective line into~$\Fl(\g, \alpha)$}\label{sect_plong_proj_line}
 In this section, we construct embeddings of the projective line into~$\Fl(\g, \alpha)$. The images of these embeddings are what are defined as \emph{photons} in Section~\ref{sect_photons} below.

Let us consider the~$\sll_2$-triple~$\tr_{\mathsf{std}}$ defined in Notation~\ref{notation_nagano}. As explained in Section~\ref{sect_sl2-triples}, it induces a Lie algebra embedding~$\plongsl_{\tr_{\mathsf{std}}}: \sll(2, \mathbb{R}) \rightarrow \g$, which itself induces a group homomorphism~$\plong: \SL(2, \mathbb{R}) \hookrightarrow G$ with kernel contained in~$\{\pm \operatorname{id}\}$ and with  differential~$\plong_* = \plongsl_{\tr_{\mathsf{std}}}$ at~$\operatorname{id}$.

\begin{lem}\label{lem_stab_parap}\cite[Lem.\ 3.7]{beyrer2024positivity}
    The stabilizer of~$P$ in~$\SL(2, \mathbb{R})$ is the standard Borel subgroup~$P_1$ of~$\SL(2, \mathbb{R})$.
\end{lem}

By Lemma~\ref{lem_stab_parap} above, the group homomorphism~$\plong$ induces a~$\plong$-equivariant embedding {$\isl_\std:\mathbb{P}(\mathbb{R}^2) \hookrightarrow \Fl(\g, \alpha)$}. It will be convenient to write this map explicitly:
\begin{equation}\label{eq_param_photons}
    \isl_\std ([1:t]) =  \exp(tv^-) \cdot \LP^+\quad \forall t \in \mathbb{R}.
\end{equation}

\subsection{Photons}\label{sect_photons}
In this section, we define \emph{photons} in Nagano spaces and investigate their properties.

Let~$(\g, \alpha)$ be an irreducible Nagano pair. Recall the map~$\isl_\std$ defined before Equation~\eqref{eq_param_photons}. We define the topological circle
\begin{equation}\label{eq_def_photon_standard}
    \Phot_\std := \isl_\std(\mathbb{P}(\mathbb{R}^2))~,
\end{equation}
called the \emph{standard photon} of~$\Fl(\g, \alpha)$. The map~$\isl_\std$ is then a parametrization of~$\Phot_{\mathsf{std}}$.
\begin{definition}\label{def_photon}
    A \emph{photon} of~$\Fl(\g, \alpha)$ is an~$\Aut_\Theta(\g)$-translate of~$\Phot_\std$ in~$\Fl(\g, \alpha)$.
\end{definition}

Recall that~$\Aut_\Theta(\g)$ is defined in Section~\ref{sect_aut_lie_alg}.

For the Nagano pair~$(\sll(p+q, \mathbb{R}), \alpha_p)$ (item~(iv) of Table~\ref{table_nagano_full}), which by Section~\ref{sect_grass_example} give the Grassmannnians, the photons can be described explicitly. Two~$p$-planes~$V, W \in \Grass$ are on a same photon if and only if~$\dim(V \cap W) \geq p-1$. Given a~$(p-1)$-plane~$V_0$ and a~$(p+1)$-plane~$V_1$, the set
\begin{equation*}
    \{V \in \Grass \mid V_0 \subset V \subset V_1\}
\end{equation*}
defines a photon. Reciprocally, any photon comes from this kind of construction. Note that, in particular, for~$(\g, \alpha) = (\sll(n, \mathbb{R}), \alpha_1)$, photons are simply the projective lines of~$\mathbb{P}(\mathbb{R}^n)$.

See~\cite[Sect.\ 6.3.4]{galiay2025convex} for a description of photons in other concrete examples (Einstein Universe, causal flag manifolds). 

\begin{rmk}
    Photons have been defined and investigated for general flag manifolds defined by roots of multiplicity~$1$ in \cite{beyrer2024positivity}, to prove that the property of being~$\Theta$-positive (in the sense of \cite{labourie2021positivity}) for representations of a surface groups is closed. 
\end{rmk}

\begin{lem}\label{fact_stab_u_std}(\cite[Prop.\ 3.4, Prop.\ 3.6, 3.9 and Lem.\ 3.11]{beyrer2024positivity}, see also \cite{galiay2024rigidity} for the case of causal flag manifolds)
Let~$(\g, \alpha)$ be an irreducible real-type Nagano pair. Then
\begin{enumerate}
    \item~$U^+ \cdot \Phot_{\mathsf{std}} = \Phot_{\mathsf{std}}$;
    \item~$L$ acts transitively on the set of photons through~$\LP^+$;
    \item Let~$(\g, \alpha)$ be an irreducible real-type Nagano pair. For all~$x,y \in \Fl(\g, \alpha)$ there exists at most one photon through~$x$ and~$y$.
    \item One has 
\begin{equation*}
    \operatorname{Stab}_G (\Phot_{\mathsf{std}}) = \plong(\SL(2, \mathbb{R})) \times \operatorname{Cent}_G \left(\plong\left(\SL(2, \mathbb{R})\right)\right),
\end{equation*}
where~$\operatorname{Stab}_G (\Phot_{\mathsf{std}})$ is the stabilizer of~$\Phot_{\mathsf{std}}$ in~$G$ and~$\operatorname{Cent}_G \left(\plong\left(\SL(2, \mathbb{R})\right)\right)$ is the centralizer in~$G$ of the group~$\plong(\SL(2, \mathbb{R}))$, acting trivially on~$\Phot_\std$.

\item The Hausdorff limit in~$\Fl(\g, \alpha)$ of a sequence of photons is still a photon.
\end{enumerate}
\end{lem}

Given a photon~$\Phot$, there exists~$g \in G$ such that~$\Phot = g \cdot \Phot_{\mathsf{std}}$. Since the map~$\isl_\std$ defined before Equation~\eqref{eq_param_photons} is a parametrization of~$\Phot_\std$, the map 
\begin{equation}\label{eq_param_photons_3}
 \begin{array}{cccc}
   \isl_g : &
        \mathbb{P}(\mathbb{R}^2) &\longrightarrow &\Fl(\g, \alpha) \\
        &x &\longmapsto &g \cdot \isl_\std(x)
 \end{array}
\end{equation}
is then a parametrization of~$\Phot$ (recall that~$\isl$ is defined before~\eqref{eq_param_photons}). A priori, this parametrization depends on the choice of~$g \in G$ such that~$\Phot = g \cdot \Phot_{\mathsf{std}}$. By Lemma~\ref{fact_stab_u_std}.(4), two parametrizations given by different choices of~$g \in G$ such that~$\Phot = g \cdot \Phot_{\mathsf{std}}$ only differ by a projective reparametrization of~$\Phot \simeq \mathbb{P}(\mathbb{R}^2)$. This motivates the following definition:

\begin{definition}\label{def_param_photons}
    Let~$\Phot$ be a photon of~$\Fl(\g, \alpha)$. A \emph{projective parametrization} of~$\Phot$ is a map~$\zeta: \mathbb{P}(\mathbb{R}^2) \rightarrow \Phot$ as defined in~\eqref{eq_param_photons_3}.
\end{definition}

Note that the set of projective parametrizations of a photon~$\Phot$ does not depend on the choice of a group~$G \in \mathcal{G}_{\{\alpha\}}(\g)$.

The following fact states that photons intersecting~$\Affstdstd$ are compactifications of certain affine lines of the standard affine chart~$\Affstdstd$. Using Lemma~\ref{fact_stab_u_std}, the proof of Fact~\ref{lem_photons_affine} is the same as the one of~\cite[Lem. 5.4]{galiay2024rigidity}:

\begin{fact}\label{lem_photons_affine} Let~$(\g, \alpha)$ be an irreducible real-type Nagano pair.
   Let~$\Phot$ be a photon. If~$\Phot \cap \Affstdstd$ is nonempty, then it is an affine line in~$\Affstdstd$, and~$\Phot \cap \hyp_{\LP^\opp}$ is a singleton.
\end{fact}

Fact~\ref{lem_photons_affine} implies that for all~$\xi \in \Fl(\g, \alpha)^\opp$ such that~$\Phot_{\mathsf{std}} \not\subset \hyp_{\xi}$, the set~$\Phot_{\mathsf{std}} \cap \hyp_{\xi}$ is a singleton.

Let us add that Fact~\ref{lem_photons_affine} is proven in \cite{van2019rigidity} for Grassmannians (item~(iv) of Table~\ref{table_nagano_full}). In that case~$\Affstdstd$ can be identified with the space~$\operatorname{Mat}_{q,p}(\mathbb{R})$ of~$(q \times p)$-matrices; the nonempty intersection of a photon with it is then of the form~$X + \mathbb{R} S$, where~$S$ is a rank-one matrix and~$X \in \operatorname{Mat}_{q,p}(\mathbb{R})$.

To end this section, note that Fact~\ref{fact_kostant} directly gives the following fundamental fact:

\begin{fact} \label{lem_L_v_generates}
    The set~$\Ad(L^{\circ}) \cdot v^\opp = \{\Ad(g) \cdot v^\opp \mid g \in L^{\circ} \}$ generates~$\uu^\opp$ as a vector space.
\end{fact}

\subsection{Arithmetic distance}\label{sect_arithm_distance}
The \emph{arithmetic distance}~$d_H(x,y)$ between two distinct points~$x,y \in \Fl(\g, \alpha)$ is the minimal integer such that there exist~$x_0 := x, x_1, \dots, x_k :=y \in \Fl(\g,\alpha)$ such that~$x_i$ and~$x_{i+1}$ are on the same photon. If~$x=y$, we set~$d_H(x,y) =0$. 

The following fact is a direct consequence of a theorem of Takeuchi \cite[Thm 6.4]{takeuchi1988basic}. In the statement, we use the notation introduced in Section~\ref{sect_embedding_concompact_dual}. Also recall that the element~$w_0$ is the \emph{longest element} of the restricted Weyl group of~$G$ defined in Section~\ref{sect_restricted_weyl_grouup}. 

\begin{fact}[\cite{takeuchi1988basic}]\label{thm_takeuchi_hyp} 
For all~$1 \leq k \leq \operatorname{rk}(\g, \alpha)$, one has~$\{x \in \Fl(\g, \alpha) \mid d_H (\LP^+, x) = k \} = \mathsf{C}_{\overline{s_k}}(\LP^+)$. Hence for all~$x_0 \in \Fl(\g, \alpha)$, one has 
$$\{x \in \Fl(\g, \alpha) \mid d_H (x_0, x) \leq \operatorname{rk}(\g, \alpha)-1 \} \subset \Fl(\g, \alpha) \smallsetminus \mathsf{C}_{\overline{w_0}}(x_0)~.$$ 
In particular, if~$(\g, \alpha)$ has higher rank, then any photon through~$x_0$ is contained in~$\Fl(\g, \alpha) \smallsetminus \mathsf{C}_{\overline{w_0}}(x_0)$.
\end{fact}

\section[Plücker triples]{Plücker triples}\label{sect_type_I}

The purpose of this section is to show that, if~$(\g, \alpha)$ is an irreducible real-type Nagano pair, then the image of photons under a suitably chosen embedding of~$\Fl(\g, \alpha)$ in a projective space are projective lines.

\subsection{Embeddings of real-type Nagano spaces} Recall that the~$\Theta$-proximal triples have been introduced in Definition~\ref{def_proximal_triple}. In this section, we prove the following lemma, concerning the image of the element~$v^\opp$ introduced in Notation~\ref{notation_nagano}, under the representation given by a proximal triple:

\begin{lem}\label{lem_degree_rep_1} 
Let~$(\g, \alpha)$ be an irreducible real-type Nagano pair. Let~$(G, \rho, V)$ be a linear or projective~$\{ \alpha\}$-proximal triple for~$\g$, with highest weight~$ \chi := N \omega_\alpha$ for some~$N \in \mathbb{N}$. Let~$\vv_0 \in V^{\chi} \smallsetminus \{ 0 \}$. Then~$\rho_*(v^-)^k \cdot \vv_0 \ne 0$ for all~$k \leq N$, and~$\rho_*(v^-)^k \cdot \vv_0 = 0$ for all~$k \geq N+1$.
\end{lem}

\begin{proof} Recall that~$(\omega_\beta)_{\beta \in \FS}$ is the basis of~$\aaa^*$, dual to~$(h_{\beta})_{\beta \in \FS}$ (Section~\ref{sect_real_lie_alg}). Also recall the weight spaces introduced in Section~\ref{sect_reps_proximales}. It is a general fact (see for instance~\cite{knapp1996lie}) that
\begin{equation*}
    \rho_*(\g_{\beta}) \cdot V^\lambda \subset V^{\lambda + \beta}
\end{equation*}
for any restricted root~$\beta \in \Sigma$ and and weight~$\lambda $ of~$(V, \rho)$. Hence one has~$\rho_*(v^-)^k \cdot \vv_0 \in V^{N\omega - k\alpha}$ for all~$k \in \mathbb{N}$.

On the other hand, one has~$\rho_*(v^-)^k \cdot \vv_0 \ne 0$ for all~$0 \leq k \leq N$ (see e.g.\ \cite[Lem.\ 3.2.9]{goodman2009symmetry})
(see e.g.\ \cite[Lem.\ 3.2.9]{goodman2009symmetry}). Thus it suffices to prove that~$V^{N\omega_\alpha - k \alpha} = \{ 0 \}$ for~$k \geq N+1$. Since~$\dim(\g_{\alpha}) = 1$, this is satisfied whenever~$N\omega_\alpha - k \alpha \notin \operatorname{Conv} (W \cdot (N \omega_\alpha))$ (see e.g.\ \cite[Prop.\ 3.2.10]{goodman2009symmetry})), where~$W$ is the restricted Weyl group of~$\g$ defined in Section~\ref{sect_restricted_weyl_grouup}. Let us check this property. In particular, it suffices to prove that~$\omega_\alpha - \mu \alpha \notin \operatorname{Conv} (W \cdot \omega_\alpha)$ for all~$\mu > 1$.

Recall that we denote by~$B$ the Killing form on~$\g$. Its restriction to~$\aaa$ induces a~$W$-invariant inner product. For any root~$\beta \in \Sigma$, let~$h_\beta' \in \aaa$ be such that~$\beta = B(\cdot , h_\beta')$ on~$\aaa$. Then the element~$h_\beta$ defined in Section~\ref{sect_real_lie_alg} is just~$h_\beta = \frac{2 h_\beta'}{B(h_\beta' , h_\beta')}$. By~$W$-invariance of~$B$, we have 
\begin{equation}\label{eq_inv_h}
    w \cdot h_\beta = h_{w^{-1} \cdot \beta}
\end{equation}
for all~$w \in W$. By~\eqref{eq_longest_root}, for all~$\beta \in \Sigma$ we have~$\beta = \sum_{\beta' \in \FS \smallsetminus \{ \alpha\}} n_{\beta'} \beta' + \delta \alpha$, with~$n_{\beta'} \in \mathbb{N}$  and~$\delta \in \{-1, 0, 1\}$. Thus a direct computation gives~$$h_\beta = X +  \delta \frac{B(h_\alpha', h_\alpha')}{B(h_\beta', h_\beta')} h_\alpha~,$$
where~$X  \in \sum_{\beta' \in \FS \smallsetminus \{ \alpha \}} \mathbb{R} h_{\beta'}$. Thus by definition of~$\omega_\alpha$ (see Section~\ref{sect_real_lie_alg}), we have~$\omega_\alpha (h_\beta) = \delta \frac{B(h_\alpha', h_\alpha')}{B(h_\beta', h_\beta')}$. If moreover we have~$\beta = w^{-1} \cdot \alpha$ for some~$w \in W$, by~\eqref{eq_inv_h} and~$W$-equivariance of~$B$, we have~$\omega_\alpha (h_{w^{-1} \cdot \alpha}) = \delta$. Thus
\begin{equation}\label{eq_majoration_omega}
    |\omega_\alpha (h_{w^{-1} \cdot \alpha})| \leq 1~.
\end{equation}
Let~$\mu \in \mathbb{R}_{>0}$ be such that~$\omega_\alpha - \mu \alpha \in \operatorname{Conv} (W \cdot \omega_\alpha)$. Then there exist~$(\mu_w) \in \mathbb{R}_{>0}^{|W|}$ such that~$\sum_{w \in W} \mu_w = 1$ and ~$\omega_\alpha - \mu \alpha = \sum_{w \in W} \mu_w w \cdot \omega_\alpha$. Evaluating in~$h_\alpha$ and taking the absolute value, we get
\begin{equation}
\begin{split}
    2 \lambda -1 &\leq | 1 - 2 \lambda  |=  |\omega_\alpha (h_\alpha) - \mu \alpha(h_\alpha)| = \Big{|}\sum_{w \in W} \mu_w w \cdot \omega_\alpha(h_\alpha) \Big{|} \\
    &\overset{\text{by \eqref{eq_inv_h}}}{=} \Big{|}\sum_{w \in W} \mu_w \omega_\alpha(h_{w^{-1} \cdot \alpha}) \Big{|} \leq \sum_{w \in W} \mu_w |\omega_\alpha(h_{w^{-1} \cdot \alpha})| \overset{\text{by \eqref{eq_majoration_omega}}}{\leq} \sum_{w \in W} \mu_w = 1~.
\end{split}
\end{equation}
Thus~$\mu \leq 1$. This ends the proof the lemma.~$\qed$
\end{proof}

\begin{rmk}\label{rmk_degree_rep}
    In the proof of Lemma~\ref{lem_degree_rep_1}, the assumption that~$\dim(\g_\alpha) = 1$ is necessary. Indeed, if~$(\g, \alpha)$ is given by item (ix) of Table~\ref{table_nagano_full}, then the triple~$(\PO(n,1), \rho_1, \mathbb{R}^{n+1})$ defined by the canonical embedding~$\PO(n, 1) \subset \PGL(n+1, \mathbb{R})$ is a projective~$\{\alpha_1\}$-proximal triple for~$\soo(n,1)$ with highest weight~$\omega_{\alpha_1}$. However, one has~$V^{\omega_\alpha - k \alpha} \ne 0$ for all~$1 \leq k \leq 2$.
\end{rmk}

\subsection{Plücker triples}\label{sect_plucker_triples}
If~$(\g, \alpha)$ is an irreducible real-type Nagano pair, then there exist~$G \in \mathcal{G}_{\{\alpha\}}(\g)$ and a projective~$\{\alpha\}$-proximal representation~$(V, \rho)$ of~$G$ with highest weight~$\omega_\alpha$. Such a triple~$(G, \rho, V)$ is called a \emph{Plücker triple} of~$(\g, \alpha)$. Recall that the associated embeddings from Fact~\ref{prop_ggkw} are denoted by~$\iota_\rho, \iota_\rho^\opp$. In the next section, we investigate the consequences on the structure of real-type Nagano spaces of the existence of Plücker triples.

But first, let us give an explicit and well-known example of Plücker triple for the Nagano pair~$(\sll(p+q, \mathbb{R}), \alpha_p)$. In that case, the triple
$$(\PGL(p+q, \mathbb{R}), \PRSP, \rho_0)$$
is a Plücker triple of~$(\sll(p+q, \mathbb{R}), \alpha_p)$; explicitly, it is defined by the natural action of \( \PGL(p+q, \mathbb{R}) \) on \( \mathbb{P}(\bigwedge^p \mathbb{R}^{p+q}) \):
\begin{equation}\label{eq_def_plucker_grassman}
    \rho_0(g) \cdot [v_1 \wedge \cdots \wedge v_p] = \big[(\widetilde{g} \cdot v_1) \wedge \cdots \wedge (\widetilde{g} \cdot v_p))\big]
\end{equation}
for any free family~$(v_1, \dots, v_p)$ of~$\mathbb{R}^{p+q}$, where~$\widetilde{g}$ is any lift of~$g$ in~$\GL(p+q, \mathbb{R})$.
The associated embeddings via Fact~\ref{prop_ggkw} are the classical \emph{Plücker embeddings}:
\begin{equation}\label{eq_def_plucker_grassman_iota}
\begin{array}{cccc}
\iota_{\rho_0}: &
     \Grass &\longrightarrow &\mathbb{P} (\bigwedge^p \mathbb{R}^{p+q}) \\
&\operatorname{Span}(v_1, \dots, v_p) &\longmapsto &[v_1 \wedge \cdots \wedge v_p]~;
\end{array}
\end{equation}

\begin{equation*}
\begin{array}{cccc}
\iota^{\opp}_{\rho_0}: &
    \Grassq &\longrightarrow &\mathbb{P} (\bigwedge^p \mathbb{R}^{p+q})^* \\
&\operatorname{Span}(v_1, \dots, v_q) &\longmapsto &[x \mapsto x \wedge v_1 \wedge \cdots \wedge v_q]~.
\end{array}
\end{equation*}

\subsection{Images of photons}   
Plücker triples will be of fundamental importance in our study of proper domains in real-type Nagano spaces. The structural result from which all their main properties will follow is:

\begin{prop} \label{prop_photons_egal_proj_lines}
Let~$(\g, \alpha)$ be a real-type irreducible Nagano pair, and let~$(G, \rho, V)$ be a Plücker triple for~$(\g, \alpha)$. The images of photons in~$\mathbb{P}(V)$ under~$\iota_\rho$ are projective lines of~$\mathbb{P}(V)$.
\end{prop}
This Proposition is proven for the Nagano pair~$(\sll(p+q, \mathbb{R}), \alpha_p)$ and the Plücker triple given in~\eqref{eq_def_plucker_grassman} in~\cite{van2019rigidity}. The authors actually prove that photons are \emph{exactly} the converse images of projective lines that are contained in~$\iota_{\rho_0}(\Grass)$.

\begin{proof}[Proof of Proposition~\ref{prop_photons_egal_proj_lines}] By~$\rho$-equivariance of~$\iota_\rho$, it suffices to prove the proposition for the standard photon~$\Phot_\std$ defined in~\eqref{eq_def_photon_standard}. Let~$\vv_0 \in V^{\omega_\alpha} \smallsetminus \{ 0 \}$ be a lift of~$\iota_\rho(\LP^+)$. For all~$t \in \mathbb{R}$, we have:
\begin{equation}\label{eq_photons_proj_lines}
    \rho(\exp(tv^-)) \cdot \vv_0 = \exp(t \rho_*(v^-)) \cdot \vv_0 = \sum_{k=0}^\infty \frac{t^k \rho_*(v^-)^k}{k!} \cdot \vv_0 = \vv_0 + t \rho_*(v^-) \cdot \vv_0~,
\end{equation}
the last equality holding by Lemma~\ref{lem_degree_rep_1}. Thus
\begin{equation*}
    \iota_\rho(\Phot_\std) = \overline{\{\iota_\rho(\exp(tv^\opp) \cdot \LP^+)\}} =  \overline{\{[\vv_0 + t \rho_*(v^-) \cdot \vv_0] \mid t \in \mathbb{R}\}}
\end{equation*}
is a projective line.~$\qed$
\end{proof}

\begin{rmk} If, instead of considering a Plücker triple for \( (\g, \alpha) \), we take a projective \( \{\alpha\} \)-proximal triple \( (G, \rho', V') \) of \( \g \) with highest weight \( 2\omega_\alpha \), then Proposition~\ref{prop_photons_egal_proj_lines} fails. Indeed, consider for instance de irreducible embedding of~$\SL(2, \mathbb{R})$ in~$\SL(3, \mathbb{R})$. The image of the induced embedding of~$\mathbb{P}(\mathbb{R}^2)$ in~$\mathbb{P}(\mathbb{R}^3)$ is then an ellipsoïd.
\end{rmk}

\section{Kobayashi pseudometric}\label{sect_kob}
In this section, we define and investigate the first properties of the Kobayashi pseudometric. As in Section~\ref{sect_HS_photons_defs}, the content of this section mainly consists of generalizations to general irreducible real-type Nagano spaces of material from \cite{galiay2024rigidity}. 

Even if in \cite{galiay2024rigidity}, the Kobayashi metric is defined and studied only for causal flag manifolds, the proofs of its main properties generalize verbatim. this is the case for all the results of Section~\ref{sect_definition_kobayashi} to~\ref{sect_comp_caratheodory}. Consequently, several results are stated citing~\cite{galiay2024rigidity}, in which case it means that the proof is exactly the same verbatim as the corresponding one in \cite{galiay2024rigidity}. 

On the contrary, Section~\ref{sect_geod_sym_domains} is new even for causal flag manifolds. As mentionned in the introduction, it was proven in collaboration with Chalumeau for item (viii) of Table~\ref{table_nagano_full} with tools coming from conformal pseudo-riemannian geometry, which do not adapt to our general context.

\subsection{Reminders on the Hilbert metric on intervals}\label{sect_def_Hilbert_metric} The definition of the Kobayashi pseudometric uses the Hilbert metric on intervals of the circle, so we need to fix some preliminary notations. We denote by $(\cdot: \cdot: \cdot : \cdot)$ the classical cross ratio on $\mathbb{P}(\mathbb{R}^2)$. Recall that it is~$\SL_2(\mathbb{R})$-invariant and satisfies $([1:0] : [1:1] : [1:t] : [0:1]) = t$. If $I \subset \mathbb{P}(\mathbb{R}^2)$ is a proper open interval with (possibly equal) endpoints $t_1$ and $t_2$, then the \emph{Hilbert pseudometric} on $I$ is denoted $H_I$ and defined as follows: for any pair~$s_1, s_2 \in I$ such that $t_1, s_1, s_2, t_2$ are aligned in this order (taking any order if~$s_1 = s_2$ or~$t_1 = t_2$), one has $H_I(s_1,s_2) := \log(t_1: s_1: s_2 : t_2)$. If $I = \mathbb{P}(\mathbb{R}^2)$, then $H_I$ is by convention the constant map equal to $0$ on $I^2$.

\subsection{The pseudometric}\label{sect_definition_kobayashi}
In this section, we take~$(\g, \alpha)$ to be an irreducible real-type Nagano pair.

We say that two points~$x,y \in \O$ are \emph{photon-related} and write $x\phoo_\O y$ (or simply $x\phoo y$ when there is no ambiguity) if they lie on a same photon~$\Phot$ and in a same connected component of~$\Phot \cap \O$, denoted by~$I$. Let~$\zeta$ be a projective parametrization of~$\Phot$, as in Definition~\ref{def_param_photons}. Let 
\begin{equation*}
    \ro_{\O}(x,y) := \mathsf{H}_{\isl^{-1}\left(I \right)} \left(\isl^{-1}(x), \isl^{-1}(y) \right),
\end{equation*}
where~$\mathsf{H}_{J}$ is by definition the \emph{Hilbert metric} on an interval~$J$ of~$\mathbb{P}(\mathbb{R}^2)$, as recalled in the previous Section~\ref{sect_def_Hilbert_metric}. By Lemma~\ref{fact_stab_u_std}.(4), the quantity~$\ro_{\O}(x,y)$ does not depend on the choice of the projective parametrization of~$\Phot$. 

Now if~$x,y \in \O$, we say that a finite sequence~$(x_0, \dots, x_k) \in \O^{k+1}$, with~$k \in \mathbb{N}$, is \emph{a~$k$-chain} if~$x_0 = y$,~$x_k = y$, and~$x_i, x_{i+1}$ are photon-related for all~$0 \leq i \leq k-1$. We denote by~$\mathcal{C}_{x,y}^k(\O)$ the set of~$k$-chains from~$x$ to~$y$, and
\begin{equation}\label{eq_def_set_og_chains}
    \mathcal{C}_{x,y}(\O) := \bigcup_{k \in \mathbb{N}} \mathcal{C}_{x,y}^k(\O)~.
\end{equation}
An element of~$\mathcal{C}_{x,y}(\O)$ is called a \emph{chain}. We define:
\begin{equation}\label{def_kob_pseudometric}
    \kob_{\O}(x,y) = \inf \Big{\{} \sum_{i=0}^N \ro_{\O}(x_i, x_{i+1}) \mid N \in \mathbb{N}^*, \ \ (x_0, \cdots x_N) \in \mathcal{C}_{x,y}(\O) \Big{\}}~.
\end{equation}

For~$x,y$ sufficiently close to each other, by Remark~\ref{rmk_realizations}.(2) there exists a realization~$\mathbb{B}$ of~$\mathbb{X}(\g, \alpha)$ such that~$x,y \in g\cdot \mathbb{B} \subset\O$. Thus, by the following Lemma~\ref{lem_chain_in_realization}, the set~$\mathcal{C}_{x,y}(\O)$ (and even~$\mathcal{C}^{\operatorname{rk}(\g, \alpha)}_{x,y}(\O)$) is nonempty. Hence the relation ``$x$ and~$y$ can be joined by a chain'' is locally trivial. Since it is an equivalence relation and~$\O$ is connected, it is the trivial relation. Hence the set~$\mathcal{C}_{x,y}(\O)$ is never empty for two points~$x,y \in \O$. Since~$\kob_{\O}(a,b)$ is always finite whenever~$a,b$ are photon-related, the quantity~$\kob_{\O}(x,y)$ is always finite as well. Thus~$\kob_{\O}$ is actually a map~$\O \times \O \rightarrow \mathbb{R}_{+}$.
\begin{lem}\label{lem_chain_in_realization}
    Let~$\O$ be a realization of~$\mathbb{X}(\g, \alpha)$. For all~$x,y \in \O$, there exists a~$\operatorname{rk}(\g, \alpha)$-chain between~$x$ and~$y$.
\end{lem}

This lemma is well known, but an explicit construction of the~$\operatorname{rk}(\g, \alpha)$-chain between~$x$ and~$y$ is given in Section~\ref{sect_geod_sym_domains}. Besides, we have:

\begin{fact}\label{prop_properties_kobayashi_metric}\cite[Prop.\ 6.5]{galiay2024rigidity}
Let~$\O_1$ and~$ \O_2$ be two domains of~$\Fl(\g, \alpha)$, and~$G \in \mathcal{G}_{\{\alpha\}}(\g)$. Then:
\begin{enumerate}
    \item If~$\O_1 \subset \O_2$, then for any~$x,y \in \O_1$ one has~$\kob_{\O_2}(x,y) \leq \kob_{\O_1}(x,y)$.
    \item For any~$g \in G$, for any~$x,y \in \O_1$, one has~$\kob_{g \cdot \O_1}(g \cdot x, g \cdot y) = \kob_{\O_1}(x,y)$. In particular, the metric~$\kob_{\O_1}$ is~$\Aut_G(\O_1)$-invariant.
\end{enumerate}
\end{fact}
Fact~\ref{prop_properties_kobayashi_metric} is proven in~\cite{galiay2024rigidity} only when~$\g$ is a Lie algebra of \emph{Hermitian tube type} and~$\alpha$ defines the \emph{Shilov boundary} of the symmetric space of~$G$, but the proof is exactly the same in our case.

Since one can concatenate chains and reverse the orientation of a chain, the map~$\kob_{\O}$ is symmetric and satisfies the triangle inequality. It is thus a pseudometric, and we call it the \emph{Kobayashi pseudometric}.

\subsection{Reminders on length}\label{section_chains_length} Let $\O$ be a domain contained in an affine chart~$\Affstd$. For a continuous path~$\gamma:[0,1] \rightarrow \O$, we define the \emph{Kobayashi length} or \emph{$\kob_{\O}$-length} of $\gamma$ in the the usual way, as
\begin{equation*}
    \leng(\gamma) = \sup \sum_{i=0}^N \kob_{\O}(\gamma(t_i),\gamma(t_{i+1}))~,
\end{equation*}
where the supremum is taken over all finite subdivisions of $\gamma$. 

Let $x,y \in \O$ be two distinct conjugate points, and let $\Phot$ be the unique photon containing $x$ and $y$. We denote by $[x,y]$ the closure of the only connected component of $\Phot \smallsetminus \{x,y\}$ that is contained in $\O$. By Fact~\ref{lem_photons_affine}, it is an affine segment in $\Affstd$. This segment can be parametrized by~$[t_1,t_2] \rightarrow [x,y]; t  \mapsto \isl([1:t])$, where $\isl$ is a projective parametrization of~$\Phot$ (Definition~\ref{def_param_photons})satisfying $x = \isl([1:t_1])$ and $y = \isl([1:t_2])$. 

Now, let $x, y \in \O$ be any two points. Any element of $u = (x_0, \cdots, x_{N}) \in \mathcal{C}_{x,y}(\O)$ gives rise to a continuous path $\gamma$ from $x$ to $y$, defined as the concatenation of all the segments $[x_0, x_1], \cdots , [x_{N-1}, x_N]$ in this order, each endowed with a parametrization as described above. This path is uniquely defined by $u$ up to reparametrization.
The $\kob_{\O}$-length $\leng(\gamma)$ of $\gamma$ does not depend on the choice of parametrization of the $[x_i, x_{i+1}]$ as constructed above, for $0 \leq i \leq N$. This defines a unique $\kob_{\O}$-length for the chain $(x_0, \cdots, x_{N})$. 

In the rest of the paper, we will identify a chain with the unique (up to parametrization) path it defines by the process described above. In particular, this will allow us to consider the $\kob_{\O}$-length of a chain.

\subsection{Comparison with the Caratheodory pseudometric}\label{sect_comp_caratheodory} Recall that dual convexity is defined in Section~\ref{sect_the_dual}. The Kobayashi pseudometric on dually convex domains can be related to the Caratheodory pseudometrics defined in~\eqref{Cara}: 

\begin{prop}[Metric comparison]\label{lem_key_lemma} Let~$(\g, \alpha)$ be an irreducible real-type Nagano pair and~$(G, V, \rho)$ a linear~$\{\alpha\}$-proximal triple for~$(\g, \alpha)$. Let~$\O \subset \Fl(\g, \alpha)$ be a dually convex domain, and let~$C_{\O}^\rho$ be the Caratheodory metric on~$\O$ induced by~$(V, \rho)$. Then for any two photon-related points~$x,y \in \O$, one has
\begin{equation}\label{eq_ineq_carat}
        \ro_\O (x,y) = \frac{1}{\chi(h_\alpha)}C_{\O}^\rho(x,y)~.
\end{equation}
Furthermore,  
    \begin{enumerate}
        \item one has~$\kob_{\O} \geq \chi_\rho(h_\alpha)^{-1}C_{\O}^\rho$;
        \item Given two photon-related points~$x,y \in \O$, the~$\kob_{\O}$-length of the~$1$-chain~$(x,y)$ is equal to~$\ro_{\O}(x,y) = \kob_{\O}(x,y)$. 
    \end{enumerate}
\end{prop}
 
This proposition is a straightforward generalization to any Nagano space, any dually convex domains and any~$\{\alpha\}$-proximal triple of~$\g$, of \cite[Prop.\ 6.10]{galiay2024rigidity}. It relies on the following Lemma~\ref{lem_comp_crossratios}, which is proved in \cite[Lem.\ 6.13]{galiay2024rigidity} in the case where~$\g$ is a Hermitian Lie algebra of tube type, and~$\Fl(\g,\alpha)$ is the Shilov boundary of the symmetric space of~$\g$, for~$\chi = m \omega_\alpha$ ($m \in \mathbb{N}{>0}$). We subsequently generalized our proof to all irreducible real Nagano spaces in \cite[Lem.\ 6.3.13]{galiay2025convex}. On the other hand, Beyrer--Guichard--Labourie--Pozzetti--Wienhard independently proved this lemma in \cite[Prop.\ 3.27]{beyrer2024positivity}. Their proof is valid in any flag manifold defined by a subset~$\Theta$ whose roots have multiplicity~1.

Given a photon~$\Phot$, we set
$$\Contlam_\Phot = \left\{ \xi \in \Fl(\g, \alpha)^\opp \mid \Phot \not\subset \hyp_{\xi}\right\}~.$$
Recall that, by Fact~\ref{lem_photons_affine}, given a point~$\xi \in \Contlam_{\Phot}$, the set~$\hyp_{\xi} \cap \Phot$ is a singleton. We can then define a projection~$\pstd_{\Phot}: \Contlam_{\Phot} \rightarrow \Phot$ by setting~$\pstd_{\Phot}(\xi) := a$ where~$\{a\} = \hyp_{\xi} \cap \Phot$ if~$\xi \in \Contlam_{\Phot}$.  

\begin{lem}\label{lem_comp_crossratios}\cite{beyrer2024positivity, galiay2025convex}
Let~$(\g, \alpha)$ be an irreducible real-type Nagano pair, and let~$(G, \rho, V)$ be a linear~$\{\alpha\}$-proximal triple of~$\g$. Let~$\Phot$ be a photon and let~$\isl$ be a projective parametrization of ~$\Phot$. Let~$\xi_1, \xi_2 \in \Contlam_{\Phot}$, and for~$i \in \{1,2\}$, let~$b_i \in \mathbb{P}(\mathbb{R}^2)$ be such that~$\pstd_\Phot(\xi_i) = \isl(b_i)$. Then for all~$a_1, a_2 \in \mathbb{P}(\mathbb{R}^2)$, one has
    \begin{equation*}
       \log \big{|}[\xi_1 : \isl(a_1) : \isl(a_2) : \xi_2]_{\rho} \big{|} = \chi_\rho(h_\alpha) |\log (b_1: a_1: a_2: b_2)|~.
    \end{equation*}
\end{lem}

Since Proposition~\ref{lem_key_lemma} is not established anywhere in this setting and generality, we provide a proof here:

 \begin{proof}[Proof of Proposition~\ref{lem_key_lemma}] First note that we may assume~$\O \ne \Fl(\g, \alpha)$, otherwise the proposition is immediate. Hence~$\partial \O \ne \varnothing$ which by dual convexity implies that~$\O$ is contained in an affine chart.

Let~$\Phot$ be the photon through~$x$ and~$y$, and let~$\zeta$ be a projective parametrization of~$\Phot$. Then there exist~$a_1, a_2 \in \mathbb{P}(\mathbb{R}^2)$ such that~$x= \isl(a_1)$ and~$ y = \isl(a_2)$. 

Recall that we denote by~$I_{x,y}$ the connected component of~$\Phot_{\mathsf{std}} \cap \O$ containing~$x$ and~$y$. Since~$\O$ is contained in an affine chart, the interval~$I_{x,y} \subset \Phot_\std$ has nonempty boundary. Let~$p_1,p_2 \in \partial \O$ be the (potentially equal) endpoints of~$I_{x,y}$, such that~$p_1, x,y, p_2$ are aligned on~$\Phot_{\mathsf{std}}$ in this order. Then there exist~$b_1, b_2 \in \mathbb{P}(\mathbb{R}^2)$ such that~$b_1, a_1, a_2, b_2$ are aligned in this order and~$p_1 = \isl(b_1)$, and~$p_2 = \isl(b_2)$.

By dual convexity, for~$i \in \{ 1,2 \}$ there exists~$\xi_i\in \O^*$ such that~$p_i \in \hyp_{\xi_i}$. Then, by Lemma~\ref{lem_comp_crossratios}, one has~$\ro_{\O}(x,y) =  \log \left| (b_1 : a_1 : a_2: b_2)\right|  = \chi(h_\alpha)^{-1}\log \big{|} \left[\xi_1 : x : y : \xi_2 \right]_{\rho}\big{|}$. By the definition of~$C_{\O}^\rho$, this implies that~$\ro_{\O}(x,y) \leq \chi(h_\alpha)^{-1} C_{\O}^\rho(x,y)$. 

For the converse inequality, let~$\eta_1, \eta_2 \in \O^*$ be such that~$C_{\O}^\rho(x,y) = \log\left|[\eta_1 : x:y: \eta_2]_{\rho}\right|$. 
For~$i \in \{ 1,2\}$, let~$b_i' \in \mathbb{P}(\mathbb{R}^2)$ be such that~$\isl(b_i') = \pstd_\Phot(\eta_i)$. Then, again by Lemma~\ref{lem_comp_crossratios}:
\begin{equation*}
\big{| } \log \left| (b_1' : a_1 : a_2 : b_2') \right| \big{| }=  \chi(h_\alpha)^{-1} \log \left| [\eta_1 : x: y: \eta_2]_{\rho} \right|~.
\end{equation*}
Since~$\eta_1, \eta_2 \in \O^*$, the two points~$\isl(b_1'), \isl(b_2')$ are not contained in~$I_{x,y}$. Thus one has
\begin{equation*}
    \big{| } \log \left| (b_1' : a_1 : a_2 : b_2') \right| \big{| } \leq \log \left| \left(b_1 : a_1 : a_2 : b_2 \right) \right| = \ro_{\O}(x,y)~.
\end{equation*}
Hence one has~$
    \chi(h_\alpha)^{-1}C_{\O}^\rho(x,y) \leq \ro_{\O}(x,y)$. This proves~\eqref{eq_comp_kob_carat}.

Now let us prove point~(1). Let~$x, y\in \O$ be any two points, and let~$(x_0, \dots , x_M) \in \mathcal{C}_{x,y}(\O)$. Then one has
\begin{equation}\label{eq_comp_kob_carat}
    \begin{split}
        \sum_i \ro_\O(x_i, x_{i+1}) &= 
    \chi(h_\alpha)^{-1}\sum_i C_{\O}^\rho(x_i, x_{i+1}) \geq \sum_i \chi(h_\alpha)^{-1} \sup_{\xi_1^i, \xi_2^i \in \O^*}\log[\xi_1^i : x_i: x_{i+1}: \xi_2^i]_{\rho} \\
    &\geq  \chi(h_\alpha)^{-1}\sup_{\xi_1, \xi_2 \in \O^*} \sum_i \log[\xi_1 : x_i: x_{i+1}: \xi_2]_{\rho} \\
    &\geq  \chi(h_\alpha)^{-1} \sup_{\xi_1, \xi_2 \in \O^*}  \log[\xi_1 : x: y: \xi_2]_{\rho}   = \chi(h_\alpha)^{-1}C_{\O}^\rho(x,y).
    \end{split}
\end{equation}
Since this is true for all~$(x_0, \dots, x_{M}) \in \mathcal{C}_{x,y}(\O)$, by taking the infimum we get the statement of point~(1).

Point~(2), just follows from taking $x$ and $y$ to be two conjugate points in~\eqref{eq_comp_kob_carat}.

The fact that $C_{\O}^\rho(x,y) = N \ro_\O(x,y)$ and Equation~\eqref{eq_comp_kob_carat} imply that the segment~$[x,y]$ has $\kob_{\O}$-length $\ro_{\O}(x,y)$. Hence the $\kob_{\O}$-length of $\gamma = (x_0, \cdots, x_M) \in \mathcal{C}_{x,y}(\O)$ is 
\begin{equation}\label{eq_longueur_chaine}
    \leng(\gamma) = \sum_i \kob_{\O}(x_i, x_{i+1})~.
\end{equation}
Then one has $\kob_{\O}(x,y) = \inf \left\{ \leng(\gamma) \mid \gamma \in \mathcal{C}_{x,y}(\O) \right\}$.

Now let $\mathcal{C}_{x,y}'(\O)$ the set of all rectifiable curves joining $x$ and $y$ in $\O$. By the definition of the length of a curve, one has $\kob_{\O}(x,y) \leq \inf \left\{ \leng(\gamma) \mid \gamma \in \mathcal{C}_{x,y}'(\O) \right\}$. Since chains are rectifiable (for the identification with continuous paths, see Section~\ref{section_chains_length}), this last inequality is an equality.

~$\qed$
 \end{proof}

\section{Kobayashi-hyperbolicity}\label{sect_kob_hyperbolicity}
We say that a domain~$\O \subset \Fl(\g, \alpha)$ is \emph{Kobayashi-hyperbolic} if~$\kob_{\O}$ is a metric, that is, if~$\kob_{\O}$ separates points. The map~$\kob_{\O}$ is then called the \emph{Kobayashi metric} of~$\O$. In this section, we prove that both proper domains and~\emph{$\MRr$-proper} dually convex domains are Kobayashi-hyperbolic. An important consequence is that~$ \Aut(\O)$ acts properly on~$\O$ (Corollary~\ref{cor_action_proper}). This was already known for proper domains (Fact~\ref{fact_autom_group_proper}), but not for~$\MRr$-proper dually convex domains).

\subsection{Continuity} In this section, we prove two auxiliary lemmas for the study of Kobayashi-hyperbolicity.

Given an affine space~$\operatorname{A}$ and a bounded convex open subset~$C \subset \operatorname{A}$, we denote by~$\mathsf{H}_C$ the classical \emph{Hilbert metric} on~$C$ (see \cite{papadopoulos2014handbook}).

\begin{lem}\label{lem_comp_hilbert_kob}
    Let~$\Affstd$ be an affine chart of~$\Fl(\g, \alpha)$, and let~$C \subset \Affstd$ be a bounded convex open subset of~$\Affstd$. Let~$x,y \in C$ be two points on a same photon. Then~$\ro_C(x,y) = \mathsf{H}_C(x,y)$.
\end{lem}

\begin{proof}
    This is just a consequence of the definition of~$\mathsf{H}_C$ and the fact that the trace of a photon in~$\Affstd$ is either empty or an affine line of~$\Affstd$ (Fact~\ref{lem_photons_affine}).~$\qed$
\end{proof}

\begin{lem}\label{lem_kob_contin}
    Let~$\O \subset \Fl(\g, \alpha)$ be a domain. Then~$(x,y) \mapsto \kob_\O(x,y)$ is continuous with respect to the standard topology on~$\O$.
\end{lem}

\begin{proof} By the inequality~$|\kob_{\O}(x_0, y_0 ) - \kob_{\O}(x, y)| \leq \kob_{\O}(x_0, x ) + \kob_{\O}(y_0, y)$ for all~$x_0, y_0, x, y \in \O$, one only needs to show that for any~$x_0 \in \O$ the map~$x \mapsto \kob_{\O}(x_0, x)$ is continuous at~$x_0$. For this we fix an euclidean norm~$|.|$ on~$\Affstd$. For any~$0 <\delta< 1$, let~$B_\delta(x_0)$ be the Euclidean ball of center~$x_0$ and of radius~$\delta$. For~$\delta$ small enough,~$B_\delta(x_0)$ is contained in a realization of~$\mathbb{X}(\g, \alpha)$, itself contained in~$\O$. In this case, for~$x \in B_\delta(x_0)$, by Lemma~\ref{lem_chain_in_realization}, there exists an~$N$-chain~$(x_0, x_1, \dots , \ x_{N} = x)$ contained in~$B_\delta(x_0)$. Then, by Fact~\ref{prop_properties_kobayashi_metric} and Lemma~\ref{lem_comp_hilbert_kob}, one has
\begin{align*}
  \kob_{\O}(x_0,x) &\leq \kob_B (x_0, x) 
  \leq \sum_{k=0}^{N-1} \kob_B(x_i, x_{i+1}) = \sum_{k=0}^{N-1} \mathsf{H}_B(x_i, x_{i+1}) \leq \sum_{k=0}^{N-1} \mathsf{H}_B(x_0, x_{i}) + \mathsf{H}_B(x_0, x_{i+1}) \\
  &= \sum_{k=1}^{N-1}  \Big{(}\log \frac{1 +\lvert x_i-x_0 \rvert }{1 - \lvert x_i-x_0 \rvert} + \log \frac{1+ \lvert x_{i+1}-x_0\rvert }{1 - \lvert x_{i+1}-x_0\rvert}\Big{)} \leq  \sum_{k=1}^{N-1}  2\log \frac{1 + \delta }{1 - \delta} \ \underset{\delta \rightarrow 0}{\longrightarrow} 0~.
\end{align*}
The last equality of the first line is a consequence of Lemma~\ref{lem_comp_hilbert_kob}
This proves that~$\kob_{\O}(x_0, x ) \rightarrow 0$ as~$x \rightarrow x_0$.~$\qed$
\end{proof}

A direct consequence of Lemma~\ref{lem_kob_contin} is the following, which gives the first statement of Theorem~\ref{thm_kobayashi_metric_nagano}.(1) (the completeness statement is proven in Corollary~\ref{cor_kobayashi_geodesic}):

\begin{prop}\label{prop_generate_topo_std}Proper domains of~$\Fl(\g, \alpha)$ are Kobayashi-hyperbolic. Moreover, for any proper domain~$\O \subset \Fl(\g, \alpha)$, the metric~$\kob_{\O}$ generates the standard topology.
\end{prop}

Using Lemma~\ref{lem_kob_contin}, Proposition~\ref{prop_generate_topo_std} admits the  exact same proof as the one of \cite[Prop.\ 6.7]{galiay2024rigidity}. Hence, we do not provide it. The more involved part of Theorem~\ref{thm_kobayashi_metric_nagano} is point (2), which is carried out in next section.

\subsection{$\MRr$-proper domains}
Since the Kobayashi metric is defined with photons, natural domains to investigate in this context are~\emph{$\MRr$-proper} domains:

\begin{definition}
    A domain~$\O \subset \Fl(\g, \alpha)$ is said to be~\emph{$\MRr$-proper} if for any photon~$\Phot \subset \Fl(\g, \alpha)$, the set~$\Phot \smallsetminus (\O \cap \Phot)$ contains at least~$2$ points. It is \emph{strongly~$\MRr$-proper} if~$\overline{\O}$ does not contain any photon (note that strong~$\MRr$-properness is stronger than~$\MRr$-properness).
\end{definition}

In this section, we investigate the metric properties of~$\MRr$-proper domains.

\subsubsection{Photon-convexity} We say that an~$\MRr$-proper domain~$\O \subset \Fl(\g, \alpha)$ is \emph{photon-convex} if for any photon~$\Phot$, the set~$\Phot \cap \O$ is connected, and~$\overline{\Phot \cap \O} = \Phot \cap \overline{\O}$ whenever~$\Phot\cap \Omega \ne \varnothing$.

\begin{lem}\label{lem_continuité_intersection_photon}
    Let~$\O \subset \Fl(\g, \alpha)$ be a~$\MRr$-proper dually convex domain. Then~$\O$ is photon-convex.
\end{lem}

\begin{proof} Since~$\O$ is~$\MRr$-proper and dually convex, there exists~$\xi \in \O^*$ such that~$\hyp_\xi \cap \Phot \subset \partial \O$. Hence we may assume that~$\O \subset \Affstdstd$.

Let~$(a,b)$ be a connected component of~$\Phot \cap \O$. There exist~$\xi_a, \xi_b \in \O^*$ such that~$a \in \hyp_{\xi_a}$ and~$b \in \hyp_{\xi_b}$. If~$\xi_a = \xi_b = \LP^\opp$, then~$a = b$, and~$\O$ is not~$\mathcal{R}$-proper, so we may assume~$\xi_a \ne \LP^\opp$.

We fix a Plücker triple~$(G, \rho, V)$ of~$(\g, \alpha)$. Let~$\Affstd$ be the affine chart of~$\mathbb{P}(V)$ defined by~$\Affstd := \mathbb{P}(V) \smallsetminus \iota^\opp_\rho(P^\opp)$. Then~$\iota_\rho(\O) \subset \Affstd$. Moreover, the image~$\iota_\rho(\Phot)$ is a projective line of~$\mathbb{P}(V)$ by Proposition~\ref{prop_photons_egal_proj_lines}, hence~$\iota_\rho(\Phot) \cap \Affstd$ is an affine line of~$\Affstd$, and we denote by~$(x,y)$ a segment of this affine line, for~$x,y \in \iota_\rho(\Phot) \cap \Affstd$. We may chose a parametrization of~$\iota_\rho(\Phot) \cap \Affstd$ so that~$(\iota_\rho(b), \iota_\rho(a)) \subset (-\infty, a)$ (where the equality~$b = - \infty$ is allowed).

Since~$\xi \ne \LP^\opp$, we have~$\iota^\opp_\rho(\xi) \ne \iota^\opp_\rho(\LP^\opp)$. Hence~$\iota^\opp_\rho(\xi) \cap \Affstd$ is an affine hyperplane of~$\Affstd$; thus it separates it into two half spaces~$\Affstd^+$ and~$\Affstd^\opp$. Since~$\iota_\rho(\Phot) \cap \Affstd$ is an affine line of~$\Affstd$ not contained in~$\iota^\opp_\rho(\xi)$, up to exchanging~$\Affstd^+$ and~$\Affstd^\opp$, we have~$(- \infty , \iota_\rho(a)) = \iota_\rho(\Phot) \cap \Affstd^+$ or~$(- \infty , \iota_\rho(a)) = \iota_\rho(\Phot) \cap \Affstd^\opp$, and we may assume for instance the first equality. Then~$(\iota_\rho(b), \iota_\rho(a)) \subset \Affstd^+$, and since~$\iota_\rho(\O)$ is connected and does not intersect~$\iota^\opp_\rho(\xi)$, one has~$\iota_\rho(\O) \subset \Affstd^+$. Hence~$\iota_\rho(\O \cap \Phot) \subset (-\infty, \iota_\rho(a))$. Now if~$b = - \infty$, the lemma is proved. If~$b \ne - \infty$, then~$b \in \Affstdstd$, so the arguments above replacing~$a$ by~$b$ and~$- \infty$ by~$+ \infty$ give~$\iota_\rho(\O \cap \Phot) \subset (\iota_\rho(b), + \infty)$. In conclusion, we have~$\iota_\rho(\O \cap \Phot) \subset (\iota_\rho(b), \iota_\rho(a))$, which is equivalent to~$\O \cap \Phot \subset (b,a)$. Since~$(b,a)$ is a connected component of~$\O \cap \Phot$, this inclusion is an equality. This proves the lemma.~$\qed$
\end{proof}

\subsubsection{Kobayashi norm}
Recall the element~$v^\opp$ of~$\uu^\opp$ defined in Notation~\ref{notation_nagano}. We define the subset~$\RUN \subset \uu^\opp$ by
\begin{equation}\label{eq_A_1_def}
    \RUN := \Ad(P^\opp) \cdot \g_{-\alpha} = \Ad(L) \cdot v^\opp.
\end{equation}
The second equality holds by Lemma~\ref{fact_stab_u_std}.(1). By Lemma~\ref{fact_stab_u_std}.(2) and Section~\ref{sect_dilations_translations}, the set~$\RUN$ is the set of elements~$X \in \uu^\opp$ ``generating a photon through~$\LP^+$'', i.e.\ such that~$\overline{\exp(\mathbb{R}X) \cdot \LP^+}$ is a photon. A reformulation of Fact~\ref{fact_stab_u_std}.(5)
is the following:

\begin{fact}\cite[Lem.\ 3.11]{beyrer2024positivity}\label{fact_photon_closed}
  The set~$\mathbb{P}(\mathcal{A}_1) \subset \mathbb{P}(\uu^\opp)$ is closed.
\end{fact}

Now let~$\O \subset \Fl(\g, \alpha)$ be a domain contained in the standard affine chart~$ \Affstdstd$. Let~$\lvert \cdot \rvert$ be any norm on~$\uu^\opp$. Let~$\delta_\O: \O \times \mathcal{A}_1 \rightarrow \mathbb{R}_{\geq 0}$ be the function defined by
    \begin{equation*}
        \delta_\O(x,v) = \inf\{\lvert Y - \varphistd^{-1}(x)\rvert \mid Y \in \partial \varphistd^{-1}(\O) \cap (\varphistd^{-1}(x) + \mathbb{R}v)\}~,
    \end{equation*}
where~$\varphistd$ is the parametrization of the standard affine chart~$\Affstdstd$, defined in~\eqref{eq_A_egal_exp}. If~$\O$ is~$\MRr$-proper, then we have~$\delta_\O < + \infty$.
\begin{lem}[No overhangs]
    If~$\O$ is photon-convex, then the function~$\delta_\O$ is continuous.
\end{lem}

\begin{proof}
    Let~$(x,v) \in \O \times \mathcal{A}_1$ and let~$(x_k, v_k) \in ( \O \times \mathcal{A}_1)^\mathbb{N}$ such that~$(x_k, v_k) \rightarrow (x,v)$. For all~$k \in \mathbb{N}$, there exists~$Y_k \in (\varphistd^{-1}(x_k) + \mathbb{R}v_k)$ such that~$\lvert Y_k - \varphistd^{-1}(x_k)\rvert$ is minimal. By definition of photon-convexity, the set~$(\varphistd^{-1}(x) + \mathbb{R}v)$ contains exactly two points, so let~$Z_k$ be the other point of this set.
    
    Up to taking a subsequence, we may assume that~$Y_k \rightarrow Y \in  (\varphistd^{-1}(x) + \mathbb{R}v)$ and~$Z_k \rightarrow Z \in (\varphistd^{-1}(x) + \mathbb{R}v)$. By Lemma~\ref{lem_continuité_intersection_photon}, the set~$\varphistd^{-1}(\Omega) \cap (\varphistd^{-1}(x) + \mathbb{R}v)$ contains exactly two points. But by~$\MRr$-properness of~$\O$, we must have~$Z \ne Y$. Thus~$(\varphistd^{-1}(x) + \mathbb{R}v) = \{ Y, Z\}$. Now for all~$k \in \mathbb{N}$, we have~$\lvert Z_k - \varphistd^{-1}(x_k) \rvert \geq  \lvert Y_k - \varphistd^{-1}(x_k) \rvert$, so taking the limit, we have~$\lvert Z - \varphistd^{-1}(x) \rvert \geq  \lvert Y - \varphistd^{-1}(x) \rvert$. Hence~$\delta_\O(x, v) = \lvert Y - \varphistd^{-1}(x) \rvert = \lim_{k \rightarrow + \infty} \delta_\O(x_k, v_k)$.~$\qed$
\end{proof}

\begin{rmk}[Infinitesimal form]\label{rmk_infinitesimal_form} A piecewise smooth path~$\gamma: I \rightarrow \Affstdstd$, where~$I$ is an interval of~$\mathbb{R}$, is said to be a \emph{null geodesic} if we have~$\gamma'(t) \in \mathcal{A}_1$, where we have identified~$\uu^\opp$ with the tangent space to~$\Fl(\g, \alpha)$ at~$\gamma(t)$. The set of null geodesics between two points~$x,y \in \O$ is denoted by~$\mathcal{G}_{x,y}(\O)$. For any~$\gamma \in \mathcal{G}_{x,y}(\O)$, we define the \emph{length of~$\gamma$ as}
\begin{equation*}
    L(\gamma) := \int_{I}\delta_\O(\gamma(t), \gamma'(t)) dt~.
\end{equation*}
Let~$\O \subset \Fl(\g, \alpha)$ be a domain. Then~$\delta_\O$ is upper semicontinuous and one has
\begin{equation}\label{eq_kobayashi_egal_inf}
    \kob_\O(x,y) = \inf \{L(\gamma) \mid \gamma \in \mathcal{G}_{x,y}(\O)\}~.
\end{equation}
This is due to~\cite[Thm.~4.8]{markowitz_1981}, in which the result is stated in the pseudo-Riemannian conformal setting, that is, for item~(viii) of Table~\ref{table_nagano_full}, but more general arguments from parabolic geometry allow one to extend the result to irreducible Nagano spaces of real type (see~\cite[Rmk.~2.11]{markowitz_1981}).

In fact, our paper focuses on Kleinian $(G,G/P)$-manifolds where~$G/P$ is a Nagano space. However, one may also consider the more general framework of manifolds that are \emph{infinitesimally modeled on~$(G,G/P)$}; in other words, manifolds~$M$ endowed with a \emph{Cartan geometry} modeled on~$(G,G/P)$. This amounts to the existence of a \emph{Maurer--Cartan form} on a principal~$P$-bundle~$\mathcal{G}$ over~$M$. One then says that~$(\mathcal{G},M)$ is a $(G,G/P)$-parabolic geometry. This theory, initiated by Cartan and recently developed by Čap and Slovák~\cite{vcap2024parabolic}, encompasses the case of $(G,G/P)$-manifolds: this corresponds to the situation where the curvature of the Maurer--Cartan form vanishes, in which case the $(G,G/P)$-parabolic geometry~$(\mathcal{G},M)$ is said to be \emph{flat}. The article~\cite{markowitz_1981} treats more generally the \emph{non-flat} case of pseudo-Riemannian conformal geometry for the sake of conciseness, but \cite[Rmk.~2.11]{markowitz_1981} suggests that the arguments generalize to parabolic geometry.

We believe that~\eqref{eq_kobayashi_egal_inf} may provide a way to compute the Kobayashi pseudodistance for non-symmetric domains, using methods more directly inspired by differential geometry.~$\qed$
\end{rmk}

\subsubsection{Hyperbolicity of $\MRr$-proper domains} In this section, we prove that~$\MRr$-proper photon-convex domains are Kobayashi-hyperbolic (Theorem~\ref{thm_kob_hyp_R_propre} below), as well as Corollary~\ref{cor_action_proper}.

\begin{thm}\label{thm_kob_hyp_R_propre}
    If~$\O$ is~$\mathcal{R}$-proper and photon-convex, then~$\kob_\O$ is an~$\Aut_G(\O)$-invariant metric generating the standard topology. If~$\O$ is moreover dually convex, then~$\kob_\O$ is a length metric, and for any $N \in \mathbb{N}$ and any $N$-chain $\gamma := (x_0, \dots, x_N)$ in $\O$, one has
    \begin{equation}\label{eq_longueur_chaine_1}
        \leng (\gamma) = \sum_{i=0}^{N-1} \ro_\O(x_i, x_{i+1}) = \sum_{i=0}^{N-1} \kob_\O(x_i, x_{i+1})~.
    \end{equation}
\end{thm}

\begin{proof} The first step is to prove that~$\kob_\O$ separates points. To this end, we adapt the proof of~\cite[Thm 4.6]{van2019rigidity}, in which they prove that~$\MRr$-proper domains of the Grassmannians~$\Grass$ of~$p$-planes in~$\mathbb{R}^n$ that are \emph{convex in an affine chart} (see Remark~\ref{rmk_convex_carte_affine}) are Kobayashi-hyperbolic. Here, the formalism introduced on Nagano spaces allows us to generalize the arguments to real-type Nagano spaces, and to adapt them to domains that are not convex in an affine chart anymore but instead dually convex.

We may assume that~$\O \subset \Affstdstd$. Let~$X, Y \in \uu^\opp$ be such that~$x := \varphistd(X), y := \varphistd(Y) \in \O$ are two distinct points. There exists~$\varepsilon > 0$ such that
$$B_{\varepsilon}(x) := \{Z \in \uu^\opp \mid \lvert Z-X \rvert < \varepsilon\} \subset \varphistd^{-1}(\O)~,$$ 
but~$Y \notin B_{\varepsilon}(x)$. By Fact~\ref{fact_photon_closed} and the continuity of~$\delta_\O$, there exists~$M>0$ such that~$\delta_\O(\varphistd(Z),v) \leq M$ for all~$Z \in B_{\varepsilon}(x)$ and~$v \in \mathcal{A}_1$ --- recall that~$\mathcal{A}_1$ is defined in~\eqref{eq_A_1_def}. 

Let~$Z_1, Z_2 \in B_{\varepsilon}(x)$ such that~$Z_1 - Z_2 \in \mathcal{A}_1$. Let
$$A, B \in (Z_1 + \mathbb{R}(Z_2  -Z_1)) \cap \partial \O$$ 
be such that~$A, Z_1, Z_2, B$ are aligned in this order. Up to exchanging~$A$ and~$B$ we may assume that~$\lvert A- Z_1 \rvert = \delta_\O(\varphistd(Z_1), Z_1 -Z_2) \leq M$. Then, by definition of~$\ro_\O$, one has
\begin{align*}
        \ro_\O(\varphistd(Z_1), \varphistd(Z_2)) &= \big{|} \log \frac{\lvert Z_1 - A \rvert \lvert Z_2 - B \rvert}{\lvert Z_1 - B \rvert \lvert Z_2 - A \rvert} \big{|} \geq  \log \frac{\lvert Z_2 - A \rvert }{\lvert Z_1 - A \rvert} = \int_{\lvert Z_1 -A\rvert}^{\lvert Z_2 - A\rvert} \frac{dt}{t} \\
        &\geq \int_{\lvert Z_1 -A\rvert}^{\lvert Z_2 - A\rvert} \frac{dt}{\lvert Z_2 - A\rvert} =  \frac{\lvert Z_2 - A\rvert - \lvert Z_1 - A\rvert}{\lvert Z_2 - A\rvert}~.
\end{align*}
Since~$Z_1, Z_2$ and~$A$ are on the same affine line of~$\uu^\opp$ and~$\lvert Z_1 - Z_2 \rvert \leq \varepsilon$, we have
\begin{equation}\label{eq_minoration_locale_ro}
    \ro_\O(\varphistd(Z_1), \varphistd(Z_2)) \geq \frac{\lvert Z_1 - Z_2 \rvert}{ \varepsilon + M}~.
\end{equation}

Now let~$N \in \mathbb{N}$ and~$A_1, \dots, A_N \in \varphi^{-1}_\std(\O)$ be such that~$\gamma := (\varphistd(A_1), \dots, \varphistd(A_N))$ is in~$ \mathcal{C}_{x,y}^N(\O)$. By additivity of~$\ro_\O$ on photons, we may assume that there exists~$1 \leq k \leq N-1$ such that~$A_1, \dots, A_{k} \in  B_{\varepsilon}(x)$ and~$A_{k+1} \in \partial  B_{\varepsilon}(x)$. Then by~\eqref{eq_minoration_locale_ro}, we have
\begin{equation*}\sum_{i=1}^{N-1} \ro_\O(\varphistd(A_i), \varphistd(A_{i+1})) \geq \sum_{i=1}^{k} \ro_\O(\varphistd(A_i), \varphistd(A_{i+1}))  \geq \frac{1}{\varepsilon + M} \big{(} \sum_{i=1}^{k} \lvert A_{i+1} -A_i \rvert  \big{)} \geq \frac{\varepsilon}{ \varepsilon + M}~.
\end{equation*}
Taking the infimum over~$\mathcal{C}_{x,y}(\O)$ gives
\begin{equation}\label{eq_kob_positif}
    \kob_\O(x,y) \geq \frac{\varepsilon}{ \varepsilon + M}~.
\end{equation}

We have just proved that~$\kob_\O$ separates the points, which implies that it is a metric. Note that~\eqref{eq_kob_positif} holds for all~$y \in \O \smallsetminus \varphistd(B_\varepsilon(x))$. This implies that the standard topology is included in the topology generated by~$\kob_\O$. The converse inclusion is given by Lemma~\ref{lem_kob_contin}. Thus~$\kob_\O$ generates the standard topology.

In the dually convex case, by Lemma~\ref{lem_continuité_intersection_photon} and the above paragraph, the map~$\kob_\O$ is an~$\Aut(\O)$-invariant length metric generating the standard topology. The fact that~$\kob_{\O}$ is a length metric and Equation \eqref{eq_longueur_chaine_1} then follow by~\eqref{eq_longueur_chaine}, with the same proof as the one of Proposition~\ref{lem_key_lemma}.(2).~$\qed$
\end{proof}

A corollary of Theorem~\ref{thm_kob_hyp_R_propre} is:

\begin{cor}\label{cor_action_proper}
    If~$\O \subset \Fl(\g, \alpha)$ is~$\MRr$-proper and photon-convex, then~$\Aut_G(\O)$ acts properly on~$\O$.
\end{cor}

\begin{proof}
The proof is exactly the same as the one of~\cite[Cor.\ 5.3]{zimmer2018proper}, replacing~$C_\O$ with~$\kob_\O$. The key ingredient here is the existence of an~$\Aut_G(\O)$-invariant metric generating the standard topology.~$\qed$
\end{proof}

The last statement of this section is a partial converse implication to Theorem~\ref{thm_kob_hyp_R_propre}:

\begin{prop}\label{prop_kob_hyp_implies_R_proper}
    Let~$\O \subset \Fl(\g, \alpha)$ be a dually convex domain, and assume that~$\kob_\O$ is a metric. Then~$\O$ is~$\MRr$-proper.
\end{prop}
\begin{proof}
    By contrapsition, if~$\O$ is not~$\MRr$-proper, then there exists a photon~$\Phot$ such that~$\Phot \smallsetminus (\O \cap \Phot)$ is contained in a singleton. Let~$x, y \in \O \cap \Phot$ be two distinct points. Then by Proposition~\ref{lem_key_lemma}.(2), one has~$\kob_\O(x,y) = 0$, and~$\O$ is not Kobayashi-hyperbolic.~$\qed$
\end{proof}

\subsection{Completeness}

In the notation of Proposition~\ref{lem_key_lemma}, Zimmer proves \cite[Thm 9.1]{zimmer2018proper} that a Caratheodory metric~$C_\O^\rho$ on a proper domain~$\O$ is complete whenever~$\O$ is dually convex. The proof even shows that this is equivalent to~$C_\O^\rho$ being a proper metric. Proposition~\ref{lem_key_lemma} then directly implies:

\begin{cor}\label{cor_kobayashi_geodesic}
    Let~$(\g, \alpha)$ be an irreducible real-type Nagano pair and let~$\O \subset \Fl(\g, \alpha)$ be a proper dually convex domain. Then~$\kob_\O$ is a proper and complete geodesic metric on~$\O$.
\end{cor}
Corollary~\ref{cor_kobayashi_geodesic} is already known for \emph{causal flag manifolds} \cite{galiay2024rigidity} and Einstein Universes \cite{chalumeau2024rigidity}.
We also have a weaker version of this corollary for~$\MRr$-proper domains:

\begin{cor}\label{cor_photon_convex_quasi_hom_complete}
    Let~$(\g, \alpha)$ be an irreducible real-type Nagano pair and let~$\O \subset \Fl(\g, \Theta)$ be an~$\MRr$-proper photon-convex quasi-homogeneous domain. Then~$\kob_\O$ is a complete geodesic metric on~$\O$.
\end{cor}

\begin{proof} This directly follows from Theorem~\ref{thm_kob_hyp_R_propre} and \cite[Lem.\ 9.2]{zimmer2018proper}. $\qed$
\end{proof}

\section{Dynamics on Nagano spaces}\label{sect_dynamics_self_opposite} 
The main goal of this section is Corollary~\ref{lem_geom_prop_visuel_4}, which will allow us to deduce several results for strongly $\MRr$-proper domains later on: first regarding the structure of their boundary (Proposition~\ref{prop_geom_prop_extremal_points}.(2)), and then rigidity results (Theorem~\ref{thm_group_non_hyperbolic_1}.(2)). This corollary relies on an analysis of the dynamical points of sequences of elements in~$G$, carried out in Lemma~\ref{lem_photon_contenu_dans_D}.

In Section~\ref{sect_cor_type}, we give an immediate consequence of Corollary~\ref{lem_geom_prop_visuel_4} for self-opposite Nagano spaces.

\subsection{Geometric properties of dynamical points}

Following \cite{frances2005lorentzian}, we define, for any sequence~$(g_k) \in G^\mathbb{N}$ and~$x \in \Fl(\g, \Theta)$, the set
\begin{equation*}
    D_{(g_k)}(x) \ \ := \bigcup_{(x_k) \in \Fl(\g, \Theta)^\mathbb{N}, \ x_k \rightarrow x}A(g_k \cdot x_k)~,
\end{equation*}
where we have denoted by~$A(y_k)$ the set of limit points of the sequence~$(y_k)$.

The following lemma uses the notion of~\emph{$\{\alpha\}$-contracting sequence} defined in Section~\ref{sect_theta_limit_set}, as well as the arithmetic distance~$d_H$ on~$\Fl(\g, \alpha)$ defined in Section~\ref{sect_arithm_distance}.

\begin{lem}\label{lem_photon_contenu_dans_D}
    Let~$(\g, \alpha)$ be an irreducible real-type Nagano pair of rank~$r$, and let~$G \in \mathcal{G}_{\{\alpha\}}$. Let~$(g_k) \in G^\mathbb{N}$ be an~$\{\alpha\}$-contracting sequence with respect to~$(a,b) \in \Fl(\g, \alpha) \times \Fl(\g, \alpha)^\opp$. Let~$x \in \Fl(\g, \alpha) \smallsetminus \hyp_b$ be such that~$d_H(x,a) \leq r-1$. Then~$D_{(g_k^{-1})}(x)$ contains a photon.
\end{lem}

\begin{proof} The result is well know in rank one, corresponding to~$(\g, \alpha) = (\sll(n, \mathbb{R}), \alpha_1)$, defining real projective space. Hence we may assume that~$r \geq 2$.

Recall, in the notation of Section~\ref{sect_lie_theory}, that~$G = K\exp(\overline{\mathfrak{a}}^+)K$; this is the so-called \emph{KAK} decomposition of~$G$.

Hence for all~$k \in \mathbb{N}$, there exist~$w_k, w_k' \in K$ and~$a_k \in \exp(\overline{\mathfrak{a}}^+)$ such that~$g_k = w_k a_k w_k'$. Up to taking a subsequence, we may assume that~$(w_k)$ and~$(w_k')$ converge in~$K$. We may thus assume that they are both equal to identity, and that~$a = \LP^+$ and~$b = \LP^\opp$.

Up to taking a subsequence, we may also assume that~$(g_k^{-1} \cdot x)$ converges to some point~$m \in \Fl(\g, \alpha)$. For all~$k \in \mathbb{N}$ we have

$$d_H(g_k^{-1} \cdot x, \LP^+)  =d_H(g_k^{-1} \cdot x, g_k^{-1} \cdot \LP^+) \leq r-1~,$$
so~$d_H(m,\LP^+) \leq r-1$. In particular, we have~$m \ne \LP^\opp$, so~$\exp(\uu^\opp)$ does not stabilize~$m$. Since~$\Ad(L) \cdot v^\opp$ generates~$\uu^\opp$ as a vector space (Fact~\ref{fact_kostant}) and~$\uu^\opp$ is abelian, this implies that there exists~$g \in L$ such that
\begin{equation}\label{eq_v_m_ne_m}
    \exp(\Ad(g) \cdot v^\opp) \cdot m \ne m~.
\end{equation}.
Let us set~$v := \Ad(g) \cdot v^\opp$. For all~$t \in \mathbb{R}$, we set
\begin{equation*}
    x_k(t) := \exp(t \Ad(g_k) v) \cdot x~.
\end{equation*}
Since~$g_k \in \exp(\aaa) \subset L$, for all~$k \in \mathbb{N}$ the points~$x_k$ and~$x$ are on a same photon~$\Phot_k$.

Note that~$d_H(x, a) = r-1$ implies that there exists~$v' \in \uu^\opp$ such that~$\Phot' := \overline{\exp(\mathbb{R}v) \cdot \LP^+}$ is a photon such that~$d_H(x', a) = r$ for all~$x' \in \Phot' \smallsetminus\{ x\}$. In particular, since~$(g_k^{-1})$ is~$\{\oppinv(\alpha)\}$-contracting with respect to~$(\LP^\opp, \LP^+)$, and since~$(\g,\alpha)$ has higher rank, it proves that~$m \in \hyp_{\LP^\opp}$. Hence, taking any norm~$\lvert \cdot \rvert $ on~$\uu^\opp$, we have~$\lvert \Ad(g_k)^{-1} \cdot v \rvert \rightarrow + \infty$. We thus set
\begin{equation*}
    t_k := \frac{t}{\lvert \Ad(g_k)^{-1} \cdot v \rvert} \ \text{ and } \ \overline{v} := \lim_{k\rightarrow + \infty} \frac{\Ad(g_k)^{-1} v}{\lvert \Ad(g_k)^{-1} \cdot v \rvert}~.
\end{equation*}
Since~$g_k$ is~$\{\alpha\}$-contracting with respect to~$(\LP^+,\LP^\opp)$, we must have~$ \Ad(g_k) v \rightarrow 0$ as~$k \rightarrow + \infty$, so~$x_k(t_k) \rightarrow x$. We then have:
\begin{equation}\label{eq_limit_gk_xk}
    g_k^{-1}\cdot x_k(t_k) = \exp(t_k \Ad(g_k)^{-1}v) g_k^{-1} \cdot x \longrightarrow \exp(t\overline{v}) \cdot m~.
\end{equation}
By Fact~\ref{fact_stab_u_std}.(5), the Hausdorff limit of~$(g_k^{-1} \cdot \Phot_k)$ is a photon~$\Phot$ through~$m$. But from~\eqref{eq_v_m_ne_m} and~\eqref{eq_limit_gk_xk}, is is clear (considering a Plücker triple of~$(\g, \alpha)$ for instance) that~$t \mapsto~\exp(t\overline{v}) \cdot m$ gives a parametrization of~$\Phot$ minus a point. Hence~\eqref{eq_limit_gk_xk} gives that~$\Phot \subset D_{(g_k^{-1})}(x)$.~$\qed$
\end{proof}

We can now prove:

\begin{cor}\label{lem_geom_prop_visuel_4} Let~$(\g, \alpha)$ be a real-type Nagano pair. Let~$\O$ be a strongly~$\MRr$-proper domain of~$\Fl(\g, \alpha)$. Let~$a \in \Lambda_{\{\alpha\}}(\Aut_G(\O))$. Then~$d_H(a,x) \geq \operatorname{rk}(\g, \alpha)$ for all~$x \in \O$. In particular, if~$\oppinv(\alpha) = \alpha$, then~$a \in \O^*$.
\end{cor}

\begin{proof} Let~$(g_k) \in \Aut_G(\O)^\mathbb{N}$ be an~$\{\alpha\}$-contracting sequence with respect to~$(a,b)$, where~$b \in \Fl(\g, \alpha)^\opp$. If there existed a point~$x \in \O$ such that~$d_H(x,a) \leq \operatorname{rk}(\g, \alpha)-1$, then there would also exist such a point~$x$ in~$\O \cap \hyp_b$. Then Lemma~\ref{lem_photon_contenu_dans_D} would give that~$D_{(g_k^{-1})}(x)$ contains a photon. But by openness of~$\O$, we have~$D_{(g_k^{-1})}(x) \subset \overline{\O}$. This would contradict the strong~$\MRr$-properness of~$\O$.~$\qed$
\end{proof}

\subsection{A corollary on the type} \label{sect_cor_type} In this section, we assume that~$\Fl(\g, \alpha)$ is self-opposite, i.e.\ $\oppinv(\alpha) = \alpha$, as defined in Section~\ref{sect_flag_mfds}. Given a subset~$F \subset \Fl(\g, \alpha)$, we denote by~$F^{3, * }$ the set of triples of pairwise transverse points in~$F$. Until the end of this section, we fix~$G := \Aut_{\{\alpha\}}(\g)$. Recall the standard affine chart~$\Affstdstd$ and its parametrization~$\varphistd$ defined in~\eqref{eq_standard_affine_chart} and~\eqref{eq_A_egal_exp}. The map 
\begin{equation*}\label{eq_def_opposition}
    \inver_{\{\alpha\}}:
        \Affstdstd \longrightarrow \Affstdstd; \ \
        \varphistd(X) \longmapsto \varphistd (-X)~.
\end{equation*}
induces a homeomorphism of~$\Affstdstd  \smallsetminus \hyp_{\LP^+}$ \cite{dey2024restrictions}, and thus acts on the set~$\mathcal{E}_{\alpha}$ of connected components of~$\Fl(\g, \alpha) \smallsetminus (\hyp_{\LP^+} \cup \hyp_{\LP^\opp})$.

The subgroup~$L$ of~$G$ acts on~$\mathcal{E}_{\alpha}$. If~$(x,y,z) \in \Fl(\g, \alpha)$ are three pairwise transverse points, then there exists a unique~$[g] \in G/L$ such that~$g \cdot (x,y) = (\LP^+, \LP^\opp)$. We denote by~$\typ(x,y,z)$ the~$L$-orbit of the connected component~$\mathcal{O}$ of~$\Fl(\g, \alpha) \smallsetminus (\hyp_{\LP^+} \cup \hyp_{\LP^\opp})$ containing~$g\cdot y$. 

The \emph{type} of a triple~$(x,y,z) \in \Fl(\g, \alpha)^{3,*}$ is by definition the orbit~$\typ(x,y,z) \in \mathcal{E}_{\alpha}/L$. It is~$G$-invariant. Since~$\inver_{\{\alpha\}}$ commutes with the action of~$L$ on~$\uu_{\{\alpha\}}^\opp$, it induces a bijection of~$\mathcal{E}_{\alpha} / L$, still denoted by~$\inver_{\{\alpha\}}$.

The following lemma is proven in \cite{galiay2025transverse} and expresses a restriction on the $\inver_{\{\alpha\}}$-invariant elements of $\mathcal{E}_{\alpha}$ when $\Fl(\g, \alpha)$ contains a domain~$\O$ with sufficently many elements in~$\partial \O \cap \O^*$.

\begin{lem}\label{lem_Omega_domain_prop_I} \cite{galiay2025transverse}
    Let~$G$ be a noncompact real semisimple Lie group and~$\Theta$ be a self-opposite subset of simple restricted roots of~$G$. Let~$\O \subset \Fl(\g, \alpha)$ be a domain. Then for any subset~$\Lambda \subset \partial \O \cap \O^*$ of pairwise transverse points and of cardinal~$\geq 4$, there exists an~$\inver_{\{\alpha\}}$-invariant element~$\mathcal{O} \in \mathcal{E}_{\alpha}$ such that~$\typ(a,b,c) = [\mathcal{O}] \in \mathcal{E}_{\alpha}/L$ for all~$(a,b,c) \in \Lambda^{3, *}$.
\end{lem}

Corollary~\ref{lem_geom_prop_visuel_4} and Lemma~\ref{lem_Omega_domain_prop_I} directly imply:

\begin{fact}\label{lem_Nagano_prop_I}
    Let~$(\g, \alpha)$ be an irreducible Nagano pair of real type such that~$\oppinv(\alpha) = \alpha$ and~$G \in \mathcal{G}_{\{\alpha\}}(\g)$. Let~$\O \subset \Fl(\g, \Theta)$ be an~$\MRr$-proper domain. Then for any subset~$\Lambda \subset \Lambda_{\{\alpha\}} (\Aut_G(\O))$ of pairwise transverse points and of cardinal~$\geq 4$, there exists an~$\inver_{\{\alpha\}}$-invariant element~$\mathcal{O} \in \mathcal{E}_{\alpha}$ such that~$\typ(a,b,c) = [\mathcal{O}] \in \mathcal{E}_{\alpha}/L$ for all~$(a,b,c) \in \Lambda^{3, *}$.
\end{fact}

Lemma~\ref{lem_Nagano_prop_I} holds in particular if~$\Aut_G(\O)$ is Zariski-dense in~$G$, or if it is nonelementary and~$\{\alpha\}$-Anosov in the sense of \cite{guichard2012anosov}. It is a strong topological restriction on the type of the limit set of groups preserving~$\MRr$-proper domains, in particular in the case where~$\mathcal{E}_{\alpha}$ has no~$\inver_{\{\alpha\}}$-invariant connected component. It implies in particular:

\begin{cor}\label{thm_Z_dense_Nagano_prop_I}
Let~$(\g, \alpha)$ be an irreducible real-type Nagano pair corresponding to items \emph{(i, $n \equiv 2 \mod 4$), (iii, $n$ odd), (iv, $p=q$ odd), (x, $n$ odd), (xi, $n$ odd), (xii)} of Table~\ref{table_nagano_full}. Let~$H \subset \Aut_{\{\alpha\}}(\g)$ be a subgroup such that~$\Lambda_{\{\alpha\}}(H)$ contains at least~$4$ pairwise transverse points. Then~$H$ does not preserve any $\MRr$-proper domain in~$\Fl(\g, \alpha)$.
\end{cor}

This result contrasts with the real projective case, where an irreducible subgroup~$\Gamma \leq \GL(V)$ preserves a proper domain in~$\mathbb{P}(V)$ if and only if the top eigenvalue of every proximal element of~$\Gamma$ is positive \cite{benoist2000automorphismes}.

\begin{proof}[Proof of Corollary~\ref{thm_Z_dense_Nagano_prop_I}]
    Indeed, in these cases~$\mathcal{E}_{\alpha}$ has no~$\inver_{\{\alpha\}}$-invariant connected component, see \cite{dey2022borel, dey2024restrictions, kineider2024connected} and \cite[Proof of Cor.\ 4.6]{galiay2025transverse}. The corollary is then just a consequence of Fact~\ref{lem_Nagano_prop_I}.~$\qed$
\end{proof}

\section{Examples}\label{sect_examples_kob_hyp}
In this section, we give elementary constructions of $\MRr$-proper dually convex domains arising from Anosov representations. We fix an irreducible real-type Nagano pair~$(\g, \alpha)$ and~$G \in \mathcal{G}_{\{\alpha\}}(\g)$.

\begin{lem}\label{lem_R_propre_ex} Let~$F \subset \Fl(\g, \alpha)^\opp$ be a Zariski-dense subset. Then any connected component of~$\Fl(\g, \alpha) \smallsetminus \bigcup_{z \in F}\hyp_z$ is~$\mathcal{R}$-proper. Moreover, their Kobayashi metric is a proper geodesic metric.
\end{lem}

\begin{proof} Let us prove the first point. We fix~$(G, \rho, V)$ an~$\{\alpha\}$-proximal triple for~$\g$. Let~$\O$ be a connected component of~$\Fl(\g, \alpha) \smallsetminus \bigcup_{z \in F}\hyp_z$. Note that~$\O$ is by definition dually convex.

By definition, one has~$F \subset \O^*$. Hence~$C_\O^\rho$ is a proper metric by Fact~\ref{fact_Z_dense}. Moreover~$\O$ is dually convex by construction. By Proposition~\ref{lem_key_lemma}, the pseudometric~$\kob_\O$ is a proper metric. Hence by Proposition~\ref{prop_kob_hyp_implies_R_proper}, the domain~$\O$ is~$\MRr$-proper. Since~$\kob_\O$ is also a length metric by Theorem~\ref{thm_kob_hyp_R_propre}, it is geodesic.~$\qed$
\end{proof}
 Note that in projecive space, the examples arising in Lemma \ref{lem_R_propre_ex} are even \emph{properly convex} in the classical projective sense; in particular they are proper. However, for general real-type Nagano spaces, the $\MRr$-proper domains arising from this construction are not proper:

\begin{lem}\label{lem_R_proper_not_proper} Assume that~$\oppinv(\alpha) = \alpha$. Let~$H \leq G$ be a Zariski-dense subgroup such that there exists~$t \in \mathcal{E}_{\alpha}/L$ such that
    \begin{equation*}
        \O_F^t := \{x \in \Fl(\g, \alpha) \mid \typ(a,x,b) = t \quad \forall (a,b) \in \Lambda_{\{\alpha\}}(H)^2 \text{ transverse}\}
    \end{equation*}
is nonempty and connected. Then~$\O_F^s$ is~$\MRr$-proper and non-proper.
\end{lem}

\begin{proof}
    By Lemma~\ref{lem_R_propre_ex}, the domain~$\O := \O_F^t$ is~$\MRr$-proper. Let us assume for a contradiction that it is proper. Then by Theorem~\ref{thm_Z_dense_Nagano_prop_I}, there exists~$\mathcal{O} \in t$ such that~$\inver_{\{\alpha\}}(\mathcal{O}) = \mathcal{O}$. In particular, we have~$\inver(t) = t$.
    
    Then~$\O^*$ has nonempty interior~$U$. Since~$U$ is a proper~$H$-invariant open subset of~$\Fl(\g, \alpha)$, by Fact~\ref{fact_autom_group_proper}, the group~$H$ acts properly on it. Hence we must have~$U \subset \Fl(\g, \alpha) \smallsetminus \bigcup_{z \in \Lambda_{\{\alpha\}}(H)}\hyp_z$. Let~$x \in U$ and let~$a,b \in \Lambda_{\{\alpha\}}(H)$ be two transverse points. There exists~$t_x \in \mathcal{E}_{\alpha}/L$ such that~$\typ(a,x,b) = t_x$. Since~$U \cap \O = \varnothing$, we have~$t \ne t_x$. 

    Since~$\O$ is open and~$a \in \Lambda_{\{\alpha\}}(H)$, we must have~$a \in \overline{\O}$. Hence there exists~$y \in \O$ such that~$\typ(y,x,b) = t_x$. Since~$U$ is open and~$H$-invariant, we also have~$\Lambda_{\{\alpha\}}(H) \subset \overline{U}$. Making~$x_n \in U$ converging to a point~$c \in \Lambda_{\{\alpha\}}(H) \smallsetminus \hyp_b$, we get a sequence~$(t_{x_n})$ of the finite set~$\mathcal{E}_{\alpha}/L$, which we can assume to be equal to a constant~$t'$. Hence for all~$n \in \mathbb{N}$, we have~$\typ(y,x_n,b) = t'$ with~$t' \ne s$. On the other hand, we have
    \begin{equation*}
       t' = \typ(y, c, b) = \inver (\typ(c, y, b)) = t~,
    \end{equation*}
contradiction.~$\qed$
\end{proof}

In Corollary~\ref{thm_Z_dense_Nagano_prop_I}, we observed that for certain pairs~$(\g,\alpha)$, the hypotheses of Lemma~\ref{lem_R_proper_not_proper} are \emph{never} satisfied. However, for others, this is indeed the case. For instance, if~$(\g,\alpha) = (\soo(2,q),\alpha_1)$ with~$q \geq 3$ (item~(viii, $p=1, q \geq 3$) in Table~\ref{table_nagano_full}), then~$\mathcal{E}_{\alpha}/L$ has cardinality~$2$. One of its two elements, denoted~$t_1$, corresponds to the \emph{time-like configuration} of three points, while the other,~$t_2$, corresponds to the \emph{space-like configuration}. The element~$t_2$ is~$\inver$-invariant, and there exist Zariski-dense subgroups~$\Gamma \to \SO(2,q)$ such that~$\O_{\Lambda_\Gamma}^{t_2}$ is non-empty and connected~\cite{smai2022anosov}. In this case, Lemma~\ref{lem_R_proper_not_proper} implies that~$\O_{\Lambda_\Gamma}^{t_2}$ is~$\MRr$-proper but not proper.

\section{Structure of the boundary}\label{sect_R_extremality}

The objects introduced in this section, together with their properties, are generalizations of well-known objects and properties from convex projective geometry. Their introduction originates in \cite{van2019rigidity} in the case of the Grassmannian. We subsequently generalized them in \cite{chalumeau2024rigidity, galiay2024rigidity} for Einstein universes and causal flag manifolds. In this section, we provide a general construction for real Nagano spaces. The reader may refer to the articles cited above for concrete examples, as well as to \cite{galiay2025convex}, where the general notions for Nagano spaces are introduced and illustrated.

\subsection{$\MRr$-faces}\label{sect_R_faces}
We fix an irreducible Nagano pair~$(\g, \alpha)$.   The following definition extends to~$\overline{\O}$ the relation~$\phoo$, introduced on a proper domain~$\O\subset \Fl(\g,\alpha)$ in Section~\ref{sect_definition_kobayashi}:  

\begin{definition}\label{def_R_face}  
Let~$\O \subset \Fl(\g,\alpha)$ be a proper domain, and let~$a,b \in \overline{\O}$.  
We say that~$a \phoo_{\MRr} b$ (or simply~$a \phoo b$ if the context is clear) if there exists a photon~$\Phot$ through~$a$ and~$b$, such that~$a$ and~$b$ belong to the same connected component of the relative interior of~$\Phot \cap \overline{\O}$ in~$\Phot$.  

The~$\MRr$-face of~$a$, denoted by~$\Fl_\O^\MRr(a)$, is the set of points~$c \in \partial \O$ for which there exist~$N \in \mathbb{N}$ and a sequence~$a_0 = a, a_1, \dots, a_N = c \in \partial \O$ such that for all~$0 \leq i < N$, we have~$a_i \phoo_\MRr a_{i+1}$.  

A point~$a \in \partial \O$ is said to be \emph{$\MRr$-extremal} if~$\Fl_\O^\MRr(a) = \{ a \}$.    
\end{definition}

\begin{prop}\label{cor_face_incluse_hyp}Let ~$\O \subset \Fl(\g,\alpha)$ be an~$\MRr$-proper domain contained in an affine chart. Let~$a \in \partial \O$ and let~$b \in \O^*$ be such that~$a \in \hyp_b$. One has~$\Fl_{\O}^{\mathcal{R}}(a) \subset \hyp_b$.
\end{prop}

\begin{proof} By the definition of the~$\MRr$-face of~$a$, it suffices to prove that for any photon~$\Phot$ through~$a$ such that~$\Phot \cap \partial \O$ contains a neighborhood of~$a$ in~$\Phot$, one has~$\Phot \subset \hyp_b$. We fix such a photon~$\Phot$.

By assumption~$\O$ is contained in an affine chart~$\Affstd_\xi$, for some~$\xi \in \O^*$. Let us first assume that~$a \notin \hyp_\xi$.

Let~$(G, \rho, V)$ be a Plücker triple for~$\Fl(\g,\alpha)$. We consider the affine chart~$\Affstd_\infty$ of~$\mathbb{P}(V)$ defined by~$\Affstd_\infty = \mathbb{P}(V) \smallsetminus \iota_\rho^\opp(\xi)$ and endow it with its canonical affine structure. Since~$a \notin \hyp_\xi$, one has~$b \ne \xi$, so~$\iota_\rho^\opp(b)$ is an affine hyperplane of~$\Affstd_\infty$ given as the kernel of an affine map~$f: \Affstd_\infty \rightarrow \mathbb{R}$. Since~$\O$ is connected and~$b \in \O^*$, we have~$\iota_\rho(\O) \cap \iota_\rho^\opp(b) = \varnothing$, so we may assume that~$\O \subset \{f > 0\}$.

By Proposition~\ref{prop_photons_egal_proj_lines}, the set~$\iota_\rho(\Phot)$ is a projective line of~$\mathbb{P}(V)$. Hence~$D := \iota_\rho(\Phot) \cap \Affstd_\infty$ is an affine line of~$\Affstd_\infty$. An open subset of it is contained in~$\iota_\rho(\partial \O)$, thus in~$\{f\geq 0\}$. But we also have~$f(\iota_\rho(a)) = 0$, where~$\iota_\rho(a) \in D$. Hence we clearly have~$D \subset \iota_\rho^\opp(b)$. taking the closure of~$D$, this proves that~$\Phot \subset \hyp_b$.

Now assume that~$a \in \hyp_\xi$ for all~$\xi \in \O^*$. If~$\Phot \subset \hyp_\xi$ for all~$\xi \in \O^*$, then the Proposition is proven. If not, then consider~$\xi' \in \O^*$ such that~$\Phot \not\subset \hyp_\xi$. By thefinition of~$\Phot$, there exists a point~$a'$ in the relative interior of $ (\Phot \cap \partial \O)$ in~$\Phot$ which is different from~$a$. Since~$\Phot \not\subset \hyp_{\xi'}$, by Fact~\ref{lem_photons_affine}, one has~$a \notin \hyp_{\xi'}$. Hence we can apply the above argument replacing~$(a, \xi)$ with~$(a', \xi')$, and we get that~$\Phot \subset \hyp_b$.~$\qed$

\end{proof}

\subsection{Convex hull and~$\MRr$-extremal points}\label{sect_convex_hull} In this section, we fix~$(\g,\alpha)$ an irreducible real-type Nagano pair, and ~$\O \subset \Fl(\g,\alpha)$ a proper domain. In general, given a projective~$\{\alpha\}$-proximal triple~$(G, \rho, V)$ of~$\g$, one can define the convex hull of~$\iota_\rho(\O)$ in~$\mathbb{P}(V)$. However, this convex hull is not necessarily open in~$\mathbb{P}(V)$. Nevertheless, as we will see in this section, if~$(G, \rho, V)$ is a Plücker triple for~$(\g, \alpha)$, then thanks to Proposition~\ref{prop_photons_egal_proj_lines}, this property holds. This will allow for a deeper study of the properties of~$\O$.

Let~$(G, \rho, V)$ be a Plücker triple for~$\Fl(\g,\alpha)$. There exists~$\xi_0 \in \Fl(\g,\alpha)^\opp$ such that~$\hyp_{\xi_0} \cap \overline{\O} = \varnothing$. Then~$\iota_\rho(\overline{\O}) \cap \iota_\rho^{\opp}(\xi_0) = \varnothing$, which means that~$\iota_\rho(\O)$ is proper in~$\mathbb{P}(V)$. Since it is connected, one can lift it to a proper connected cone~$\widetilde{\iota_\rho(\O)} \subset V \smallsetminus \{ 0 \}$. Then its convex hull~$\widetilde{\mathcal{O}}_{\O}^\rho := \mathrm{Conv}(\widetilde{\iota_\rho(\O)})$ in~$V$ is a properly convex cone of~$V$, a priori not necessarily open. We define 
\begin{equation*}\label{eq_def_convex_hull_O}
    \mathcal{O}_{\O}^\rho := \mathbb{P}(\widetilde{\mathcal{O}}_{\O}^\rho)~.
\end{equation*} 
It is a properly convex subset of~$\mathbb{P}(V)$, and it does not depend on the choice of affine chart containing~$\overline{\O}$. In particular, the domain~$\mathcal{O}_{\O}^\rho$ is~$\rho(\Aut_G(\O))$-invariant. 

\begin{definition}\label{def_convex_hull_plucker}The properly convex domain~$\Oo^\rho$ is uniquely defined by~$\O$ and~$(G, \rho, V)$, and called the \emph{convex hull of~$\O$ in~$\mathbb{P}(V)$} (by construction, the set~$\Oo^\rho$ and contains~$\iota_\rho(\Omega)$).
\end{definition}

For the rest of this section, we fix~$\vv_0 \in V^{\omega_\alpha} \smallsetminus \{ 0 \}$.

\begin{prop}\label{prop_conv_nonempty} The set~$\mathcal{O}_\O^\rho$ is open. Moreover, if~$\O$ is dually convex, then~$\iota_\rho(\partial \O) \subset \partial \Oo^\rho$ and~$\iota_\rho(\O)$ is a connected component of~$\Oo^\rho \cap \iota_\rho(\Fl(\g,\alpha))$.
\end{prop}

\begin{proof}
    For the openness, by definition of~$\mathcal{O}_\O^\rho$, it suffices to prove that for all~$x \in \iota_\rho(\O)$, there exists a neighborhood~$\mathcal{V}$ of~$x$ contained in~$\mathcal{O}_\O^\rho$.

    Now let~$x, y \in \O$, and~$(x_0 = x, \dots, x_N = y ) \in \mathcal{C}_{x,y}(\O)$. For all~$i$, by Propostion~\ref{prop_photons_egal_proj_lines}, the points~$x_i$ and~$x_{i+1}$ lie on the interior of a common projective segment contained in~$\iota_\rho(\O) \subset \mathcal{O}_\O^\rho$. Thus~$x,y$ are on the same open face of~$\mathcal{O}_\O^\rho$. This proves that~$\iota_\rho(\O)$ is contained in a face of~$\mathcal{O}_\O$.

    Now we need to use the following result proven by Zimmer \cite[Lem.\ 4.7]{zimmer2018proper}, and due to the fact that~$\O$ is open and~$\rho$ is irreducible: \emph{there exist~$y_1, \dots , y_N \in \O$ such that~$V = \iota_\rho(y_1) \oplus \dots\oplus \iota_\rho(y_N)$.}
    
    Hence~$\iota_\rho(\O)$ cannot be contained in a face of~$\Oo^\rho$ of codimension~$\geq 1$. But the only face of~$\Oo^\rho$ of codimension~$0$ is~$\operatorname{int}(\Oo^\rho)$ of~$\mathcal{O}_\O^\rho$. But then, the set~$\operatorname{int}(\Oo^\rho)$ is a convex set containing~$\iota_\rho(\O)$ and contained in the convex hull~$\Oo^\rho$ of~$\iota_\rho(\O)$. By definition of the convex hull, we must have~$\Oo^\rho = \operatorname{int}(\Oo^\rho)$. Thus~$\mathcal{O}_\O^\rho$ is open. 
    
    Now let us prove the second assertion. Assume that~$\O$ is dually convex. Then for all~$a \in \partial \iota_\rho(\O) = \iota_\rho(\partial \O)$, there exists~$\xi \in \O^*$ such that~$a \in \iota_\rho^\opp(\xi)$. Let~$f \in V^* \smallsetminus \{ 0 \}$ be a lift of~$\iota_\rho^\opp(\xi)$; Since~$\iota_\rho(\O) \cap \iota_\rho^\opp(\xi) = \varnothing$, one has~$f(x) \ne 0$ for all~$x \in \widetilde{\iota_\rho(\O)} \smallsetminus \{ 0 \}$. By connectedness of~$\O$, we may assume that~$f(x) >0$ for all~$x \in \widetilde{\iota_\rho(\O)} \smallsetminus \{ 0 \}$. Then taking the convex envelope one has~$f(x) > 0$  for all~$x \in \widetilde{\Oo^{\rho}}$. Thus~$\Oo^\rho \cap \iota_\rho^\opp(\xi) = \varnothing$. In particular, since~$a \in \iota_\rho^\opp(\xi)$, one has~$a \notin \Oo^{\rho}$. On the other hand, one has~$a \in \overline{\iota_\rho(\O)} \subset \overline{\Oo^{\rho}}$. Thus~$a \in \partial \Oo^{\rho}$. 

    We have just proven that~$\partial \iota_\rho(\O) \subset \partial \Oo^{\rho} \cap \iota_\rho(\Fl(\g,\alpha))$. Thus~$\iota_\rho(\O)$ is closed in~$\Oo^\rho \cap \iota_\rho(\Fl(\g,\alpha))$. It is also open, so it is a connected component of~$\Oo^{\rho} \cap \iota_\rho(\Fl(\g,\alpha))$.~$\qed$
\end{proof}

In projective space~$\mathbb{P}(\mathbb{R}^n)$, the extremal points of a properly convex open subset generate~$\mathbb{R}^n$, by Krein--Milman's Theorem. In real-type Nagano spaces, the existence of Plücker triples allows to recover this property for~$\MRr$-extremal points of proper domains:

\begin{lem}\label{lem_existence_extreme_points}
The (projective) extremal points of~$\Oo$ are images of~$\MRr$-extremal points of~$\O$. In particular, there exist~$x_1, \dots, x_D \in \Rr(\O)$ such that~$V = \iota_\rho (x_1) \oplus \cdots \oplus \iota_\rho (x_D)$.
\end{lem}

\begin{proof}[Proof] Let~$\mathcal{O} := \mathcal{O}_{\O}^{\rho}$ be the convex hull of~$\iota_\rho(\O)$ in~$\mathbb{P}(V)$ defined in Section~\ref{sect_convex_hull}. Given two distinct points~$x,y \in \overline{\mathcal{O}}$, we denote by~$(x,y)$ the unique connected component of~$L \smallsetminus \{x,y\}$ contained in~$\mathcal{O}$, where~$L$ is the unique projective line containing~$x$ and~$y$.

Let~$x$ be an extremal point of~$\mathcal{O}$. Then by definition of~$\mathcal{O}$, there exists~$z \in \overline{\O}$ such that~$x = \iota_\rho(z)$. Moreover,~$z \notin \O$ because~$\iota_\rho(\O) \subset \mathcal{O}$. Then~$z \in \partial \O$. If~$z$ is not~$\mathcal{R}$-extremal, then there exists a photon~$\Phot$ and~$a,b \in \Phot \cap \partial \O$ such that~$a,z,b$ are distinct and aligned in this order. Then, by Proposition~\ref{prop_photons_egal_proj_lines},~$z$ is included in the nontrivial projective interval~$(\iota_\rho(a), \iota_\rho(b))$ of~$\overline{\mathcal{O}}$, which contradicts the fact that~$x$ is projectively extremal. 

Hence every~$\MRr$-extremal point of~$\mathcal{O}$ is the image of an~$\MRr$-extremal point of~$\partial \O$. Thus by Krein--Milman's theorem, the convex set~$\mathcal{O}$ is the convex hull (in a suitable affine chart) of~$\iota_\rho(\Rr(\O))$. The result then follows by openness of~$\mathcal{O}$ (Proposition~\ref{prop_conv_nonempty}).~$\qed$
\end{proof}

\subsection{$\MRr$-extremal points are dynamical points}\label{sect_geometric_prop_R_extr}
The goal of this section is to prove:

\begin{prop}\label{lem_extremal_attracting}
    Let~$\O \subset \Fl(\g, \alpha)$ be an~$\MRr$-proper photon-convex domain, and~$G \in \mathcal{G}_{\{\alpha\}}(\g)$. Assume there exists~$H \leq \Aut_G(\O)$ such that~$\partial \O \subset \bigcup_{x \in \O} \overline{H \cdot x}$. Then 
    \begin{equation*}
        \Extr(\O) \subset \Lambda_{\{\alpha\}}(H)~.
    \end{equation*}
\end{prop}

\subsubsection{Dynamics of~$\MRr$-extremal points}

We fix an irreducible Nagano pair~$(\g, \alpha)$ of real type. The proof of Lemma~\ref{lem_extremal_attracting} follows the strategy of \cite[Thm 7.4]{van2019rigidity}. 
 
 For a proper domain~$\O$ of~$\Fl(\g, \alpha)$ and~$x,y \in \O$, and~$m \in \mathbb{N}^*$. For~$\gamma = (x_0, \dots, x_m) \in \mathcal{C}_{x,y}^{m}(\O)$, we define
 $$\lenk(\gamma) :=\sum_{i=0}^{m-1} \ro_\O(x_i, x_{i+1})~.$$ 

We moreover define
\begin{equation*}
    \kob_{\O}^m (x,y) := \inf\big{\{}\lenk(\gamma) \big|  \, \gamma \in \mathcal{C}_{x,y}^{m}(\O)\big{\}}~.
\end{equation*}
The quantity~$\kob_{\O}^m (x,y)$ is finite if and only if the set~$\mathcal{C}_{x,y}^{m}(\O)$ is nonempty. The map~$\kob_{\O}^m: \O \times \O \rightarrow \Rf \cup \{ \infty \}$ is~$\Aut(\O)$-invariant. The sequence~$(\kob_{\O}^m(x,y))_{m \in \mathbb{N}}$ is nonincreasing, eventually finite,  and one has~$\kob_{\O}(x,y) = \lim_{m \rightarrow + \infty} \kob_{\O}^m (x,y)$. We moreover have:

\begin{lem}\label{lem_asymptotic_behav_extremal}
   Let~$\O \subset N$ be an~$\MRr$-proper photon-convex domain. Let~$a,b \in \overline{\O}$. Assume that there exist~$(x_k), (y_k) \in \O^{\mathbb{N}}$ such that~$x_k \rightarrow a$ and~$y_k \rightarrow b$, and such that there exist~$m \in \mathbb{N}$ and~$M>0$ such that~$\kob_{\O}^m (x_k, y_k) \leq  M$ for all~$k \in \mathbb{N}$. Then~$\Fl_\O^{\MRr}(a) = \Fl_\O^{\MRr}(b)$. In particular, if~$a \in \Rr(\O)$, then~$a=b$.
\end{lem}
\begin{proof}[Proof] Note that we just need to prove that~$b \in \Fl_\O^{\MRr}(a)$. For any~$k \in \mathbb{N}$, let~$\gamma_k = (x_k^0 := x_k ,\dots, x_k^{m} := y_k)$ in~$\mathcal{C}^{m}_{x_k,y_k}(\O)$ be such that 
   ~$$\sum_{i=0}^{m-1} \ro_{\O}(x_i^k, x_{i+1}^k)  \leq \kob_{\O}^m (x_k, y_k) + 1 \leq M + 1~,$$
    the first inequality holding because of~\eqref{eq_longueur_chaine_1}.
    Then, one has~$\ro_{\O}(x_k^i, x_k^{i+1}) \leq M + 1$ for all~$0 \leq i \leq  m-1$. Hence one can assume that~$m = 1$, and the lemma follows by induction.
    
    Let us then assume that~$m=1$. For all~$k$, the two points~$x_k$ and~$y_k$ lie in the same connected component of the intersection~$I_k := \Phot_k \cap \O$ of a photon~$\Phot_k$ with~$\O$. Let~$c_k,d_k$ be the endpoints of~$I_k$ such that~$c_k, x_k, y_k, d_k$ are aligned in this order. For all~$k$, let~$\isl^k$ be a projective parametrization of~$\Phot_k$. Then there exist~$r_k, s_k, t_k, u_k \in \mathbb{P}(\mathbb{R}^2)$, aligned in this order, such that~$ \isl^k(r_k) = c_k$, $\isl^k(s_k) = x_k$, $
            \isl^k(t_k) = y_k$,  and~$\isl^k(u_k) = d_k$. Then 
    \begin{equation*}
       \log (r_k : s_k : t_k:u_k) = \ro_{\O}(x_k, y_k) \leq M+1~.
    \end{equation*}
Up to taking a subsequence, we may assume that there exist~$c, d \in \partial \O$ and~$r,s,t,u \in \mathbb{P}(\mathbb{R}^2)$ such that~$c_k \rightarrow c$ and~$d_k \rightarrow d$ and~$(r_k,s_k, t_k, u_k) \rightarrow (r,s,t,u)$ as~$k \rightarrow + \infty$. For all~$k \in \mathbb{N}$, the points~$c_k, x_k, y_k, d_k$ lie on the same photon in this order, so~$c, a,b, d$ lie on the same photon, in this order. Moreover, since for all~$k\in \mathbb{N}$ we have~$(c_k, d_k) \subset \overline{\O}$, we have~$(c,d) \subset \overline{\O}$. We have~$\log (r : s : t:u) \leq M+1$; thus either~$s,t \in (r, u)$ or~$s=t$. Thus either~$a,b \in (c,d)$ or~$a=b$. We have just proven that~$b \in \Fl_\O^{\MRr}(a)$.~$\qed$
\end{proof}

We define the \emph{photon-diameter} of a compact subset $\compact$ of $\O$ as

$$\operatorname{pdiam} (\compact) := \sup \left\{\ro_{\O}(x,y) \mid x,y \in \compact, x \phoo y \right\}~.$$ 

Then one has:
\begin{lem}\label{lem_diamond_diameter}
    Let $D$ be a realization of~$\mathbb{X}(\g, \alpha)$ such that $\overline{D} \subset \O$. Then for any $x,y \in D$, there exists an $\operatorname{rk}(\g, \alpha)$-chain from $x$ to $y$ with $\kob_{\O}$-length less than or equal to $\operatorname{rk}(\g, \alpha) \operatorname{pdiam}(\overline{D})$. In particular, 
    $$\kob_{\O}^{\operatorname{rk}(\g, \alpha)} (x,y) \leq \operatorname{rk}(\g, \alpha) \operatorname{pdiam}(\overline{D})~.$$
\end{lem}
\begin{proof}[Proof]
    We know by Lemma~\ref{lem_chain_in_realization} that there exists an $\operatorname{rk}(\g, \alpha)$-chain $\gamma$ from~$x$ to $y$ contained in $D$. By Proposition~\ref{lem_key_lemma}, one has~$\kob_{\O}^{\operatorname{rk}(\g, \alpha)} (x,y) \leq \lenk(\gamma)  \leq \operatorname{rk}(\g, \alpha) \operatorname{pdiam}(\overline{D})$.~$\qed$
\end{proof}

\begin{lem}\label{lem_majoration_N_on_cpt}
    Let $\O \subset \Fl(\g, \alpha)$ be an~$\MRr$-proper photon-convex domain, and let $\mathcal{K} \subset \O$ be a compact subset. Then there exist $m \in \mathbb{N}$ and $M > 0$ such that~$\kob_{\O}^m (x,y) \leq M$ for any $x, y \in \mathcal{K}$. 
\end{lem}
\begin{proof}[Proof]
First choose, for any $x \in \mathcal{K}$, a realization $\mathsf{D}_x$ of~$\mathbb{X}(\g, \alpha)$ of diameter less than 1 for $\kob_{\O}$, and containing $x$. Such a realization exists by Remark~\ref{rmk_realizations}.(2) and Proposition~\ref{prop_generate_topo_std}. Now chose, inside of~$\mathsf{D}_x$, a realization $D_x$ of~$\mathbb{X}(\g, \alpha)$ of diameter less than 1 for $\kob_{\mathsf{D}_x}$, and containing $x$. Then for all~$y,z \in D_x$ such that~$y \phoo z$, we have~$\ro_\O(y,z) \leq \ro_{\mathsf{D}_x}(y,z) = \kob_{\mathsf{D}_x}(y,z) \leq 1$. The equality~$\ro_{\mathsf{D}_x}(y,z) = \kob_{\mathsf{D}_x}(y,z)$ comes from Theorem~\ref{thm_kob_hyp_R_propre}. Hence~$\operatorname{pdiam}(D_x) \leq 1$.

For any $x,y \in \mathcal{K}$, there is a minimal $m(x,y) \in\mathbb{N}$ such that~$M(x,y) := \kob_{\O}^{m(x,y)}(x,y) < + \infty$. 

Let $(z,z') \in D_x \times D_y$. There exists $\gamma\in \mathcal{C}_{x,y}^{m(x,y)}(\O)$ such that $\lenk(\gamma)\leq M(x,y)+1$. Besides, by Lemma~\ref{lem_diamond_diameter}, there exist $\gamma_1 \in \mathcal{C}_{z,x}^{\operatorname{rk}(\g, \alpha)}(\O)$ and $\gamma_2 \in \mathcal{C}_{y,z}^{\operatorname{rk}(\g, \alpha)}(\O)$, such that
$$\lenk(\gamma_1), \lenk(\gamma_2) \leq \operatorname{rk}(\g, \alpha)~.$$ 
If $m_0 := m(x,y) + 2 \nG$ and $\gamma_3$ is the concatenation of $\gamma_1$, $\gamma$ and $\gamma_2$, then $\gamma_3 \in \mathcal{C}_{z,z'}^{m_0}(\O)$ and for any $m \geq m_0$ one has 
\begin{equation*}
   \kob^{m}_{\O}(z,z') \leq \kob^{m_0}_{\O}(z,z') \leq \lenk(\gamma_3) = \lenk(\gamma_1) + \lenk(\gamma) + \lenk(\gamma_2) \leq 2\operatorname{rk}(\g, \alpha) + M(x,y) + 1~.
\end{equation*} 

Let $F \subset \mathcal{K}$ be a finite subset such that $\mathcal{K} \subset \bigcup_{x \in F} D_x$. Then for all $z,z' \in \mathcal{K}$, one has $\kob_{\O}^m (z,z') \leq M$, where~$m := \max \{ m(x,y) \mid x,y \in F \} + 2\operatorname{rk}(\g, \alpha)$ and~$M := \max \{ M(x,y) \mid x,y \in F \} + 2\operatorname{rk}(\g, \alpha)+1$.~$\qed$
\end{proof}

Proposition~\ref{lem_extremal_attracting} follows directly from the property satisfied by~$H$ and the following lemma:

\begin{lem}\label{lem_analogue_vey}
    Let~$\O$ be an~$\MRr$-proper photon-convex domain of~$\Fl(\g,\alpha)$, and let~$a \in \Rr(\O)$. If there exist~$(g_k) \in \Aut_G(\O)^{\mathbb{N}}$ and~$x \in \O$ such that~$g_k\cdot x \rightarrow a$, then for every compact subset~$\compact \subset \O$, one has~$g_k \cdot \compact \rightarrow \{a\}$ for the Hausdorff topology. In particular, this sequence~$(g_k)$ is~$\{\alpha\}$-contracting.
\end{lem}

\begin{proof}
Let~$y \in \O$ and~$m \in \mathbb{N}$ such that~$\kob_{\O}^m(x,y) < + \infty$. Then, by~$\Aut_G(\O)$-invariance of~$\kob_{\O}^m$, one has 
\begin{equation*}
    \kob_{\O}^m(g_k \cdot x, g_k \cdot y) = \kob_{\O}^m(x, y)
\end{equation*}
for all~$k \in \mathbb{N}$. Thus by Lemma~\ref{lem_asymptotic_behav_extremal}, we have~$g_k \cdot y \rightarrow a$. This holds for all~$y \in \O$.

Let~$\compact \subset \O$ be a compact subset. If the sequence~$g_k \cdot \compact$ does not converge to~$\{a\}$ for the Hausdorff topology, then there is a neighborhood~$\mathcal{V}$ of~$a$ in~$\Fl(\g, \alpha)$ and a sequence~$(y_k) \in \compact^{\mathbb{N}}$ such that~$g_k \cdot y_k \notin \mathcal{V}$ for all~$k \in \mathbb{N}$. Since~$\compact$ is a compact subset of~$\O$, up to taking a subsequence we may assume that there exists~$y \in \compact$ such that~$y_k \rightarrow y$. Moreover, by Lemma~\ref{lem_majoration_N_on_cpt}, there exists~$m \in \mathbb{N}$ such that~$(\kob_\O^m(y_k, y))$ is bounded.  Hence~$(\kob_{\O}^m(g_k \cdot y_k, g_k \cdot y))$ is bounded, so by Lemma~\ref{lem_asymptotic_behav_extremal}, we have~$g_k \cdot y \rightarrow a$. But this is impossible, since we assumed that~$y_k \notin \mathcal{V}$ for all~$k$. Hence~$g_k \cdot \compact \rightarrow \{ a\}$ for the Hausdorff topology.~$\qed$
\end{proof}

\subsection{$\MRr$-extremal points are in the dual} The following proposition is a consequence of Proposition~\ref{lem_extremal_attracting}:

\begin{prop}\label{prop_geom_prop_extremal_points}
    Let~$(\g, \alpha)$ be an irreducible real-type Nagano pair. Let~$\O \subset \Fl(\g, \alpha)$ be a strongly~$\MRr$-proper, photon-convex, almost-homogeneous domain.
    Then one has~$d_H(x,a) \geq r$ for all~$x \in \O$. In particular, if~$\Fl(\g, \alpha)$ is self-opposite, then~$a \in \O^*$.

    If~$\O$ is moreover proper, then there exists~$\xi \in \O^*$ such that~$\pos^{\{\alpha\}, \{\oppinv(\alpha)\}}(a, \xi) = \overline{\id}$.
\end{prop}

Note that, in Proposition~\ref{prop_geom_prop_extremal_points}, if~$\O$ is proper, than by almost-homogeneity, it is also dually convex (Fact~\ref{prop_zimmer_dual_convex}), so by Lemma~\ref{lem_continuité_intersection_photon}, it automatically photon-convex.

\begin{proof} The first assertion follows from Corollary~\ref{lem_geom_prop_visuel_4} and Proposition~\ref{lem_extremal_attracting}.

Let us now prove the last assertion. By Proposition~\ref{lem_extremal_attracting}, we have~$a \in \Lambda_{\{\alpha\}}(\Aut_G(\O))$. Thus there exist a sequence~$(g_k) \in \Aut_G(\O)^\mathbb{N}$ and a point~$b \in \Fl(\g, \alpha)^\opp$ such that~$g_k \cdot x \rightarrow a$ uniformly on compact subsets of~$\Fl(\g, \alpha) \smallsetminus \hyp_b$. Since~$\O$ has nonempty interior, there exists~$x \in \O \smallsetminus \hyp_b$. Hence~$g_k \cdot x \rightarrow a$.  Since~$\mathsf{C}_{\overline{w_0}}^{(\{\alpha\}, \{\oppinv(\alpha)\})}(b)$ is dense in~$\Fl(\g, \alpha)^\opp$ and~$\O^*$ has nonempty interior (see Section~\ref{sect_generalities_proper_domains}), we can fix~$y \in \O^*$ such that~$w_0 \in \pos^{(\{\oppinv(\alpha)\}, \{\oppinv(\alpha)\})}(b,y)$. Now let~$a' \in \Fl(\g, \alpha)$ such that~$\pos^{(\{\alpha\}, \{\oppinv(\alpha)\})}(a', y) = \overline{\id}$. Then necessarily by Lemma~\ref{lem_positions_weyl} one has~$\pos^{(\{\alpha\}, \{\oppinv(\alpha)\})}(a', b) = \overline{w_0}$. Hence we have~$g_k \cdot a' \rightarrow a$. On the other hand, up to taking a subsequence we may assume that~$(g_k \cdot y)$ converges to some~$c \in \O^*$. For all~$k \in \mathbb{N}$, we have~$\pos^{(\{\alpha\}, \{\oppinv(\alpha)\})}(g_k \cdot a', g_k \cdot y) = \overline{\id}$, so by \cite[Lem.\ 3.15]{kapovich2017dynamics}, we can take the limit and get~$\pos^{(\{\alpha\}, \{\oppinv(\alpha)\})}(a,c) = \overline{\id}$. The result follows.~$\qed$
\end{proof}

\section{Explicit computation for realizations of the noncompact dual}\label{sect_geod_sym_domains} The content of this section is the proof of Theorem~\ref{prop_calcul_geod}. More precisely, we prove:
\begin{thm}\label{thm_calcul_kob}
    Let~$(\g, \alpha)$ be an irreducible real-type Nagano pair of rank~$r \geq 1$, and let~$\O \subset \Fl(\g, \alpha)$ be a realization of~$\mathbb{X}(\g, \alpha)$. Then:
    
    (1). Any pair of points can be joined by a geodesic $r$-chain for~$\kob_\O$. Moreover, for any $\{\alpha\}$-proximal representation~$(V,\rho)$ of~$G$ of highest weight~$\chi$, we have~$C_\O^\rho =\chi(h_\alpha)\kob_\O$. 
    
    (2). The pullback of the Kobayashi metric restricted to a flat of~$\mathbb{X}(\g, \alpha)$ coincides with the integration of the $L^1$-norm on the flat.
\end{thm}
We will prove point~(1) in Sections~\ref{sect_slef_op_case} and~\ref{sect_non_self_op_case}, distinguishing the self-opposite and non self-opposite cases. We will prove point~(2) in Section~\ref{sect_infinitesimal_form}.

We fix~$(\g, \alpha)$ a rank-$r$ irreducible Nagano pair and $G \in \mathcal{G}_{\{\alpha\}}(\g)$. Our proof relies on the existence of chains contained in a flat of a realization~$\O$ of~$\mathbb{X}(\g, \alpha)$ (for its symmetric-space structure); such sequences yield two boundary points~$\xi, \xi'$ at which both the Caratheodory metrics and the Kobayashi metric are simultaneously realized, which is why their~$\kob_\O$-length equals both the Kobayashi and the Caratheodory metrics.

\begin{rmk}
    In \cite{falbel2025hilbert}, Falbel--Guilloux--Will compute the Caratheodory metric induced by a Plücker embedding for \emph{causal flag manifolds} (items (iii), (viii, $\min(p,q)=2$), (x), (xi), and (xii)).
\end{rmk}

The computation of the Kobayashi metric of a realization of the noncompact dual of~$\mathbb{X}(\g, \alpha)$ does not depend on the chosen realization, by Fact~\ref{prop_properties_kobayashi_metric}. In this section, it will be convenient to consider the realization given in the next section.

In this section, we will extensively use the  notation of Sections~\ref{sect_construction_symetri_structure} and~\ref{sect_embedding_concompact_dual}. In particular, recall that we denote by~$\sigma_0$ the Cartan involution of~$\g$ respecting the grading~$\g  = \g_{-1} \oplus \g_0 \oplus \g_1$, and~$K_0$ the fixed point set of~$\sigma_0$ in~$G$. Recall that the~$(\beta_i)_{1 \leq i\leq r}$ are the strongly orthogonal roots defined in Section~\ref{sect_embedding_concompact_dual}, and~$(h_{\beta_i})$ their coroots as defined in Section~\ref{sect_lie_theory}; those roots determine elements~$s_1, \dots, s_r$ of the Weyl group of~$G$, as well as a family~$(E_i)_{1\leq i \leq r}$, all also defined in Section~\ref{sect_embedding_concompact_dual}. We define, for all~$1 \leq i \leq r$:
\begin{equation}\label{eq_def_F_i}
    F_i := - \sigma_0(E_i)~.
\end{equation}
Then for all~$1 \leq i \leq r$, the triple~$(E_i, h_{\beta_i}, F_i)$ is an~$\sll_2$-triple.

Recall the standard embeddings~$\ncd, \ncdn$ of~$\mathbb{X}(\g, \alpha)$ defined in~\eqref{eq_def_ncd}, and
\begin{equation*}
    G^* = \operatorname{Aut}_G(\ncd(\mathbb{X}(\g, \alpha)))~.
\end{equation*}

\subsection{The nonbounded realization}\label{ex_nonbounded_realizatiob} In this section we describe explicitly a realization of~$\mathbb{X}(\g, \alpha)$ which will be convenient for the computation of the Kobayashi and Caratheodory metrics. This realization is contained in~$\Affstdstd$ (according to Lemma~\ref{lem_LP^+_in_O_dual} below) but not bounded in it, contrary to~$\ncd(\mathbb{X}(\g, \alpha))$.

Let~$e := \sum_{i=1}^r E_i$, $f := \sum_{i=1}^r F_i$ and~$h := \sum_{i=1}^r h_{\beta_i} = H_0$. Then~$(e,h,f)$ is an~$\sll_2$-triple of~$\g$, inducing a group morphism~$\plong^{\mathsf{reg}}: \SL(2, \mathbb{R}) \rightarrow G$ with kernel included in~$\{\pm \id\}$. The stabilizer of~$P^+$ in~$\SL(2, \mathbb{R})$ is the standard Borel subgroup~$B$ of~$\SL(2, \mathbb{R})$, so by quotient one gets a~$\plong^{\mathsf{reg}}$-equivariant embedding~$\isl^{\mathsf{reg}}: \mathbb{P}(\mathbb{R}^2) \hookrightarrow \Fl(\g, \alpha)$. Explicitly, it satisfies~$\isl^{\mathsf{reg}}([1: t]) = \exp(t f) \cdot \LP^+$ for all~$t \in \mathbb{R}$. Let
\begin{equation*}
    g_{\mathsf{reg}} := \plong^{\mathsf{reg}}\Big{(} \frac{1}{\sqrt{2}}\begin{pmatrix}
    1 & -1 \\ 1 & 1
\end{pmatrix} \Big{)} = \plong^{\mathsf{reg}}\Big{(} e^{\F}e^{-\log(\sqrt{2})\operatorname{H}}e^{-\operatorname{E}} \Big{)} = \exp(f)g_0(2)\exp(-e)~,
\end{equation*}
where~$(\operatorname{E}, \operatorname{H}, \operatorname{F})$ is the standard~$\mathfrak{sl}_2$-triple of~$\mathfrak{sl}(2, \mathbb{R})$ defined in Section~\ref{sect_sl2-triples} and~$g_0$ is defined in Section~\ref{sect_dilations_translations}.

Since~$\exp(-e)$ and~$g_0(2)$ are in~$P^+$, one has~$g_{\mathsf{reg}} \cdot \LP^+ = \exp(f) \cdot \LP^+$. The domain
\begin{equation*}
        \O_{\mathsf{nb}} := g_{\mathsf{reg}} \cdot \ncd(\mathbb{X}(\g, \alpha)) =  \exp(\Ad(g_{\mathsf{reg}}) \cdot \g^*) \exp(f)\cdot \LP^+
\end{equation*}
is a realization of~$\mathbb{X}(\g, \alpha)$.

Recall that the family~$(m_i)_{1 \leq i \leq r}$ is defined in Section~\ref{sect_embedding_concompact_dual} by~$m_i := E_i - \sigma_0(E_i)$. By Fact~\ref{fact_construction_noncompact_dual}, the subspace
\begin{equation*}
    \mathfrak{c} := \sum_{i=1}^r \mathbb{R} m_i
\end{equation*}
is a Cartan subspace of~$G^*$. An explicit computation gives that for all~$(\lambda_i)_{1 \leq i \leq r} \in \mathbb{R}^r$, one has
\begin{equation}\label{eq_ad_g_reg}
\Ad(g_{\mathsf{reg}})\cdot \sum_{i=1}^r \lambda_i m_i = - \sum_{i=1}^r \lambda_i h_{\beta_i}~,
\end{equation}
which in particular implies that~$\exp(\sum_{i=1}^r \mathbb{R}h_{\beta_i}) \subset \Aut_G(\O_{\mathsf{nb}})$. Since~$g_{\mathsf{reg}} \cdot \LP^+ = \exp(-f) \cdot \LP^+ \in \O_{\mathsf{nb}}$, we thus have
\begin{equation*}
   \exp(\sum_{i=1}^r \mathbb{R}_{>0}F_i) \cdot \LP^+ =  \exp(\sum_{i=1}^r \mathbb{R}h_{\beta_i}) \exp(f) \cdot \LP^+ \in \O_{\mathsf{nb}}~.
\end{equation*}
Hence~$\O_{\mathsf{nb}}$ is not bounded in~$\Affstdstd$, even though it is contained in~$\Affstdstd$ according to Lemma~\ref{lem_LP^+_in_O_dual} below. Here ``$\mathsf{nb}$'' stands for ``not bounded''.

Before stating the next lemma, recall that~$\LP^+_{\{\oppinv(\alpha)\}}$ is the standard parabolic subgroup defining~$\Fl(\g, \alpha)^\opp$ (see Section~\ref{sect_flag_mfds}), and~$ \LP_{\{\oppinv(\alpha)\}}^-$ is its standard opposite parabolic subgroup.

\begin{lem}\label{lem_LP^+_in_O_dual}
    One has~$\LP^+, \LP_{\{\oppinv(\alpha)\}}^- \in \partial\O_{\mathsf{nb}}$ and~$\LP^-, \LP_{\{\oppinv(\alpha)\}}^+ \in \partial\O_{\mathsf{nb}}^*$.
\end{lem}

\begin{proof} The lemma is symmetric for~$\O_{\mathsf{nb}}$ and~$\O_{\mathsf{nb}}^*$, so it suffices to prove it for~$\O_{\mathsf{nb}}$.

Let us first prove that~$\LP^+ \in \partial\O_{\mathsf{nb}}$. To this end, notice that~$\LP^+ = \lim_{n \rightarrow + \infty} g_n \exp(f) \cdot \LP^+$, where~$g_n = \exp(n \sum_{i=1}^r h_{\beta_i})$ is a sequence of elements of~$\Aut_G(\O_{\mathsf{nb}})$ by~\eqref{eq_ad_g_reg}, and~$\exp(f)\cdot \LP^+ = g_{\mathsf{reg}} \cdot \LP^+ \in \O_{\mathsf{nb}}$. Hence we have~$\LP^+ \in \overline{\O_{\mathsf{nb}}}$. Since~$(g_n)$ is a divergent sequence of~$G$, by properness of the action of~$\Aut_G(\O_{\mathsf{nb}})$ on~$\O_{\mathsf{nb}}$, we must have~$\LP^+ \in \partial\O_{\mathsf{nb}}$.

Now let~$1 \leq i \leq r$, and we claim that there exists a representative~$k_i$ of~$s_i$ such that~$k_i \cdot \LP^+ \in \partial \Omega_{\mathsf{nb}}$. Let us set
\begin{equation*}
    g_n:= \exp(\frac{1}{n} (\sum_{k=1}^i h_{\beta_k} - \sum_{k=i+1}^r h_{\beta_k}))~.
\end{equation*}
According to~\eqref{eq_ad_g_reg}, the element~$g_n$ is in~$\Aut_G(\O_{\mathsf{nb}})$. Hence
\begin{equation*}
    x:= \lim_{n \rightarrow +\infty} g_n \exp(f) \cdot \LP^+ = \lim_{n \rightarrow +\infty} \exp\big{(} e^n \sum_{k = 1}^i \sigma_0(E_k)\big{)} \cdot \LP^+
\end{equation*}
is in~$\overline{\O_{\mathsf{nb}}}$, and again, since~$\Aut_G(\O_{\mathsf{nb}})$ acts properly on~$\O_{\mathsf{nb}}$, it is actually in~$\partial \O_{\mathsf{nb}}$. It remains to prove that~$x = k_i \cdot \LP^+$ for some representative~$k_i$ of~$s_i$. To this end, we set 
\begin{equation*}
    h' :=  \sum_{k = 1}^i h_{\beta_k}~; \ \  e' :=\sum_{k = 1}^i E_{k}~; \ \  f':= -\sum_{k = 1}^i \sigma_0(E_k)~.
\end{equation*}
Then~$(e',h',f')$ is an~$\sll_2$-triple, inducing a morphism~$\rho: \SL(2, \mathbb{R}) \rightarrow G$ with kernel contained in~$\{\pm \id\}$, as well as a~$\rho$-equivariant embedding~$\zeta: \mathbb{P}(\mathbb{R}^2) \hookrightarrow \Fl(\g, \alpha)$. Then we have
\begin{equation*}
    x = \zeta(\lim_{n \rightarrow + \infty} \exp(e^n \F) \begin{bmatrix}
        1 \\ 0
    \end{bmatrix}) = \zeta(\begin{pmatrix}
        0 &1 \\ -1 & 0
    \end{pmatrix}\begin{bmatrix}
        1 \\ 0
    \end{bmatrix}) = \rho(e^{\frac{\pi}{2}(\E -\F)}) \cdot \LP^+ = \exp(\frac{\pi}{2} (e' - f')) \cdot \LP^+~.
\end{equation*}
But~$\exp(\frac{\pi}{2} (e'- f'))$ is a representative of~$s_i$ (see Section~\ref{sect_real_lie_alg}). This proves the claim.

Now let us prove that~$\LP_{\{\oppinv(\alpha)\}}^-  \in \partial\O_{\mathsf{nb}}$. Let~$k_0 \in K$ be a representative of the longest element~$w_0$ of~$W$. Then recall that one has~$k_0 \cdot \LP^+ = \LP_{\{\oppinv(\alpha)\}}^\opp $. By Fact~\ref{lem_cardinal_coset}, there exists~$1 \leq i \leq r$ such that~$\overline{s_i} = \overline{w_0}$ in~$W_{\FS \smallsetminus\{\alpha\}} \backslash W / W_{\FS \smallsetminus\{\alpha\}}$. Then~$s_i^{-1}w_0 \in W_{\FS \smallsetminus\{\alpha\}}$ is in the Weyl group of the Levi~$L$, and in particular preserves~$\LP_{\{\oppinv(\alpha)\}}^-$. This proves that for any representative~$k_i$ of~$s_i$ in~$K$, one has~$k_i\cdot\LP_{\{\oppinv(\alpha)\}}^- = k_0\cdot \LP_{\{\oppinv(\alpha)\}}^- = \LP^+$. According to the claim, we thus have~$\LP_{\{\oppinv(\alpha)\}}^- \in \partial \O_{\mathsf{nb}}$.~$\qed$
\end{proof}

Let~$x,y \in \O_{\mathsf{nb}}$. By the definition of~$\O_{\mathsf{nb}}$, there exist~$\mathsf{x}, \mathsf{y} \in \mathbb{X}(\g, \alpha)$ such that
$$(x, y) = \left(g_{\mathsf{reg}} \cdot \ncd(\mathsf{x}), \  g_{\mathsf{reg}} \cdot \ncd(\mathsf{y})\right)~,$$
where~$\ncd$ is defined in~\eqref{eq_def_ncd}.

Recall that the Kobayashi metric~$\kob_{\O_{\mathsf{nb}}}$ is~$\Aut_G(\O_{\mathsf{nb}})$-invariant, so for the computations we will always be able to translate~$\mathsf{x}, \mathsf{y}$ in~$\mathbb{X}(\g, \alpha)$. Since two elements of~$\mathbb{X}(\g, \alpha)$ are always contained in a same flat, and~$\Iso(\mathbb{X}(\g, \alpha))$ acts transitively on the set of flats of~$\mathbb{X}(\g, \alpha)$ (see e.g.\ \cite{helgason1979differential}), we may assume that~$\mathsf{x}, \mathsf{y} \in \exp(\mathfrak{c}) \cdot K_0$. Now since~$\exp(\mathfrak{c})$ acts transitively on the flat~$\exp(\mathfrak{c}) \cdot K_0$, we may assume that~$\mathsf{x} = K_0$. Then we have
\begin{equation}\label{eq_def_x_y}
\begin{split}
    x &= \exp(f) \cdot \LP^+ = \exp\big{(}\sum_{i=1}^r F_i\big{)} \cdot \LP^+, \text{ and }  \\
    y &= \exp(-\sum_{i=1}^r t_i h_{\beta_i}) \cdot x = \exp\big{(}\sum_{i=1}^r \exp(t_i) F_i\big{)} \cdot \LP^+, \ \text{ with }t_1, \dots, t_r \in \mathbb{R}~.
\end{split}
\end{equation}
We set~$x_0 := x$, and for all~$1 \leq i \leq r$, we set
$$x_i := \exp(M_i) \cdot \LP^+, \quad \text{ where } \ \ M_i = \sum_{k=1}^i \exp(t_k) F_k + \sum_{k = i+1}^r F_k~,$$ 
so that in particular~$x_r = y$. See Figure~\ref{fig:placeholder}.

For all~$0 \leq i \leq r-1$, the points~$x_i$ and~$x_{i+1}$ are on the same photon. Indeed, they both lie on~$\Phot_i := \overline{\exp(\mathbb{R}F_{i+1}) \cdot x_i} = \exp(M_i) s_i \Phot_\std$.

\subsection{The self-opposite case}\label{sect_slef_op_case} In this subsection we prove Theorem~\ref{thm_calcul_kob}.(1), assuming that~$\oppinv(\alpha) = \alpha$. 

Let us set
\begin{equation}\label{eq_def_J}
   J := \{1 \leq i \leq r \mid t_i < 0\}~.
\end{equation}

We set
\begin{equation*}
     \xi_1 := \lim_{n \rightarrow + \infty}\exp(\exp(n)\sum_{j \in J} F_j) \cdot \LP^+; \quad 
    \xi_2 := \lim_{n \rightarrow + \infty}\exp(\exp(n)\sum_{j \notin J} F_j) \cdot \LP^+~;
\end{equation*}
see Figure~\ref{fig:placeholder}. One has:
\begin{lem}
    $\xi_1, \xi_2 \in \O_{\mathsf{nb}}^*$.
\end{lem}

\begin{proof} By Lemma~\ref{lem_LP^+_in_O_dual}, it suffices to prove that there exist~$g_i \in \Aut(\O_{\mathsf{nb}}^*)$ such that~$\xi_i = g_i \cdot \LP^+_{\{\oppinv(\alpha)\}} = \LP^+$, for~$i = 1,2$.

We then set~$e' = \sum_{i\in J} E_i$,~$f' = -\sigma_0(e')$ and~$h' = \sum_{j \in J}h_{\beta_j}$. Then~$(e',h',f')$ is an~$\sll_2$-triple of~$\g$, inducing a morphism~$\tau: \SL(2, \mathbb{R}) \rightarrow G$ with kernel contained in~$\{ \pm\id\}$. The stabilizer of~$\LP^+$ in~$\SL(2, \mathbb{R})$ is the Borel subgroup~$B$ of~$\SL(2, \mathbb{R})$ which is the stabilizer of~$\begin{bmatrix}
    1 \\ 0
\end{bmatrix} \in \mathbb{P}(\mathbb{R}^2)$, inducing a~$\tau$-equivariant embedding~$\phi: \mathbb{P}(\mathbb{R}^2) \rightarrow \Fl(\g, \alpha); \ 
    g \cdot \begin{bmatrix}
        1 \\ 0
    \end{bmatrix} \mapsto \tau(g) \cdot \LP^+~.$

Recall the standard~$\sll_2$-triple~$(\E, \operatorname{H}, \F)$ of~$\sll(2, \mathbb{R})$ defined in Section~\ref{sect_sl2-triples}. One has
\begin{equation*}
    \xi_1 = \phi\big{(}\lim_{n \rightarrow + \infty} \exp(n \F\big{)} \cdot  \begin{bmatrix}
        1 \\ 0
    \end{bmatrix}) = \phi\big{(}\begin{bmatrix}
        0 \\ 1
    \end{bmatrix}\big{)} = \phi\big{(}\exp(\frac{\pi}{2}(\E - \F)\big{)} \cdot \begin{bmatrix}
        1 \\ 0
    \end{bmatrix}) = \exp\big{(}\frac{\pi}{2}(e'-f')\big{)} \cdot \LP^+~.
\end{equation*}
As mentionned in Section~\ref{sect_real_lie_alg}, the element~$\exp(\frac{\pi}{2}(e'-f'))$ is a representative of the symmetry~$s$ of~$\aaa^*$ associated to the kernel of~$\sum_{j \in J} \beta_j$. A similar argument as the one giving~$\LP^-_{\oppinv(\alpha)} \in \partial\O_{\mathsf{nb}}$ in the proof of Lemma~\ref{lem_LP^+_in_O_dual} gives that~$\xi_1 \in \partial\O_{\mathsf{nb}}$.  The proof is similar for~$\xi_2$.~$\qed$
\end{proof}

Let us consider the set~$J$ defined in~\eqref{eq_def_J}. We set
$$
\begin{array}{cc}
         a_i = & \begin{cases}
        \exp\big{(}\sum_{k=1}^{i-1} \exp(t_k) F_k + \sum_{k = i+1}^r F_k\big{)} \cdot \LP^+ \quad \text{ if } i \in J; \\
        \lim_{n \rightarrow + \infty}\exp\big{(}\sum_{k=1}^{i-1} \exp(t_k) F_k + \exp(n) F_i + \sum_{k = i+1}^r F_k\big{)} \cdot \LP^+ \quad \text{ if } i\notin J~.
    \end{cases}  \\
 &   \\
         b_i = & \begin{cases}
        \lim_{n \rightarrow + \infty}\exp\big{(}\sum_{k=1}^{i-1} \exp(t_k) F_k + \exp(n) F_i + \sum_{k = i+1}^r F_k\big{)} \cdot \LP^+ \quad \text{ if } i \in J \\
        \exp\big{(}\sum_{k=1}^{i-1} \exp(t_k) F_k + \sum_{k = i+1}^r F_k\big{)} \cdot \LP^+ \quad \text{ if } i\notin J~.
    \end{cases}.
\end{array}
$$

\begin{lem}\label{lem_extremités_des_photons} One has~$\Phot_i \cap \O = (a_i, b_i)$, and~$a_i, x_i, x_{i+1}, b_i$ are aligned in this order.
\end{lem}

\begin{proof}
    Let us assume that~$i \in J$, the proof being similar for~$i \notin J$. By definition one has~$a_i, b_i \in \Phot_i$ and that~$a_i, x_i, x_{i+1}, b_i$ are aligned in this order. Let us prove that~$a_i, b_i \in \partial \O$. 

In the notation of Example~\ref{ex_nonbounded_realizatiob}, for all~$n \in \mathbb{N}$, we set
$$g_n := \exp(n h_{\beta_i}) \in \Aut_G(\O_{\mathsf{nb}})~.$$ 
Then~$g_n \cdot x_i = \exp\big{(}\sum_{k=1}^{i-1} \exp(t_k) F_k + \exp(n) F_i + \sum_{k = i+1}^r F_k\big{)} \cdot \LP^+$ converges to~$b_i$, and~$g_n^{-1} \cdot x_i$ converges to~$a_i$, as~$n \rightarrow + \infty$. Since~$(g_n)$ diverges in~$G$ and~$\Aut_G(\O_{\mathsf{nb}})$ acts properly on~$\O_{\mathsf{nb}}$ (Fact~\ref{fact_autom_group_proper}), we deduce that~$a_i, b_i \in \partial \O$. Then, by Lemma~\ref{lem_continuité_intersection_photon}, the points~$a_i$ and~$b_i$ are the extremities of~$\Phot_i \cap \O$.~$\qed$
\end{proof}

 On the other hand, recall from Section~\ref{sect_comp_caratheodory} that we denote by~$\pr_{\Phot}$ the natural projection from a dense open subset of~$\Fl(\g, \alpha)^\opp$ ($= \Fl(\g, \alpha)$ here) on a photon~$\Phot$.

\begin{lem}
    For all~$1 \leq i \leq r$, we have~$a_i \in \hyp_{\xi_1}$ and~$b_i \in \hyp_{\xi_2}$; hence~$a_i =\pr_{\Phot_i}(\xi_1)$ and~$b_i = \pr_{\Phot_i}(\xi_2)$.
\end{lem}

\begin{proof} By Lemma~\ref{lem_extremités_des_photons}, we have that~$d_H(a_i, \xi_1) \leq r-1$ and~$d_H(b_i, \xi_2) \leq r-1$ by construction. The Lemma then follows by Fact~\ref{thm_takeuchi_hyp}. $\qed$
\end{proof}

Now let~$(G, \rho, V)$ be an~$\{ \alpha\}$-proximal triple for~$(\g, \alpha)$, and let~$C_{\O_{\mathsf{nb}}}^\rho$ be the associated Caratheodory metric on~$\O_{\mathsf{nb}}$. Let us also define~$m := \chi_\rho(h_\alpha)$. We have:
\begin{equation}\label{eq_calcul_geod_1}
     \begin{split}
        C_{\O_{\mathsf{nb}}}^{\rho}(x,y) &\geq |\log[\xi_1 : x : y : \xi_2]_{\rho}| = m |\sum_{i=0}^{r-1} 
        \log(\pr_{\Phot_i}(\xi_1) : x_i : x_{i+1} : \pr_{\Phot_{i}}(\xi_2)) | \\
        &  =  m |\sum_{i=0}^{r-1} \log(a_i : x_i : x_{i+1} : b_i) | \overset{(*)}{=} m \sum_{i=0}^{r-1}|\log(a_i : x_i : x_{i+1} : b_i))| \\
        &\overset{(**)}{=} m \sum_{i=0}^{r-1} \kob_{\O_{\mathsf{nb}}}(x_i, x_{i+1}) \overset{(***)}{=} m \operatorname{len}_{\O_{\mathsf{nb}}}(x_0, \dots , x_r) \geq m \kob_{\O_{\mathsf{nb}}}(x,y)~,
    \end{split}
\end{equation} 
where Equality~$(*)$ follows from the positive signs in  the sum, due to the choices of~$\xi_1$ and~$\xi_2$ (see Figure~\ref{fig:placeholder} for a concrete picture). Besides, Equalities~$(**)$ and~$(***)$ are due to Proposition~\ref{lem_key_lemma}.

Since we also have~$m\kob_{\O_{\mathsf{nb}}}(x,y) \geq C_{\O_{\mathsf{nb}}}^{\rho}(x,y)$ by Proposition~\ref{lem_key_lemma}, all inequalities in Equation~\eqref{eq_calcul_geod_1} are equalities. We thus have~$m\kob_{\O_{\mathsf{nb}}}(x,y) = C_{\O_{\mathsf{nb}}}^{\rho}(x,y)$. This ends the proof of Theorem~\ref{prop_calcul_geod}.(1) in the case where~$\oppinv(\alpha) = \alpha$.

\begin{figure}
    \centering
    \begin{tikzpicture}[scale = 1.2]
          \draw (-2,0) -- (0,-2) -- (2,0);
          
          \draw (-2, 0) -- (0,2) -- (2,0);
          \draw[red] (0, 0) -- (0.6,0.6) -- (1.2, 0);
          \draw[red, dashed] (-1, 2.2) -- (2.2, -1);
          \draw[red, dashed] (-2, -2) -- (2, 2);

          \fill[fill=blue] (2,0) circle (1.5pt);
          \fill[fill=blue] (-2,0) circle (1.5pt);
          \fill[fill=black] (0.6,0.6) circle (1.4pt);
          \fill[fill=black] (1.2,0) circle (1.4pt);
          \fill[fill=black] (0,0) circle (1.4pt);
          \fill[fill=red] (-0.4,1.6) circle (1.4pt);
          \fill[fill=red] (-1,-1) circle (1.4pt);
          \fill[fill=red] (1,1) circle (1.4pt);
          \fill[fill=red] (1.6,-0.4) circle (1.4pt);
          \draw[red] node at (-0.75,1.6) {$a_2$};
          \draw[red] node at (-1.35,-1) {$a_1$};
           \draw[red] node at (-2.2,-1.7) {$\Phot_1$};
          \draw[red] node at (1,1.4) {$b_1$};
          \draw[red] node at (1.65,-0.7) {$b_2$};
          \draw[red] node at (2.5,-1) {$\Phot_2$};
          \draw node at (-0.25,0) {$x_0$};
          \draw node at (0.9,0) {$x_2$};
          \draw node at (0.6,0.85) {$x_1$};
          \draw node at (-2.3,0) {$\xi_1$};
          \draw node at (2.3,0) {$\xi_2$};
        \end{tikzpicture}
    \caption{Proof of Theorem~\ref{thm_calcul_kob}.(1) in the self-opposite case, for~$r = 2$. The picture is in the pushforward in~$\O_{\mathsf{nb}}$ of a flat of~$\mathbb{X}(\g, \alpha)$. The rays issuing from~$\xi_1$ and~$\xi_2$ are contained in~$\hyp_{\xi_1}\cap\ \partial\O_{\mathsf{nb}}$ and~$\hyp_{\xi_2} \cap \ \partial\O_{\mathsf{nb}}$ respectively.}
    \label{fig:placeholder}
\end{figure}
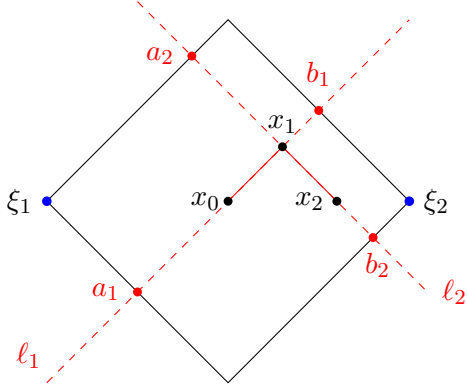

\subsection{The non-self-opposite case}\label{sect_non_self_op_case} Now let us prove Theorem~\ref{thm_calcul_kob}.(1) in the case where~$\oppinv(\alpha) \ne \alpha$. According to the list of irreducible real-type Nagano pairs, when~$\oppinv(\alpha) \ne \alpha$, the root system of~$\g^*$ is of type~$B_r$, where~$r$ is the rank of~$(\g, \alpha)$ (see Table \ref{table_nagano_full}). Hence its Weyl group acts by signed permutations on the~$E_i$,~$1 \leq i \leq r$. Hence we may assume that~$t_i \geq 1$ for all~$1 \leq i \leq r$. We then set
\begin{equation*}
    \begin{split}
        a_i &:= \exp\big{(}\sum_{k=1}^{i-1} \exp(t_k) F_k + \sum_{k = i+1}^r F_k\big{)} \cdot \LP^+~; \\
   b_i &:= \lim_{n \rightarrow + \infty}\exp\big{(}\sum_{k=1}^{i-1} \exp(t_k)F_k + \exp(n) \sigma_0(E_i) + \sum_{k = i+1}^r F_k\big{)} \cdot \LP^+~.
    \end{split}
\end{equation*}
It is clear that~$a_i \in \hyp_{\LP_{\{\oppinv(\alpha)\}}^+} \cap \Phot_i$ for all~$1 \leq i \leq r$, so~$a_i = \pr_{\Phot_i}(\LP_{\{\oppinv(\alpha)\}}^+)$. On the other hand, clearly~$b_i \notin \Affstdstd$, so~$b_i \in \hyp_{\LP^\opp}$; since~$b_i  \in \Phot_i$ we also have~$b_i = \pr_{\Phot_i}(\LP^\opp)$. Now the same computation as~\eqref{eq_calcul_geod_1} holds, replacing~$\xi_1$ by~$\LP^+_{\{\oppinv(\alpha)\}}$ and~$\xi_2$ by~$\LP^\opp$. Note that both~$\LP^\opp$ and~$\LP_{\{\oppinv(\alpha)\}}^+ $ are in ~$\O_{\mathsf{nb}}^*$, by Lemma~\ref{lem_LP^+_in_O_dual}. This ends the proof of Theorem~\ref{thm_calcul_kob}.(1) in the case where~$\oppinv(\alpha) \ne \alpha$.

\subsection{The infinitesimal form}\label{sect_infinitesimal_form} In this subsection, we prove Theorem~\ref{thm_calcul_kob}.(2).  Let~$\lVert \cdot \rVert_1$ be the~$L^1$-norm on~$\mathfrak{c}$ with respect to the basis~$(m_i)_{1 \leq i \leq r}$: 
\begin{equation*}
    \rVert\sum_{i=1}^r t_i m_i \lVert_1 = \sum_{i=1}^p |t_i|~.
\end{equation*}
Let~$x,y \in \O_{\mathsf{nb}}$ such that~$(g_{\mathsf{reg}} \ncd)^{-1}(x), (g_{\mathsf{reg}} \ncd)^{-1}(y)$ are on a same flat of~$\mathbb{X}(\g, \alpha)$. By~$G^*$-invariance of the Kobayashi metric~$\kob_{\mathsf{nb}}$, we may again assume that~$x$ and~$y$ are those defined in~\eqref{eq_def_x_y}. 

By the last inequality of~\eqref{eq_calcul_geod_1} which we proved to actually be an equality, we then have 
\begin{equation}\label{eq_infinitesimal_form_kob}
        \kob_{\O_{\mathsf{nb}}}(x,y) 
        =\sum_{i=1}^{r} |\log (0 : 1 : \exp(t_k) : \infty)| = \sum_{i=1}^{r} |t_k|~.
\end{equation}
The end of the proof then just follows form the observation that~$\mathsf{y} = \sum_{i = 1}^r t_i m_i$.

Note that this formula is intrinsic: it does not depend on the realization of~$\mathbb{X}(\g, \alpha)$ in~$\Fl(\g,\alpha)$.

\begin{rmk}\label{rmk_calcul_kob} (1) The Kobayashi metric computed in this section on realizations of the noncompact dual is not induced by a Riemannian metric on~$\mathbb{X}(\g, \alpha)$, except in the case where~$\operatorname{rk}(\g, \alpha) = 1$, i.e.\ where~$\Fl(\g, \alpha) = \mathbb{P}(\mathbb{R}^{n+1})$ for some~$n \geq 0$ (see item (iv, $\min(p,q) = 1$) of Table~\ref{table_nagano_full}). In this latter case, we recover~$\mathbb{X}(\g, \alpha) \simeq \mathbb{H}^{n}$, and classically, the metric~$\kob_\O = \mathsf{H}_\O$ of a realization~$\O$ of~$\mathbb{X}(\g, \alpha)$ is a multiple of the pushforward of the hyperbolic metric on~$\mathbb{H}^{n}$, and~$\O$ is an ellipsoïd. Except in the real projective case, the Kobayashi metric on a realization of~$\mathbb{X}(\g, \alpha)$ does not coincide with the Riemannian metric of~$\mathbb{X}(\g, \alpha)$. This property distinguishes the higher-rank case from the rank-one case.

(2) In the notation of Theorem~\ref{thm_calcul_kob}, the inequality in Proposition~\ref{lem_key_lemma} is an equality. We expect that, in higher-rank Nagano spaces, the equality-case in Proposition~\ref{lem_key_lemma} characterizes realizations of the noncompact dual among~$\MRr$-proper dually convex domains.
\end{rmk}

\section{Higher-rank rigidity}\label{sect_rigidity}

In this section, we address Conjecture~\ref{question_lim_Zim} mentioned in the introduction. We provide partial results on the higher-rank behavior of proper almost-homogeneous domains in irreducible real Nagano spaces of higher rank. Recall that the only rank-one irreducible real-type Nagano pairs are~$(\sll(n, \mathbb{R}), \alpha_1)$ and~$(\sll(n, \mathbb{R}), \alpha_{n-1})$, for~$n \in \mathbb{N}_{>0}$. In other words, the only rank-one irreducible real-type Nagano spaces are real projective spaces and their duals. Hence, in this context, comparing proper divisible domains in rank one versus in higher rank amounts to comparing them in real projective space with those in other irreducible real-type Nagano spaces. 

In this section, we show that well-known properties of some proper divisible domains in real projective space, exhibiting a rank-one behavior, are \emph{never} satisfied by those in other irreducible real-type Nagano spaces. These are therefore genuinely higher-rank phenomena, as is evident in the proofs: the crucial steps rely on Fact~\ref{thm_takeuchi_hyp}, which relates the rank of a Nagano space to the possible degrees of transversality between two points.

Note that in causal flag manifolds and Einstein Universe (items~(iii), (viii), (x), (xi) and (xii) of Table~\ref{table_nagano_full}), it has been shown that proper almost-homogeneous domains are higher-rank symmetric spaces \cite{galiay2024rigidity, chalumeau2024rigidity}; the results of this section that concern \emph{proper} domains therefore follow from these works in that setting. They nevertheless remain of interest for the other irreducible real-type higher-rank Nagano spaces.

\subsection{Higher rank and complexity}\label{sect_rank_and_complexity}
Recall that the complexity of a boundary point of an~$\MRr$-proper domain is defined in Definition~\ref{def_complexity_boundary_point}. In real projective space, the arithmetic distance~$d_H$ defined in Section~\ref{sect_arithm_distance} is the discrete distance: $d_H(x,y) = 1$ if and only if~$x \ne y$. Hence the complexity of a boundary point~$p \in \partial \O$ is always~$\leq 1$.

The purpose of this section is to highlight how higher rank increases the complexity of the boundary of a proper almost-homogeneous domain, thereby producing a major difference with the rank-one case, and to prove Theorem~\ref{lem_geom_prop_visuel_general}. The proof is given in Paragraph~\ref{lem_geom_prop_visuel_general}. Sections~\ref{sect_proximal_limit_set} and~\ref{sect_asymptotic_complexity} contain preliminary results.

\subsubsection{The proximal limit set}\label{sect_proximal_limit_set}
In this section, using the results from Section~\ref{sect_geometric_prop_R_extr}, we prove Proposition~\ref{prop_proximal_limit_set} below, which generalizes a well-known fact in convex projective geometry.

Let~$G$ be a noncompact semisimple Lie group and~$\Theta$ be a subset of the simple restricted roots of~$G$. An element~$g \in G$ is~$\Theta$-proximal if it has two transverse fixed points~$x \in \Fl(\g, \Theta)$ and~$y \in \Fl(\g, \Theta)^\opp$ such that~$g^n \cdot z \rightarrow x$ for all~$z \in \Fl(\g, \Theta) \smallsetminus \hyp_y$. The points~$x$ and~$y$ are then uniquely defined by~$g$. A subgroup~$H \leq G$ is~$\Theta$-proximal if it contains at least one proximal element. In this case, we define
\begin{equation*}\Lambda_\Theta^{\operatorname{prox}}(H) = \overline{\{x \in \Fl(\g,\alpha) \mid \exists g \in H, \ g\text{ proximal with attracting fixed point }x\}}~.
\end{equation*}
By definition, we have~$\Lambda_\Theta^{\operatorname{prox}}(H) \subset \Lambda_\Theta(H)$. We then have:

\begin{prop}\label{prop_proximal_limit_set} Let~$(\g, \alpha)$ be an irreducible real-type Nagano pair, and let~$\O \subset N$ be a proper domain. If~$H \leq \Aut_G(\O)$ acts almost-homogeneously on~$\O$, i.e.\ if
\begin{equation*}
    \{a \in \partial \O \mid \exists (h_n) \in  H^\mathbb{N}, \ \exists x \in \O, \ h_n\cdot x \rightarrow a\} = \partial \O~,
\end{equation*}
then~$H$ is~$\{\alpha\}$-proximal and~$\Lambda_{\{\alpha\}}^{\operatorname{prox}}(H) = \overline{\Rr(\O)} = \Lambda_{\{\alpha\}}(H)$.
\end{prop}
In the notation of Proposition~\ref{prop_proximal_limit_set}, if~$H$ acts cocompactly on~$\O$, then it acts almost-homogeneously. If~$\O$ is almost-homogeneous, then the conditions of Proposition~\ref{prop_proximal_limit_set} are satisfied for~$H = \Aut_G(\O)$.

\begin{proof}[Proof of Proposition~\ref{prop_proximal_limit_set}] Let us first prove that~$\Rr(\O) \subset \Lambda_\Theta^{\operatorname{prox}}(H)$. To this end we use the strategy of the proof of \cite[Prop.\ 2.3.15]{blayac2021aspects}. By Proposition~\ref{lem_extremal_attracting}, there exist~$b \in \Fl(\g,\alpha)^\opp$ and~$(g_k) \in H^\mathbb{N}$ such that~$g_k \cdot x \rightarrow a$ uniformly on compact subsets of~$\Fl(\g,\alpha) \smallsetminus \hyp_b$.

    Let~$(G, \rho, V)$ be a Plücker triple for~$(\g, \alpha)$. Then the sequence~$(\rho(g_k))$ is~$\{\alpha_1\}$-contracting in~$\SL(V)$. Thus there exists a rank-one endomorphism~$\pi_1 \in \operatorname{End}(V)$ such that~$\frac{\rho(g_k)}{||\rho(g_k)||} \rightarrow \pi_1$. Then~$\operatorname{Im}(\pi_1) = \iota_{\rho}(a)$ and~$\ker(\pi_1) = \iota_{\rho}^\opp (b)$. By Lemma~\ref{lem_existence_extreme_points}, we do not have~$\iota_{\rho}(\Rr(\O)) \subset \iota_{\rho}^\opp (b)$. Thus there exists~$a_2 \in \Rr(\O)$ such that~$\iota_{\rho}(a_2) \notin \ker(\pi_1)$. Denoting~$a_1 := a$, by induction there exist~$\dim(V)$ points~$a_1, a_2, \dots, a_{\dim(V)} \in \Rr(\O)$ and rank-one endomorphisms~$\pi_1, \dots, \pi_{\dim(V)} \in \operatorname{End(V)} \cap  \ \overline{\rho(H)}^{\operatorname{End}(V)}$ such that
\begin{equation}\label{eq_iota_etc}
        \iota_{\rho}(a_i) = \operatorname{Im}(\pi_i) 
\end{equation}
for all~$1 \leq i \leq \dim(V)$ and
\begin{equation}\label{eq_iota_etc_2}
    \iota_{\rho}(a_{i+1}) \notin \ker(\pi_i)
\end{equation}
for all~$1 \leq i \leq \dim(V)-1$. Thus~$\iota_{\rho}(a_i) = \pi_i(a_{i+1}) \in \overline{\rho(H) \cdot \iota_{\rho}(a_{i+1})}$ and 
    \begin{equation}\label{eq_a_i1}
        a_{i} \in \overline{H \cdot a_{i+1}} \quad \forall 1 \leq i \leq \dim(V)-1~.
    \end{equation} 
By induction, we thus have~$a \in \overline{H \cdot a_{i+1}}$.

Assume for a contradiction that~$a \notin \Lambda_{\{\alpha\}}^{\operatorname{prox}}(H)$. Since~$\Lambda_{\{\alpha\}}^{\operatorname{prox}}(H)$ is closed and~$H$-invariant, by~\eqref{eq_a_i1} for all~$1 \leq i \leq \dim(V)$ we have~$a_i \notin \Lambda_{\{\alpha\}}^{\operatorname{prox}}(H)$. By Equations~\eqref{eq_iota_etc} and~\eqref{eq_iota_etc_2}, for all~$1 \leq i \leq j \leq \dim(V)$, we have~$\iota_\rho (a_i) = \operatorname{Im}(\pi_i \circ\pi_{i+1} \circ \cdots \circ \pi_j)$. Thus~$\pi_i \circ\pi_{i+1} \circ \cdots \circ \pi_j$ is not proximal. Thus~$\iota_\rho (a_i) \in \ker(\pi_i \circ \pi_{i+1} \circ \cdots \circ \pi_j) = \ker(\pi_j)$. 

We have just proved that for all~$1 \leq i \leq j \leq \dim(V)$, we have~$\iota_\rho (a_i) \in \ker(\pi_j)$.
In paricular, for all~$1 \leq i \leq \dim(V)-1$, we have~$\iota_{\rho}(a_{i+1}) \in (\ker(\pi_{i+1}) \cap \cdots \cap \ker(\pi_{\dim(V)})) \smallsetminus \ker(\pi_{i})$. Thus the sequence~$(\ker(\pi_{i}) \cap \cdots \cap \ker(\pi_{\dim(V)}))_{1 \leq i \leq \dim(V)}$ is an increasing sequence of nonempty vector subspaces of~$V$. Thus~$\dim(\ker(\pi_{\dim(V)})) \geq \dim(V)$. This is in contradiction with the fact that~$\pi_{\dim(V)} \ne 0$. Thus~$a \in \Lambda_{\{\alpha\}}^{\operatorname{prox}}(H)$. By closedness of~$\Lambda_{\{\alpha\}}^{\operatorname{prox}}(H)$, we have~$\overline{\Rr(\O)} \subset \Lambda_{\{\alpha\}}^{\operatorname{prox}}(H)$.

Now let us prove that~$\Lambda_{\{\alpha\}}(H) \subset \overline{\Rr(\O)}$. Let~$x \in \Lambda_{\{\alpha\}}(H)$. Let~$(g_k) \in H^\mathbb{N}$ and~$y \in \Fl(\g,\alpha)^\opp$ such that~$(g_k)$ is~$\{\alpha\}$-contracting with respect to~$(x,y)$. By Lemma~\ref{lem_existence_extreme_points}, there exists~$z \in \Rr(\O)$ such that~$\iota_{\rho}(z) \notin \iota_{\rho}^\opp(y)$. Thus~$z \notin \hyp_y$, and~$g_k \cdot z \rightarrow x$. By~$\Aut_G(\O)$-invariance of~$\Rr(\O)$, we then have~$x \in \overline{\Rr(\O)}$. We have proven that~$\Lambda_{\{\alpha\}}(H) \subset \overline{\Rr(\O)}$. We thus have
\begin{equation*}
    \Lambda_{\{\alpha\}}(H) \subset \overline{\Rr(\O)} \subset \Lambda_{\{\alpha\}}^{\operatorname{prox}}(H) \subset \Lambda_{\{\alpha\}}(H)~.
\end{equation*}
Hence these inclusions are equalities.~$\qed$ 
\end{proof}

\subsubsection{Complexity and asymptotic properties of the automorphism group}\label{sect_asymptotic_complexity}
The goal of this section is to prove Theorem~\ref{lem_geom_prop_visuel_general}. To this end, we define the notion of \emph{complexity} of a boundary point of~$\O$.

\begin{definition}\label{def_complexity_boundary_point}
   Let~$a \in \partial \O$. We call the \emph{complexity  of~$a$ with respect to~$\O$} (or just the \emph{complexity of~$a$}, when there is no ambiguity) the integer
   \begin{equation*}
       \cpl(p) := \max\{d_H(a,b) \mid b \in \overline{\Fl_\O^\MRr(a)}\}\ \in \ \{0, \dots, \operatorname{rk}(\g,\alpha)\}~.
   \end{equation*}
\end{definition}

We first need the following lemma:

\begin{lem}\label{lem_existence_extremal_dans_face}
    Let~$(\g, \alpha)$ be an irreducible real-type Nagano pair. 
    Assume that~$\O$ is a proper almost-homogeneous domain of~$\Fl(\g,\alpha)$. Let~$p \in \partial \O$. Then there exists~$q_\infty \in \overline{\Extr(\O)} \cap \overline{\Fl_\O^\MRr(p)}$.
\end{lem}

\begin{proof} By almost-homogeneity,  and a point~$x \in \O$ such that~$g_k \cdot x \rightarrow p$. By Lemma~\ref{lem_asymptotic_behav_extremal}, one then has~$a \in \Fl_\O^\MRr(p)$ for any limit point~$a$ of any sequence~$(g_k \cdot y)$ with~$y \in \O$.

Let~$(G, V, \rho)$ be any Plücker triple for~$(\g, \alpha)$. Up to taking a subsequence, we may assume that~$\rho(g_k)$ converges to some~$T \in \mathbb{P}(\operatorname{End}(V))$. By definition of~$T$, one has~$T(\iota_\rho(\O)) \subset \iota_\rho(\Fl_\O^\MRr(p))$. By Lemma~\ref{lem_existence_extreme_points}, there exists~$q \in \Extr(\O) \smallsetminus \iota_\rho^{-1}(\ker(T))$. Consider a 
sequence~$(q_k) \in \O^{\mathbb{N}}$ converging to~$q$.

Since~$T$ is continuous on~$\mathbb{P}(V) \smallsetminus \Ker(T)$, we have~$T(\iota_\rho(q_k)) \rightarrow T(\iota_\rho(q))$. Since for all~$k \in \mathbb{N}$ one has~$T(q_k) \in \iota_\rho(\Fl_\O^\MRr(p))$, we thus have
\begin{equation*}
    T(\iota_\rho(q)) \in \overline{\iota_\rho(\Fl_\O^\MRr(p))} = \iota_\rho(\overline{\Fl_\O^\MRr(p)})~.
\end{equation*}
In particular, we can write~$T(\iota_\rho(q)) = \iota_\rho(q_\infty)$, with~$q_\infty \in \overline{\Fl_\O^\MRr(p)}$. On the other hand, we have~$\iota_\rho(q_\infty) = T(\iota_\rho(q)) = \lim_{k \rightarrow + \infty} \iota_\rho(g_k \cdot q)$. Hence~$q_\infty = \lim_{k \rightarrow + \infty} g_k \cdot q$. But for all~$k \in \mathbb{N}$, we have~$g_k \cdot q \in \Extr(\O)$, so we get~$q_\infty \in \overline{\Extr(\O)}$. We have just found an element~$q_\infty \in \overline{\Extr(\O)} \cap \overline{\Fl_\O^\MRr(p)}$.~$\qed$
\end{proof}

\subsubsection{End of the proof of Theorem~\ref{lem_geom_prop_visuel_general}}\label{sect_end_proof_theorem_1.6} We now take the notation of Theorem~\ref{lem_geom_prop_visuel_general} and end its proof. Let~$q_\infty \in \overline{\Extr(\O)} \cap \overline{\Fl_\O^\MRr(p)}$ be given by Lemma~\ref{lem_existence_extremal_dans_face}, and let~$(q_k) \in \overline{\Extr(\O)}^\mathbb{N}$ be such that~$q_k \rightarrow q_\infty$. By Proposition~\ref{prop_geom_prop_extremal_points}, for all~$k \in \mathbb{N}$, there exists~$\xi_k \in \O^*$ such that
\begin{equation}\label{eq_id_in_pos}
    \id \in \pos^{(\alpha, \oppinv(\alpha))}(q_k, \xi_k)~.
\end{equation}
Let~$\xi$ be any limit point of~$\xi_k$ in~$\O^*$. then taking the limit in~\eqref{eq_id_in_pos} (by \cite{kapovich2017dynamics}), we get~$\id \in \pos^{(\alpha, \oppinv(\alpha))}(q_\infty, \xi)$. 

Let us take an arbitrary point~$x_0 \in \O$. Then we have~$x_0 \notin \hyp_\xi$, which means that~$w_0 \in \pos^{(\alpha, \oppinv(\alpha))}(x_0, \xi)$. Then by Lemma~\ref{lem_positions_weyl}.(2), we have~$w_0 \in \pos^{(\alpha, \alpha)}(x_0, q_\infty)$. Then, by Fact~\ref{thm_takeuchi_hyp}, we have 
\begin{equation*}
    d_H(x_0, q_\infty) = \operatorname{rk}(\Fl(\g,\alpha))~.
\end{equation*}
Now the triangle inequality gives~$\operatorname{rk}(\Fl(\g,\alpha)) \leq d_H(x_0, p) + d_H(p, q_\infty) \leq  d_H(x_0, p) + \cpl(p)$. Taking the infimum over~$x_0 \in \O$ ends the proof of Theorem~\ref{lem_geom_prop_visuel_general}.~$\qed$

\subsection{Non-hyperbolicity of the Kobayashi metric}\label{sect_non_hyp_kob}

The analysis conducted in Section~\ref{sect_rank_and_complexity} on the boundary of proper almost-homogeneous domains allows us to prove Theorem~\ref{thm_group_non_hyperbolic_1} in this section. We first state and prove an immediate corollary of Theorem~\ref{thm_group_non_hyperbolic_1}:

\begin{cor}\label{cor_group_non_hyperbolic_I} Let~$(\g, \alpha)$ be an irreducible real-type Nagano pair and of rank~$\geq 2$, and~$G \in \mathcal{G}_{\{\alpha\}}(\g)$. 
   Let~$\O \subset \Fl(\g,\alpha)$ be a domain. Assume that one of the following assumptions is satisfied:
   \begin{enumerate}
       \item $\O$ is proper;
       \item $\O$ is strongly~$\MRr$-proper and dually convex.
   \end{enumerate}
   If there exists~$\Gamma \leq \Aut_G(\O)$ dividing~$\O$, then~$\Gamma$ is not Gromov hyperbolic. 
\end{cor}

\begin{proof} If~$\O$ is proper, then since it is divisible, it is almost-homogeneous and thus dually convex (recall Fact~\ref{prop_zimmer_dual_convex}). By Corollary~\ref{cor_kobayashi_geodesic}, the metric space~$(\O, \kob_\O)$ is thus proper and geodesic. If~$\O$ is stronlgy~$\MRr$-proper and dually convex, then the same conclusion holds by Corollary~\ref{cor_kobayashi_geodesic}. Now since~$\Gamma$ acts cocompactly and properly discontinuously (Fact~\ref{fact_autom_group_proper} and Corollary~\ref{cor_action_proper}) on the metric space~$(\O, \kob_\O)$, by Svark--Milnor's Lemma and Theorem~\ref{thm_group_non_hyperbolic_1}, the group~$\Gamma$ is not Gromov hyperbolic.~$\qed$
\end{proof}

The aim of this section is to prove Theorem~\ref{thm_group_non_hyperbolic_1}. To this end, we adapt the proof of~\cite{zimmer2015convexity}, in which point (1) is proven for item (iv) of Table~\ref{table_nagano_full}. With the formalism on real-type Nagano spaces and~$\MRr$-proper dually convex domains that we have introduced, the proof given in \cite{zimmer2015convexity} generalizes to any higher-rank real-type Nagano pair.

Let us fix some notation. We fix a higher-rank real-type Nagano pair~$(\g, \alpha)$, a Plücker triple~$(G, \rho, V)$ of~$(\g, \alpha)$, and~$\vv_0 \in V^{\omega_{\alpha}} \smallsetminus \{ 0 \}$. We moreover take Notation~\ref{notation_nagano}. Let~$\O \subset \Fl(\g,\alpha)$ be a proper almost-homogeneous domain. We may assume that~$\O$ is contained in the standard affine chart~$\Affstdstd$ and that~$\LP^+ \in \O$. Then~$\Phot_\std \cap \O \ne \varnothing$, so by~$\MRr$-properness of~$\O$ there exists some~$t \in \mathbb{R} \smallsetminus \{ 0 \}$ such that~$\exp(tv^-) \cdot \LP^+\in\partial \O$. Up to dilating in~$\Affstdstd$ (see Section~\ref{sect_graded_Lie_algebra}), we may assume that~$a :=\exp(v^-) \cdot \LP^+\in\partial \O$. 

We claim that~$a$ is not~$\MRr$-extremal. If it was, then by Proposition~\ref{prop_geom_prop_extremal_points}, we would have~$d_H(x,a ) \geq \operatorname{rk}(\g, \alpha)$ for all~$x \in \O$. But on the other hand, we have~$d_H(P, a) = 1$ by construction. This is impossible since~$\operatorname{rk}(\g, \alpha) \geq 2$. This proves the claim.

Since~$a$ is not~$\MRr$-extremal, its~$\MRr$-face is nontrivial. Thus by Lemma~\ref{fact_stab_u_std}.(2), there exists~$w := \Ad(g) \cdot v^-$, with~$g \in L$ and~$\delta > 0$, such that~$\exp(v^- + s w) \in \partial \O$ for all~$s \in (-\delta, \delta)$. Up to considering~$\Ad(g_0(\frac{1}{\delta})) \cdot w$ instead of~$w$, we may assume that~$\delta =1$ (recall that~$g_0$ is defined in Section~\ref{sect_dilations_translations}). 

By Lemma~\ref{lem_continuité_intersection_photon}, there exists~$\varepsilon >0$ such that~$\exp(t v^- +s w) \in \O$ for all~$t \in (1- \varepsilon, 1)$ and~$s \in (- \varepsilon, \varepsilon)$. Besides, since~$\O$ is dually convex, there exists~$\xi \in \O^*$ such that~$a \in \hyp_{\xi}$; by Proposition~\ref{cor_face_incluse_hyp}, we have~$\Fl_{\O}^\MRr(a) \subset \hyp_{\xi}$. Hence we also have 
\begin{equation}\label{eq_face_included}
    \exp(v^- + s w) \cdot \LP^+\in \hyp_{\xi} 
\end{equation}
for all~$s \in (-1, 1)$. Let~$f \in V^* \smallsetminus \{0\}$ be the unique lift of~$\iota^{\opp}_\rho(\xi)$ such that~$f(\vv_0) =1$. By~\eqref{eq_face_included} and Fact~\ref{prop_ggkw}, we have 
\begin{equation}\label{eq_face_included_2}
    f(\rho(\exp(v^- + sw)) \cdot \vv_0) = 0  
\end{equation}
for all~$s \in (-1, 1)$. Since Equation~\eqref{eq_face_included_2} is polynomial, we have~$f(\rho(\exp(v^- + sw)\cdot \vv_0 ) \cdot x_0) = 0$ for all~$s \in \mathbb{R}$. Note that for all~$s,t \in \mathbb{R}$, we have 
\begin{equation}\label{equ_expression_f}
\begin{split}
    f(\rho(\exp(tv^- + sw)&\cdot \vv_0)\cdot x_0) = \\
    &1 + s f(\rho_*(w)\cdot \vv_0) +t f(\rho_*(v^-) \cdot \vv_0) + ts f(\rho_*(v^-)\rho_*(w) \cdot \vv_0)~.
\end{split}
\end{equation}
Taking~$t=1$, by~\eqref{eq_face_included_2}, this gives
$$0 = 1 + s f(\rho_*(w) \cdot \vv_0) + f(\rho_*(v^-)\cdot \vv_0) + s f(\rho_*(v^-)\rho_*(w) \cdot \vv_0)$$ 
for all~$s \in \mathbb{R}$, which implies:
\begin{equation*}
    \begin{cases}
        1+ f(\rho_*(v^-)\cdot\vv_0) &= 0 \\
        f(\rho_*(w) \cdot \vv_0) + f(\rho_*(v^-)\rho_*(w) \cdot \vv_0) &=0~.
    \end{cases}
\end{equation*}
This system allows us to simplify Equation~\eqref{equ_expression_f} and get
$$f(\rho(\exp(tv^- + sw)\cdot \vv_0) \cdot x_0) = (1-t)(1+s\lambda), \quad \text{with } \lambda := f(\rho_*(w) \cdot \vv_0)~.$$ 

Now for~$s,t \in \mathbb{R}$, we write~$x_{t,s} := \exp(t v^- + s w) \cdot \LP^+$. By Proposition~\ref{lem_key_lemma},  if~$x_{t_1, s_1}, x_{t_2, s_2} \in \O$, then:

\begin{equation}\label{eq_minoration_K}
    \begin{split}
        \kob_{\O}(x_{t_1,s_1}, x_{t_2,s_2}) &\geq C_{\O}^\rho (x_{t_1,s_1}, x_{t_2,s_2})   \geq \Big{|} \log \Big{|} [\xi : x_{t_1,s_1} : x_{t_2,s_2} : \LP^\opp]_\rho \Big{|} \Big{|} \\
    & = \Big{|} \log \Big{|} \frac{f(\rho(\exp(t_1v^- + s_1w))\cdot \vv_0)}{f\left(\rho\left(\exp(t_2v^- + s_2w)\right)\cdot \vv_0 \right)} \Big{|} \Big{|}  = \Big{|} \log \Big{|} \frac{(1-t_1)(1+s_1 \lambda)}{(1-t_2)(1 +s_2 \lambda)} \Big{|} \Big{|}.
    \end{split}
\end{equation}
By Lemma~\ref{lem_continuité_intersection_photon}, for all~$0 \leq t < 1$ there exist~$m_t < 0 <  M_t$ such that 
\begin{equation*}
    \O \cap \{\exp(t v^- + s w) \cdot \LP^+\mid s \in \mathbb{R}\} = \{\exp(t v^- + s w) \cdot \LP^+\mid m_t < s < M_t\}.
\end{equation*}
We claim that~$(m_t)$ converges to~$m := \liminf_{u \rightarrow 1} m_u$, as~$t \rightarrow 1$. Then 
\begin{equation*}
    \{\exp(v^- + sw) \cdot \LP^+\mid m \leq s \leq 0\} \subset \partial \O~.
\end{equation*}
By Lemma~\ref{lem_continuité_intersection_photon}, for any~$\varepsilon >0$ there exists~$\delta >0$ such that~$\exp(tv^- + sw) \cdot \LP^+ \in \O$ for all~$s \in [m+ \varepsilon, 0]$ and~$t \in [-\delta, 1)$. Hence~$m + \varepsilon \geq \limsup_{u \rightarrow 1} m_u$. This is true for all~$\varepsilon > 0$, so~$m = \limsup_{s \rightarrow 1} m_s$. This proves the claim. 

Similarly, there exists~$M \in \mathbb{R}_{\geq  0}$ such that~$M_t \rightarrow M$ as~$t \rightarrow 1$. Up to replacing~$w$ with~$-w$, we may assume that~$M \leq \frac{-1}{\lambda}$. Hence
\begin{equation}\label{eq_lambda_ineq}
    \lambda < 0 \quad \text{ and } \quad 1 + m \lambda > 1~.
\end{equation}
With this inequality, we can now end the proof of Theorem~\ref{thm_group_non_hyperbolic_1}. This follows from the following Lemma~\ref{lem_zim_convex_Omega}, which is proven in \cite{zimmer2015convexity} for item (iv) of Table~\ref{table_nagano_full} if~$\O$ is proper, and the proof generalizes verbatim to any higher-rank real-type Nagano space and the~$\MRr$-proper case, using~\eqref{eq_lambda_ineq}. We reproduce the proof for completeness of the paper.

A geodesic rectangle~$R$ in a geodesic metric space~$(X,d)$ is~\emph{$A$-thin} if any side of~$R$ is contained in an~$A$-neighborhood of the three others. In a Gromov hyperbolic space, there exists~$A>0$ such that any geodesic rectangle is~$A$-thin.

\begin{lem}\label{lem_zim_convex_Omega}\cite{zimmer2015convexity}
    For any~$A>0$, there exists a geodesic rectangle in~$(\O, \kob_{\O})$ which is not~$A$-thin.
\end{lem}

\begin{proof}
    Let us first fix some notation. We fix~$R > 2A+ \log(1 + m \lambda)+2$, and~$s_0 \in (m,0)$ so that~$(m : s_0 : 0: M) > R + 1$. We moreover fix~$t_0 \in [e^{-1},1)$ such that~$s_0 \in (m_t,0)$ and~$(m_t : s_0 : 0 :M_t)> R$ for every~$t > t_0$.

Now, for~$r \in (t_0,1)$, consider the segments of photon
\begin{equation*}
        \gamma_1 = [x_{t_0, s_0}, x_{t_0, 0}], \ \gamma_2(r) = [x_{t_0, 0}, x_{r, 0}], \ \gamma_3(r) = [x_{t_1, s_0}, x_{r, 0}], \ \gamma_4(r) = [x_{t_0, s_0}, x_{r, s_0}].
\end{equation*}
By Proposition~\ref{lem_key_lemma}, each of the~$\gamma_i$ is a geodesic for~$\kob_{\O}$. By Theorem~\ref{thm_kobayashi_metric_nagano}, the concatenation of these segments thus forms a geodesic rectangle~$\gamma(r)$ in~$(\O, \kob_\O)$. We will prove that for~$r$ sufficiently large, the side~$\gamma_4(r)$ is not contained in an~$A$-neighborhood of~$\gamma_1 \cup \gamma_2(r) \cup \gamma_3(r)$.

By properness of~$\kob_\O$, there exists~$u_0 \in (t_0,1)$ so that
\begin{equation*}
    \kob_{\O}(x_{u_0, s_0}, \gamma_1) \geq A~.
\end{equation*}
It remains to prove that there exists~$r \in (u_0, 1)$ such that
\begin{equation*}
    \kob_\O(x_{u_0, s_0}, \gamma_2(r) \cup \gamma_3(r)) \geq A~.
\end{equation*}
Since~$\kob_\O$ is a proper metric on~$\O$, there exists~$r \in (u_0, 1)$ such that~$ \kob_\O(x_{u_0, s_0}, \gamma_3(r)) \geq A$. Hence it is sufficient to prove that 
\begin{equation}\label{eq_ineq_gamma_2}
    \kob_\O(x_{u_0, s_0}, \gamma_2(r)) \geq A~.
\end{equation}
Let~$t \in (t_0, r)$. We will prove that~$\kob_\O(x_{u_0, s_0}, x_{t,0}) \geq A$. 

  \begin{enumerate}
      \item Let us first assume that~$\log \frac{1-u_0}{1-t} \leq R-A-|\log(t_0)|$. Then  since~$|\log(t/u_0)| \leq |\log(t_0)|$, one has~$(0 : u_0 :t : 1) = \log \frac{(1 - u_0)t}{ (1 - t)u_0 } \leq R - A$. Thus one has 
\begin{align*}
    \kob_{\O}(x_{u_0, s_0}, x_{t,0}) &\geq  \kob_{\O}(x_{u_0, s_0}, x_{u_0, 0}) - \kob_{\O}(x_{u_0, 0}, x_{t, 0}) \\
    &\geq ( m_{u_0} : s_0 : 0 : M_{u_0}) - (0 : u_0 : t : 1) \geq R - (R - A) = A~.
\end{align*}

\item Now let us assume that~$\log \frac{1-u_0}{1-t} \geq R-A-|\log(t_0)|$. By~\eqref{eq_minoration_K} and definition of~$R$ one has 
  \begin{align*}
      \kob_{\O}(x_{u_0, s_0}, x_{t, 0}) &\geq \log \frac{1-u_0}{1 -t} + \log(1+s_0\lambda) \\
      &\geq \log \frac{1-u_0}{1-t} - |\log(t_0)|-|\log(1 +s_0\lambda)| \\
      &\geq R-A-2|\log(t_0)|- \log(1+ m \lambda) \geq A~,
  \end{align*}
  the last inequality holding by~\eqref{eq_lambda_ineq} and because~$|\log(t_0)| \leq 1$. 

  \end{enumerate}

We have proven that~$\kob_\O(x_{u_0, s_0}, x_{t, 0}) \geq A$ for all~$t \in (t_0, r)$, which proves~\eqref{eq_ineq_gamma_2}. 
Hence the rectangle~$\gamma(r)$ is not~$A$-thin.~$\qed$    
\end{proof}

\begin{rmk}\label{rmk_comp_rang_1}
   Theorem~\ref{thm_group_non_hyperbolic_1} is specific to higher rank. In rank one, which corresponds to the case where~$\Fl(\g, \alpha)$ is real projective space (item (iv, $\min(p,q) = 1$) of Table~\ref{table_nagano_full}), the following result, due to Benoist, is well known \cite{benoist2001convexes}: \textit{Let~$\Gamma \leq \PGL(n, \mathbb{R})$ be a discrete subgroup acting cocompactly on a proper domain~$\O \subset \mathbb{P}(\mathbb{R}^n)$. Then~$\Gamma$ is Gromov hyperbolic if and only if~$\O$ is strictly convex.}

Now let~$(\g,\alpha)$ be a higher-rank irreducible real-type Nagano pair,~$G \in \mathcal{G}_{\{\alpha\}}(\g)$ and~$\O \subset \Fl(\g,\alpha)$ be a proper domain, divisible (and even just almost-homogeneous) by a discrete subgroup~$\Gamma \leq G$. In the proof of Theorem~\ref{thm_group_non_hyperbolic_1}, we saw that the domain~$\O$ cannot be ``strictly convex'', in the sense that not all points of~$\partial \O$ can be~$\MRr$-extremal. Besides, Corollary~\ref{cor_group_non_hyperbolic_I} states that~$\Gamma$ is never hyperbolic. Hence \cite{benoist2001convexes} still holds.

We see here a higher-rank phenomenon: proper divisible domains of~$\Fl(\g,\alpha)$ obey the same principle as those of~$\mathbb{P}(\mathbb{R}^n)$, but excluding the hyperbolic behavior (which is a rank-one behavior).
\end{rmk}

\subsection{Proper almost-homogeneous domains are not of rank one}\label{sect_rank_one_almost} Let us finish this paper with a last higher-rank behavior of proper almost-homogeneous domains in higher-rank Nagano spaces.

In fact, the Gromov hyperbolicity of the Hilbert metric of a proper divisible domain of projective space might not be the correct interpretation of the rank-one behavior in convex projective geometry: more generally, a proper divisible domain~$\O \subset \mathbb{P}(\mathbb{R}^n)$ is said to be \emph{rank-one} if there exists~$a \in \partial \O$ such that~$[a,b] \cap \O \neq \varnothing$ for all~$b \in \partial \O$. Such proper divisible domains were introduced by Islam \cite{islam2019rank} and have properties analogous to real hyperbolic space \cite{benoist2003convexes, benoist2006convexes, crampon2009entropies, cooper2015convex, zimmer2020higher}, although they are in general not strictly convex and hence not Gromov hyperbolic (equipped with their Hilbert metric). A.\ Zimmer \cite{zimmer2020higher} proved that all irreducible non-symmetric divisible convex sets are rank-one. Thus, all non-symmetric divisible convex sets exhibit this ``rank-one behavior''. 

In \cite{blayac2021boundary}, Blayac proved the following: if~$\O  \subset \mathbb{P}(\mathbb{R}^n)$ ($n \geq 2$) is an irreducible convex domain \emph{of rank one}, divisible by some discrete subgroup~$\Gamma \leq \PGL(n, \mathbb{R})$, then one has~$\Lambda_{\{\alpha_1\}}^{\operatorname{prox}}(\Gamma) = \partial \O$. The combination of this result with \cite{zimmer2020higher} gives:  \emph{If~$\O  \subset \mathbb{P}(\mathbb{R}^n)$ is a nonsymmetric irreducible convex domain, divisible by some discrete subgroup~$\Gamma \leq \PGL(n, \mathbb{R})$, then one has $\Lambda_{\{\alpha_1\}}^{\operatorname{prox}}(\Gamma) = \partial \O$}.

The situation is opposite in higher rank:

\begin{prop}\label{prop_limit_proximal_pas_tout}
    Let~$(\g, \alpha)$ be an irreducible real-type Nagano pair and of higher rank. Then, for any proper almost-homogeneous domain~$\O \subset \Fl(\g,\alpha)$, we have~$\partial \O \smallsetminus  \Lambda_{\{\alpha\}}(\Aut_G(\O)) \ne \varnothing$.
\end{prop}

In other words, in higher rank, proper almost-homogeneous domains are ``never of rank one''. 

\begin{proof}[Proof of Proposition~\ref{prop_limit_proximal_pas_tout}] By Proposition~\ref{prop_proximal_limit_set}, we have~$\Lambda_{\{\alpha\}}^{\operatorname{prox}}(\Aut_G(\O)) = \Lambda_{\{\alpha\}}(\Aut_G(\O))$. Let~$x \in \O$, and~$\Phot$ be a photon through~$x$. Then there exists~$a \in \Phot \cap \partial \O$.  
If~$a \in \Lambda_{\{\alpha\}}(\Aut_G(\O))$, then by Proposition~\ref{prop_geom_prop_extremal_points}, there exists~$b \in \O^*$ such that~$\pos^{(\{\alpha\}, \{\oppinv(\alpha)\})}(a,b) = \overline{\id}$. Then, by Lemma~\ref{lem_positions_weyl}.(2), in the notation of Fact~\ref{thm_takeuchi_hyp}, we have~$\hyp_a^* \subset \hyp_b$. Hence by Fact~\ref{thm_takeuchi_hyp} and since~$\operatorname{rk}(\g, \alpha) \geq 2$, we must have~$x \in \hyp_b$, which contradicts the fact that~$b \in \O^*$.~$\qed$
\end{proof}

\newpage

\begin{table}[H]
\centering
\setlength\extrarowheight{3pt}
\rotatebox{90}{
\begin{tabular}{|c|c|c|c|c|c|c|c|}
\Xhline{2.0pt}
\# & $\g$ & $\alpha$ & $\Fl(\g,\alpha)$ & $\dim(\g_\alpha)$ &
$\Iso(\Fl(\g,\alpha), g_{\g,\alpha})$ & $\Iso(\mathbb{X}(\g,\alpha))$ & $\operatorname{rk}(\g,\alpha)$ \\
\Xhline{2.0pt}

\cellcolor{gray!30}(i) & $\soo(n,n)$ & $\alpha_i,\, i \in \{n-1, n\}$ & $\SO(n)$ & $1$ & $\SO(n)\times \SO(n)$ & $\SO(n,\Cf)$ & $\lfloor n/2 \rfloor$ \\ \hline
(ii) & $\spp(n,n)$ & $\alpha_n = \varepsilon_n$ & $\Sp(n)$ & $3$ & $\Sp(n)\times \Sp(n)$ & $\Sp(2n,\Cf)$ & $n$ \\ \hline
\cellcolor{gray!30}(iii) & $\suu(n,n)$ & $\alpha_n = 2\varepsilon_n$ & $\operatorname{U}(n)$ & $1$ & $\mathrm{S}(\operatorname{U}(n)\times \operatorname{U}(n))$ & $\SL(n,\Cf)\times \Rf$ & $n$ \\ \hline

\cellcolor{gray!30}(iv) & $\sll(p+q,\Rf)$ & $\alpha_p=\varepsilon_p-\varepsilon_{p+1}$ & $\Grass$ & $1$ & $\SO(p+q)$ & $\SO(p,q)$ & $\min(p,q)$ \\ \hline
(v) & $\sll(p+q,\Cf)$ & $\alpha_p=\varepsilon_p-\varepsilon_{p+1}$ & $\GrassC$ & $2$ & $\SU(p+q)$ & $\SU(p,q)$ & $\min(p,q)$ \\ \hline
(vi) & $\sll(p+q,\Hf)$ & $\alpha_p=\varepsilon_p-\varepsilon_{p+1}$ & $\GrassH$ & $4$ & $\Sp(p+q)$ & $\Sp(p,q)$ & $\min(p,q)$ \\ \hline
\cellcolor{gray!30}(vii) & $\eeee_{6(6)}$ & $\alpha_i,\, i\in\{1,5\}$ & $\Sp(4)/(\Sp(2)\times\Sp(2))$ & $1$ & $\Sp(4)$ & $\Sp(2,2)$ & $2$ \\ \hline

\cellcolor{gray!30}(viii) & $\soo(p+1,q+1)$ & $\alpha_1=\varepsilon_1-\varepsilon_2$ & $\mathbb{S}^p\times\mathbb{S}^q$ & $1$ & $\SO(p+1)\times\SO(q+1)$ & $\SO(p,1)\times\SO(1,q)$ & $2$ \\ \hline
(ix) & $\soo(n,1)$ & $\alpha_1=\varepsilon_1-\varepsilon_2$ & $\mathbb{S}^{n-1}$ & $n-1$ & $\SO(n+1)$ & $\SO(n-1,1)$ & $1$ \\ \hline

\cellcolor{gray!30}(x) & $\soo^*(4n)$ & $\alpha_n=2\varepsilon_n$ & $(\SU(2n)/\Sp(n))\times\mathbb{S}^1$ & $1$ & $\operatorname{U}(2n)$ & $\SL(n,\Hf)\times\Rf$ & $n$ \\ \hline
\cellcolor{gray!30}(xi) & $\spp(2n,\Rf)$ & $\alpha_n=2\varepsilon_n$ & $(\SU(n)/\SO(n))\times\mathbb{S}^1$ & $1$ & $\operatorname{U}(n)$ & $\SL(n,\Rf)\times\Rf$ & $n$ \\ \hline
\cellcolor{gray!30}(xii) & $\eeee_{7(-25)}$ & $\alpha_3$ & $(E_7/F_4)\times\mathbb{S}^1$ & $1$ & $E_6\times \mathbb{S}^1$ & $E_{6(-26)}\times\Rf$ & $3$ \\ \hline

(xiii) & $\soo(n+2,\Cf)$ & $\alpha_1=\varepsilon_1-\varepsilon_2$ & $\SO(n+2)/\mathrm{S}(\mathrm{O}(2)\times\mathrm{O}(n))$ & $2$ & $\SO(n+2)$ & $\SO(n,2)$ & $\min(n,2)$ \\ \hline
(xiv) & $\soo(2n,\Cf)$ & $\alpha_i\in\{n-1,n\}$ & $\SO(2n)/\operatorname{U}(n)$ & $2$ & $\SO(2n)$ & $\SO^*(2n)$ & $\lfloor n/2 \rfloor$ \\ \hline
(xv) & $\spp(2n,\Cf)$ & $\alpha_n=2\varepsilon_n$ & $\Sp(n)/\operatorname{U}(n)$ & $2$ & $\Sp(n)$ & $\Sp(2n,\Rf)$ & $n$ \\ \hline
(xvi) & $\eeee_{6,\Cf}$ & $\alpha_i,\, i\in\{1,5\}$ & $E_6/(\SO(2)\times\SO(10))$ & $2$ & $E_6$ & $E_{6(-14)}$ & $2$ \\ \hline
(xvii) & $\eeee_{7,\Cf}$ & $\alpha_7$ & $E_7/(\SO(2)\times E_6)$ & $2$ & $E_7$ & $E_{7(-25)}$ & $3$ \\ \hline

(xviii) & $\eeee_{6(-26)}$ & $\alpha_i,\, i\in\{1,2\}$ & $F_4/B_4$ & $8$ & $F_4$ & $F_{4(-20)}$ & $1$ \\ \hline
\cellcolor{gray!30}(xix) & $\eeee_{7(7)}$ & $\alpha_7$ & $\SU(8)/\Sp(4)$ & $1$ & $\SU(8)$ & $\SL(4,\Hf)$ & $3$ \\ 

\Xhline{2.0pt}
\end{tabular}}
\caption{\small List of irreducible Nagano spaces. The greyed out items correspond to real-type Nagano spaces. Groups are given up to local isomorphism. The parameters $n,p,q$ are positive integers. The notations~$\g, \alpha, \Fl(\g, \alpha)$ ,and~$\g_{\alpha}$ are introduced in Section~\ref{sect_lie_theory}.
Notation~$g_{\g, \alpha}$ is introduced in Section~\ref{sect_construction_symetri_structure},~$\operatorname{rk}(\g, \alpha)$ in Definition~\ref{def_rank_nagano_space} and~$\mathbb{X}(\g, \alpha)$ in Section~\ref{sect_embedding_concompact_dual}. The information in this table come from \cite{nagano1965transformation, makarevivc1973open, takeuchi1988basic, onishchik2012lie}. \normalsize}
\label{table_nagano_full}
\end{table}

\footnotesize
\printbibliography
\vspace{1cm}
\small\noindent \textsc{Institut de Recherche Mathématique Avancée,
7 rue René Descartes, 67000 Strasbourg, France.} \emph{email address:} \texttt{galiay@unistra.fr}

\end{document}